\documentclass[hidelinks,12pt,reqno]{amsart}

\usepackage{bbm}

\usepackage{enumerate, amsmath, amsthm, amsfonts, amssymb, 
mathrsfs, graphicx, paralist} 
\usepackage[usenames, dvipsnames]{color}
\usepackage[margin=1in]{geometry} 
\usepackage{hyperref}
\usepackage{comment}
\usepackage{stmaryrd}
\usepackage{dsfont}

\usepackage{mathabx} 
\usepackage{tikz-cd}

\usepackage[all]{xy}

\usepackage{imakeidx}
\makeindex[columns=3]

\numberwithin{equation}{section}
\newtheorem{Theorem}[equation]{Theorem}

\newtheorem{Proposition}[equation]{Proposition}
\newtheorem{Lemma}[equation]{Lemma}
\newtheorem{Corollary}[equation]{Corollary}

\newtheorem{Definition/Proposition}[equation]{Definition/Proposition}
\newtheorem{Example}[equation]{Example}
\newtheorem{Warning}[equation]{Warning}

\theoremstyle{definition}
\newtheorem{Remark}[equation]{Remark}

\newtheorem{eg}[equation]{Example}

\newtheorem{Definition}[equation]{Definition}

\newtheorem{Assumption}[equation]{Assumption}

\newcommand{\R}{\mathbb{R}}
\newcommand{\A}{\mathbb{A}}

\newcommand{\N}{{\mathbb{\Z}_{\geq 0}}}
\newcommand{\Z}{\mathbb{Z}}
\newcommand{\Q}{\mathbb{Q}}
\newcommand{\C}{\mathbb{C}}
\newcommand{\Ne}{{\mathbb{Z}_{\geq 1}}}
\newcommand{\I}{\mathcal{I}}
\newcommand{\T}{\mathcal{T}}

\newcommand{\conv}{\mathrm{conv}}

\newcommand{\HC}{\mathcal{H}}

\newcommand{\htt}{\mathrm{ht}}

\newcommand{\qp}{\pi}

\newcommand{\efface}[1]{}

\newcommand{\wb}{\mathbf{w}}
\newcommand{\Tb}{\mathbf{T}}
\renewcommand{\Tb}{T}

\newcommand{\Zb}{\mathbf{Z}}
\renewcommand{\Zb}{Z}

\newcommand{\RY}{\cR[\![Y^+]\!]}

\newcommand{\pr}{\mathrm{proj}}

\newcommand{\coeff}{\mathrm{coeff}}

\newcommand{\WAF}{\mathrm{WAF}}

\newcommand{\sA}{\mathscr{A}}
\newcommand{\bB}{\mathbf{B}}

\newcommand{\cF}{\mathcal{F}}

\newcommand{\bG}{\mathbf{G}}

\newcommand{\cH}{\mathcal{H}}

\newcommand{\cK}{\mathcal{K}}

\newcommand{\cM}{\mathcal{M}}

\newcommand{\cO}{\mathcal{O}}

\newcommand{\sP}{\mathscr{P}}

\newcommand{\cR}{\mathcal{R}}

\newcommand{\cS}{\mathcal{S}}

\newcommand{\sS}{\mathscr{S}}
\newcommand{\bT}{\mathbf{T}}
\newcommand{\cT}{\mathcal{T}}

\newcommand{\sT}{\mathscr{T}}

\newcommand{\cZ}{\mathcal{Z}}

\newcommand{\sZ}{\mathscr{Z}}

\newcommand{\fm}{\mathfrak{m}}

\newcommand{\bq}{\mathbf{q}}

\newcommand{\bu}{\mathbf{u}}

\newcommand{\bv}{\mathbf{v}}

\newcommand{\bw}{\mathbf{w}}

\renewcommand{\AA}{\mathbb{A}}

\newcommand{\CC}{\mathbb{C}}

\newcommand{\TT}{\mathbb{T}}

\newcommand{\Red}{\mathrm{Red}}

\renewcommand{\phi}{\varphi}
\renewcommand{\emptyset}{\varnothing}

\renewcommand{\tilde}[1]{\widetilde{#1}}

\makeatletter
\def\Ddots{\mathinner{\mkern1mu\raise\p@
\vbox{\kern7\p@\hbox{.}}\mkern2mu
\raise4\p@\hbox{.}\mkern2mu\raise7\p@\hbox{.}\mkern1mu}}
\makeatother

\DeclareMathOperator{\supp}{Supp}

\newcommand{\suchthat}{\mid}

\newcommand{\textif}{\text{ if }}
\newcommand{\textand}{\text{ }\mathrm{and}\text{ }}
\newcommand{\textor}{\text{ or }}

\newcommand{\One}{\mathbbm{1}}
\newcommand{\kk}{\Bbbk}

\newcommand{\HBL}{\widehat{\HC}}

\newcommand{\HCW}{\HC_{W_0}}

\usepackage[left,modulo]{lineno}

\hypersetup{
pdfauthor={Hébert, Muthiah},
pdftitle={Completed Iwahori-Hecke algebra},
}

\begin{document}

\title{Completed Iwahori-Hecke algebra for Kac-Moody groups over local fields}
\maketitle

\author{Auguste \textsc{Hébert}  \and Dinakar \textsc{Muthiah}}

\begin{abstract}
Let $G$ be a split Kac-Moody group over a non-Archimedean local field, and let $\cH$ be the Iwahori-Hecke algebra of $G$. In this paper, we construct a completed Iwahori-Hecke algebra $\widehat{\cH}$ and prove that it contains a large center isomorphic to Looijenga's invariant ring. By the Kac-Moody Satake isomorphism, Looijenga's invariant ring is isomorphic to the spherical Hecke algebra. Our completion is constructed by considering Iwahori biinvariant functions on $G$ satisfying a support condition that we call \emph{Weyl almost finite support}. We contrast our construction with another completion $\widetilde{\cH}$, defined early by Abdellatif and H\'ebert, which is defined algebraically via the Bernstein-Lusztig presentation and not in terms of functions on $G$.

\end{abstract}

\tableofcontents

\section{Introduction}

Let $G$ be a split Kac-Moody group over a non-Archimedean local field. Associated to $G$ is an Iwahori-Hecke algebra $\cH$ (the \emph{Kac-Moody affine Hecke algebra}), which was constructed by Braverman, Kazhdan, and Patnaik in the untwisted affine case \cite{braverman2014affine} and Bardy-Panse, Gaussent, and Rousseau in the general Kac-Moody case \cite{bardy2016iwahori}.

In this paper, we construct a completion of $\cH$ allowing certain infinite sums. Without such a completion, the center of the algebra is much smaller than one would naturally expect. In particular, one expects the center to be isomorphic to Looijenga's invariant ring \cite{looijenga1980invariant}, which involves infinite sums. Additionally, the Satake isomorphism, due to Braverman and Kazhdan in the untwisted affine case \cite{braverman2011spherical} and due to Gaussent and Rousseau in the general Kac-Moody case \cite{gaussent2014spherical}, says that Looijenga's invariant ring is isomorphic the spherical Hecke algebra.

In both Looijenga's invariant ring and the spherical Hecke algebra, one must consider infinite sums governed by a support condition known as \emph{almost-finite support}. Our goal in the present article is to construct a completion $\widehat{\cH}$ consisting of Iwahori biinvariant functions on $G$ satisfying a slightly more subtle support condition that we call \emph{Weyl almost finite support}. We further compute the center of $\widehat{\cH}$ and show that it is isomorphic to the Looijenga's invariant ring.

We mention that another completion of $\cH$ was constructed by Abdellatif and H\'ebert. This was defined only in terms of Bernstein-Lusztig generators, and not in terms of functions on the $G$ subject to a support condition. We compare the constructions in \S \ref{ss_comparison_HBL_tilde_cH}.

\subsection{The reductive case}
We first recall the setting of reductive groups over non-Archimedean fields. Let $\mathbf{G}$ be a split reductive group and $\cK$ be a non-Archimedean local field. Let $\bT$ be a split maximal torus of $\bG$ and $\bB$ be a Borel subgroup containing $\bT$. Let $\cK$ be a non-Archimedean local field, $\cO$ be its  ring of integers, $\fm$ be the maximal ideal of $\cO$ and $\Bbbk$ be the residual field $\cO/\fm$. Let $G=\mathbf{G}(\cK)$. The Iwahori subgroup $I$ of $G$ is the preimage of $\mathbf{B}(\mathbf{\kk})$ by the natural projection $\mathbf{G}(\cO)\twoheadrightarrow \mathbf{G}(\kk)$. The group $G$ is a unimodular locally compact group, so there is a unique Haar measure normalized by the condition that $I$ has measure one. We can define the Iwahori-Hecke algebra $\cH$ of $G$ to be the set of $I$-biinvariant functions from $G$ to $\C$, with compact support. The algebra structure on $\cH$ is defined by convolution with respect to the Haar measure.

Alternatively, the convolution can also be characterized without directly mentioning the Haar measure. For $g,g'\in G$, one can check that 
\begin{equation}\label{e_product_H_L}
 \One_{IgI}\cdot \One_{Ig'I}=\sum_{Ig''I\in I\backslash G/I} |IgI\cap g''I g'^{-1}I/I| \One_{Ig''I}.
\end{equation}
where for a subset $S \subseteq G$, $\One_{S}$ denotes the indicator function of the set $S$.
This characterization of the multiplication reformulation is essential when we define Iwahori-Hecke algebras in the Kac-Moody setting, since there is no Haar measure in that case. Let $Y$ be the cocharacter lattice of $\bT$, and let $W_0$ be the Weyl group of $\bG$ with respect to $\bT$. Let $\pi\in \fm$ be a uniformizer. For every $\lambda \in Y$, we define $\pi^\lambda$ to be the image of $\pi$ under the map $\cK^\times \rightarrow T$ induced by the cocharacter $\lambda$.

The Iwahori-Matsumoto decomposition is the following double coset decomposition:
\begin{equation}
 G=\bigsqcup_{\lambda\in Y, w\in W_0} I\pi^\lambda w I
\end{equation}
For each $\lambda\in Y$ and $w\in W_0$, write $\TT_{\pi^\lambda w}=\One_{I\pi^\lambda w I}$. These elements obviously form a basis of $\cH$.

\subsubsection{Bernstein-Lusztig elements}

For every $\lambda \in Y$. we have the Bernstein-Lusztig elements $\Theta^\lambda \in \cH$. These elements satisfy the multiplcation rule $\Theta^\lambda \cdot \Theta^\mu=\Theta^{\lambda+\mu}$, for $\lambda,\mu\in Y$. Additionally, for $\lambda$ dominant, $\Theta^\lambda = q^{-\htt(\lambda)} \TT_{\pi^\lambda}$ where $\htt(\lambda)$ denotes the height of $\lambda$. Moreover, the formula for dominant coweights and the multiplication rule completely characterize the Bernstein-Lusztig generators. Additionally, the algebra $\cH$ is generated by $\cH_{W_0}$, the Hecke algebra of the Coxeter group $W_0$, and the Bernstein-Lusztig generators. Let $\CC[Y]$ denote the group algebra of $Y$. The Bernstein-Lusztig elements give rise to an embedding $\CC[Y] \hookrightarrow \cH$. We will identify $\CC[Y]$ with its image in $\cH$. As a vector space, $\cH$ is a tensor product of $\cH_{W_0}$ and $\CC[Y]$.

One can prove that the center $\cZ(\cH)$ of $\cH$ is equal $\C[Y]^{W_0}$, the $W_0$-invariants in $\C[Y]$. Using the Satake isomorphism, one further deduces that $\cZ(\cH)$ is isomorphic to the spherical Hecke algebra of $G$.

\subsection{The Kac-Moody case}

Let $\bG$ be a split Kac-Moody group (functor) as defined by by Tits in \cite{tits1987uniqueness}, and let $G=\bG(\cK)$. We keep the same notation as above in the case of reductive groups. 

A difficulty in the Kac-Moody framework is the lack of an Iwahori-Matsumoto decomposition on the entirety of $G$.  Let $Y^{++}$ denote the set of dominant cocharacters.  Then define
\begin{equation}
  G^+=\bigsqcup_{\lambda\in Y^+,w\in W_0} I\qp^\lambda w I
  \end{equation}
  where $Y^+=W_0\cdot Y^{++}$ is the integral Tits cone. The set $G^+$ is a semi-group, and unless $\bG$ is a finite-type Kac-Moody group, it is a proper subset of $G$. The Iwahori-Hecke algebra $\cH$ of $G$ is defined as the set of functions on $G^+$, which are $I$-biinvariant and supported on finitely many $I$-double cosets. In particular, $\cH$ has a basis consisting of elements of the form $\TT_{\pi^\lambda w}$ where $\lambda \in Y^+$ and $w \in W_0$.

As in the reductive case, $\cH$ admits Bernstein-Lusztig elements denoted $\Theta^\lambda$ for $\lambda \in Y^+$. In particular, we obtain an embedding $\CC[Y^+] \hookrightarrow \cH$ where $\CC[Y^+]$ denotes the semi-group algebra of the semi-group $Y^+$.
  
By the same argument as in the reductive case, one compute that the center of $\cH$ is $\C[Y^+]^{W_0}$ (\cite[Lemma 4.31]{abdellatif2019completed}). However as $W_0$ is infinite in general, $\C[Y^+]^{W_0}$ is often very small. For example, in the affine case $\C[Y^+]^{W_0}$ is spanned by just the elements $\Theta^{n \delta}$, where $\delta$ is the minimal positive imaginary coroot. In particular, unlike in the reductive case, the center of $\cH$ is much smaller than spherical Hecke algebra.

\subsubsection{Support conditions and the Looijenga invariant ring}
\label{sec:Support-conditions-and-the-Looijenga-invariant-ring}

Let $E\subset Y$.  We say that $E$ is  \textbf{almost finite}  if there exists a finite set $J\subset Y$ such that $E\subset J-Q^\vee_+$, where $Q_+$ is the positive coroot cone. A typical example of an almost finite set is an orbit $W_0\cdot \lambda$, where $\lambda\in Y^+$. We define the Looijenga algebra $\C[\![Y^+]\!]$ as the set of formal series $\sum_{\lambda\in Y^+}x_\lambda \Theta^\lambda$ such that $(x_\lambda)\in \C^{Y^+}$ has almost finite support. Because series in $\C[\![Y^+]\!]$ are supported on the integral Tits cone, the group $W_0$ acts on $\C[\![Y^+]\!]$, and we can therefore define the {\bf Looijenga's invariant ring} to be $\C[\![Y^+]\!]^{W_0}$.

\subsubsection{Main theorem}
We call a subset $E$ of $Y^+$ {\bf Weyl almost finite} if $W_0\cdot E$ is almost finite.
We define $\widehat{\cH}$ to be the set of series $\sum_{\lambda\in Y^+, w \in W_0}  a_{\pi^\lambda w} \TT_{\pi^\lambda w}$ subject to the following support conditions. First, the set
\begin{align}
  \label{eq:9}
  \left\{ \lambda \in Y^+ \suchthat a_{\pi^\lambda w} \neq 0 \text{ for some } w \in W_0 \right\}
\end{align}
should be Weyl almost finite. Additionally, for each $\lambda \in Y^+$, the set
\begin{align}
  \label{eq:10}
  \left\{ w \in W_0 \suchthat a_{\pi^\lambda w } \neq 0 \right\}
\end{align} 
should be finite. In particular, $\widehat{\cH}$ is equal to the set of Iwahori-biinvariant functions on $G^+$ subject to the above support condition.

Our main result is the following (see Theorem~\ref{thmConvolution_summable_family}, Definition~\ref{d_convolution_product}, Theorem~\ref{t_center} and Theorem~\ref{t_T_presentation_HBL}).
\begin{Theorem}
  \label{thm-intro-main-theorem}
\begin{enumerate}
\item The convolution product $\cH$ naturally extends to a product on $\widehat{\cH}$ making $\widehat{\cH}$ into an associative algebra.
\item Every element of $\widehat{\cH}$ can be written uniquely as a series $\sum_{\lambda\in Y^+}  \Theta^\lambda h_\lambda$, with $(h_\lambda)\in (\HCW)^{Y^+}$ having Weyl almost finite support. Conversely every element of this form is an element of $\widehat{\cH}$.

\item The center $\cZ(\widehat{\cH})$ is isomorphic to Looijenga's invariant ring $\C[\![Y^+]\!]^{W_0}$. 
\end{enumerate}
\end{Theorem}

In the main text, we actually take Theorem \ref{thm-intro-main-theorem}(2) as our initial definition of $\widehat{\cH}$ and then prove the above definition as a non-trivial result (Theorem \ref{t_T_presentation_HBL})).

\subsection{Motivation from the study of the spherical Hecke algebra of $G$}

Our project started from the following question.

First we recall the case of reductive groups $G$. Let $K=\bG(\cO)$.  Then the spherical Hecke algebra $\cH_K$ consists of the functions from $G$ to $\C$, which are $K$-biinvariant and whose support is compact. As $I$ is contained in $K$, we have an embedding $\cH_K$ in $\cH$, which is non-unital, but respects the multiplication up to a scalar factor given by index of $I$ in $K$.

In the Kac-Moody case, then as we have mentioned above, the definition of $\cH_K$ requires one to allow infinite sums. A natural question arises: Can we embed $\cH_K$ in  $\cH$, or in some ``completion'' of $\cH$, as a set of functions?

At the level of vector spaces, this is is fairly straightforward. For $\tilde{\lambda}\in Y^{++}$, we have $K\pi^{\tilde{\lambda}} K=\bigsqcup_{\mu\in W_0\cdot\tilde{\lambda},w\in W_0} I \pi^\mu w I$ and thus we could try to embed $\cH_K$ in a completion of $\cH$ via  
\begin{equation}
\label{eq:spherical-expand-iwahori}
\One_{K\pi^{\tilde{\lambda}} K} \mapsto \sum_{\mu\in W_0\cdot\tilde{\lambda},w\in W_0}\TT_{\pi^\mu w}. 
\end{equation}
However, when we try to expand a product of two expressions of the form \ref{eq:spherical-expand-iwahori}, we obtain infinite coefficients. In particular, we do not know how to embed $\cH_K$ in a completed Iwahori-Hecke algebra in a way compatible with their relations as functions on $G^+$.

Nonetheless, in a forthcoming paper we introduce an $\cH-\cH_K$ bi-module $\widehat{\cM}$ consisting of the series $\sum_{\lambda\in Y^+} a_\lambda \One_{I\pi^\lambda K}$ with Weyl almost finite support. We extend the left action of $\cH$ on $\widehat{\cM}$ to a right action of $\widehat{\cH}$ on $\widehat{\cM}$. Using this bimodule we study the interplay between  $\widehat{\cH}$ and $\cH_K$. In particular, we directly obtain an isomorphism between $\cZ(\widehat{\cH)}$ and $\cH_K$ that does not make use of the Satake isomorphism. In fact, we can use our methods to obtain new proofs of the Macdonald formula for the Satake isomorphism obtained in \cite{braverman2016iwahori} and \cite{bardy2019macdonald}. 

\subsection{Conventions}
\label{subsec:intro-conventions}

There are two conventions regarding the torus action on $Y$. The first one, compatible with this introduction and \cite{braverman2016iwahori}, has $\pi^\lambda$ acting by  $\pi^\lambda.\mu=\mu+\lambda$ for $\lambda,\mu\in Y$. In the second one, used in \cite{bardy2016iwahori} and in the body of the text, $\pi^\lambda$ acts on $Y$ by $\pi^\lambda.\mu=\mu-\lambda$, for $\lambda,\mu\in Y$. This leads to two different Iwahori-Hecke algebras. The first, $\cH^+$ considered in \cite{braverman2016iwahori} and denoted $\cH$ in this introduction, corresponds to Iwahori-biiinvariant functions on $G^+$. The second, $\cH^-$ considered in \cite{bardy2016iwahori} and denoted $\cH$ in the body of the text, corresponds to Iwahori biinvariant functions on $G^{-}=\bigsqcup_{w\in W_0,\lambda\in Y^+}I\pi^{-\lambda}w I$\index{g@$G^{-}$}. For $\lambda\in Y^+,w\in  W_0$, we will follow the notation of \cite{bardy2016iwahori} and write $T_{\lambda\cdot w}$ for the indicator function $\One_{I\pi^{-\lambda}wI}$.

We can pass from $\cH^+$ to $\cH^-$ via an anti-isomorphism $\iota$ defined in Appendix \ref{sec:appendix}. Therefore it is easy to pass between the two conventions, and the main difference is a swap between left and right. In the reductive case, the two Hecke algebras agree so the distinction can be safely ignored. However, the distinction is very important in the Kac-Moody case. As explained in Subsection~\ref{ss_definitionr_right}, we cannot swap left and right factors in Theorem \ref{thm-intro-main-theorem} as such a swap leads to infinities in the structure coefficients.

\subsection{Acknowledgements}

D.M. was supported by JSPS KAKENHI grant number JP19K14495 and by the Engineering and Physical Sciences Research Council grant UKRI167 ``Geometry of Double Loop Groups''.
We thank Anna Pusk\'as and Manish Patnaik for stimulating conversations.

\section{Recollections and preliminaries on Iwahori-Hecke algebras of Kac-Moody groups}\label{s_IH_algebras}

\subsection{Basic notations}

If $E$ is a set, we denote by $\sP(E)$ its power set. If $E$ and $F$ are sets, $E^F$ denotes the set of families $(e_f)_{f\in F}$, where $e_f\in E$ for all $f\in F$. If $E$ is  a vector space, we denote by $E^{(F)}$ the set of families of $E^F$ whose support is finite.

\begin{Definition}
  If $\leq$ is a preorder on a set $E$ and $a,b\in E$, we set $]-\infty,a]_{\leq}=\{c\in E\mid c\leq a\}$,  $]-\infty,a[_{\leq}=(]-\infty,a]_{\leq})\setminus\{a\}$, $[a,b]_{\leq}=\{c\in E\mid a\leq c\leq b\}$, $[a,+\infty[_{\leq}=\{c\in E\mid c\geq a\}$, etc. \index{ @$[\cdot,\cdot]_{\leq}, ]\cdot,\cdot]_{\leq}, ]-\infty,\cdot]_{\leq},[\cdot,+\infty[_{\leq},\ldots$}
\end{Definition}

If $a,b\in \Z$, we write $\llbracket a,b\rrbracket$ the set $[a,b]\cap \Z$, $\llbracket a,b\llbracket$ the set $[a,b[\cap \Z$ etc.

\subsection{Kac-Moody root data and related objects}\label{ss_Standard_apartment}

\subsubsection{Kac-Moody root data}
A \textbf{ Kac-Moody matrix} (or { generalized Cartan matrix}) is a square matrix $A=(a_{i,j})_{i,j\in I_A}$ indexed by a finite set $I_A$, with integral coefficients, and such that :
\begin{enumerate}
\item[\tt $(i)$] $\forall \ i\in I_A,\ a_{i,i}=2$;

\item[\tt $(ii)$] $\forall \ (i,j)\in I_A^2, (i \neq j) \Rightarrow (a_{i,j}\leq 0)$;

\item[\tt $(iii)$] $\forall \ (i,j)\in I_A^2,\ (a_{i,j}=0) \Leftrightarrow (a_{j,i}=0$).
\end{enumerate}
A \textbf{Kac-Moody root data} is a $5$-tuple $\mathcal{S}=(A,X,Y,(\alpha_i)_{i\in I_A},(\alpha_i^\vee)_{i\in I_A})$\index{s@$\mathcal{S}$}\index{y@$Y$} made of a Kac-Moody matrix $A$ indexed by the finite set $I_A$, of two dual free $\Z$-modules $X$ and $Y$ of finite rank, and of a free family $(\alpha_i)_{i\in I_A}$ (respectively $(\alpha_i^\vee)_{i\in I_A}$) of elements in $X$ (resp. $Y$) called \textbf{simple roots} (resp. \textbf{simple coroots}) that satisfy $a_{i,j}=\alpha_j(\alpha_i^\vee)$ for all $i,j$ in $I_A$. Elements of $X$ (respectively of $Y$) are called \textbf{characters} (resp. \textbf{cocharacters}).

Fix such a Kac-Moody root data $\mathcal{S}=(A,X,Y,(\alpha_i)_{i\in I_A},(\alpha_i^\vee)_{i\in I_A})$ and set $\A=Y\otimes \R$\index{a@$\A$}. Each element of $X$ induces a linear form on $\A$, hence $X$ can be seen as a subset of the dual $\A^*$. In particular, the $\alpha_{i}$ (with $i \in I_A$) will be seen as linear forms on $\A$. This allows us to define, for any $i \in I_A$, an involution $r_{i}$ of $\A$ by setting $r_{i}(v) = v-\alpha_i(v)\alpha_i^\vee$ for any $v \in \A$.   One defines the \textbf{Weyl group of $\mathcal{S}$} as the subgroup $W_0$\index{w@$W_0$} of $\mathrm{GL}(\A)$ generated by $\{r_i\mid i\in I_A\}$. The pair $(W_0, \{r_i\mid i\in I_A\})$ is a Coxeter system, hence we can consider the length $\ell(w)$ with respect to $\{r_i\mid i\in I_A\}$ of any element $w$ of $W_0$.

\subsubsection{Root systems} The following formula defines an action of the Weyl group $W_0$ on $\A^{*}$:  
\[\displaystyle \forall \ x \in \A , w \in W_0 , \alpha \in \A^{*} , \ (w.\alpha)(x)= \alpha(w^{-1}.x).\]
Let $\Phi= \{w.\alpha_i\mid w\in W_0,i\in  I_A\}$\index{p@$\Phi,\Phi^\vee$} (resp. $\Phi^\vee=\{w.\alpha_i^\vee\mid w\in W_0, i\in I_A\}$) be the set of \textbf{real roots} (resp. \textbf{real coroots}). Then  $\Phi$ (resp. $\Phi^\vee$) is a subset of the \textbf{root lattice} $Q = \displaystyle \bigoplus_{i\in I_A}\Z\alpha_i$ (resp. \textbf{coroot lattice} $Q^\vee=\bigoplus_{i\in I_A}\Z\alpha_i^\vee$). By \cite[1.2.2 (2)]{kumar2002kac}, one has $\R \alpha^\vee\cap \Phi^\vee=\{\pm \alpha^\vee\}$ and $\R \alpha\cap \Phi=\{\pm \alpha\}$ for all $\alpha^\vee\in \Phi^\vee$ and $\alpha\in \Phi$.

Let $\htt:Q^\vee\rightarrow \Z$\index{h@$\htt$} be defined by \begin{equation}\label{e_definition_htt}
    \htt(\sum_{i\in I_A} x_i\alpha_i^\vee)=\sum_{i\in I_A} x_i,
\end{equation} for $(x_i)\in \Z^{I_A}$. We extend $\htt$ to a $\Q$-linear map $\htt:Y\otimes \Q\rightarrow \Q$. Although the restriction of $\htt$ to $Q^\vee$ is uniquely defined, there can be several possible choices for its extension to $Y\otimes \Q$ in general (when the Kac-Moody matrix $A$ is not invertible). Note that $\htt$ is denoted $\rho$ in \cite{braverman2016iwahori}.

For $E\subset \A$ is non-empty, we denote by $\conv(E)$\index{c@$\conv(E)$} the convex hull of $E$ in $\A$. We set $\conv_{Q^\vee}(E)=(E+Q^\vee)\cap \conv(E)$\index{c@$\conv_{Q^\vee}$}. 

For $\lambda,\mu\in Y$ such that $\lambda-\mu\in Q^\vee$, we set $[\lambda,\mu]_{Q^\vee}=[\lambda,\mu]\cap (\lambda+Q^\vee)$\index{ @$[\cdot,\cdot]_{Q^\vee}$, $]\cdot,\cdot]$, ...}, $]\lambda,\mu]_{Q^\vee}=]\lambda,\mu]\cap (\lambda+Q^\vee)$, $\ldots$ Note that if $\lambda\in Y$ and $w\in W_0$, then $w(\lambda)-\lambda\in Q^\vee$, by an immediate induction on $\ell(w)$.

\subsubsection{Fundamental chamber and Tits cone}
As in the reductive case, define the \textbf{fundamental chamber} as $C_{f}^{v}= \{v\in \A \ \vert \ \forall i\in {I_A},  \alpha_i(v)>0\}$\index{c@$C^v_f$}. 

 Let $\mathcal{T}= \displaystyle \bigcup_{w\in W_0} w.\overline{C^{v}_{f}}$\index{t@$\T$} be the \textbf{Tits cone}. This is a convex cone and $\overline{C^v_f}$ is a fundamental domain for the action of $W_0$ on $\cT$ (see \cite[1.4]{kumar2002kac}). In particular, if $x\in \cT$, then $W_0\cdot x\cap \overline{C^v_f}$ contains a unique element, that we denote $x^{++}$\index{x@$x^{++}$}.
 
 For $J\subset {I_A}$, set $F^v(J)=\{x\in \A\mid  \alpha_j(x)=0, \forall j\in J, \alpha_j(x)>0, \forall j\in {I_A}\setminus J\}$\index{f@$F^v(J)$}. A \textbf{positive vectorial face} (resp. \textbf{negative}) is a set of the form $w.F^v(J)$ (resp. $-w.F^v(J)$) for some $w\in W_0$ and $J\subset {I_A}$. Then by \cite[5.1 Th\'eor\`eme (ii)]{remy2002groupes}, the family of positive vectorial faces of $\A$ forms a partition of $\T$.

We set $Y^{++}=Y\cap\overline{C^v_f}$\index{y@$Y^+$, $Y^{++}$} and $Y^+=Y\cap \T$.

For $x,y\in \A$, we write $x\leq_{Q^\vee}y$\index{ @$\leq_{Q^\vee}$} (resp. $x\leq_{Q^\vee_\R}y$\index{ @$\leq_{Q^\vee_\R}$}) if $y-x\in Q^\vee_+=\bigoplus_{i\in {I_A}} \Z_{\geq 0}\alpha_i^\vee$\index{q@$Q^\vee_+$} (resp. $y-x\in Q^\vee_{\R,+}=\bigoplus_{i\in {I_A}} \R_{\geq 0}\alpha_i^\vee$\index{q@$Q^\vee_{\R,+}$}).

Note that if $a,b\in Y$, then $[a,b]_{\leq_{Q^\vee}}=\{c\in Y\mid a\leq_{Q^\vee} c\leq_{Q^\vee} b\}$ is finite since $(Y,\leq_{Q^\vee})$ is isomorphic to $(\Z^n,\leq)$ where $n$ is the rank of $Y$ and $\underline{a}=(a_1,\ldots,a_n)\leq \underline{b}=(b_1,\ldots,b_n)$ if and only if $a_i\leq b_i$ for all $i\in \llbracket 1,n\rrbracket$, for all $\underline{a},\underline{b}\in \Z^n$.

\subsubsection{Parabolic subgroups and minimal cosets representatives}\label{sss_parabolic_subgroups}

For $E$ a non-empty subset of $\A$, we denote by $W_{0,E}$\index{w@$W_{0,E}$} its fixator in $W_0$. If $\lambda\in \A$, we write $W_{0,\lambda}$ instead of $W_{0,\{\lambda\}}$. Let $\tilde{\lambda}\in \overline{C^v_f}$. Then by \cite[1.4.2 Proposition]{kumar2002kac}, $W_{0,\tilde{\lambda}}=\langle r_i\mid i\in {I_A}, \alpha_i(\tilde{\lambda})=0\rangle\subset W_0$\index{w@$W_{0,\tilde{\lambda}}$}. Subgroups of this form are called \textbf{standard parabolic subgroups}. In particular an element $\lambda\in \cT$ is \textbf{regular}, i.e  does not belong to any wall, if, and only we have $W_{0,\lambda}=\{1\}$. Moreover, if $J\subset I_A$, then the fixator of $F^v(J)$ is the fixator of any element $\lambda$ of $F^v(J)$. 

For $\tilde{\lambda}\in \overline{C^v_f}$, we denote by $W^{\tilde{\lambda}}_0$\index{w@$W^{\tilde{\lambda}}_0$} the set of minimal length representatives for left cosets $W_0/ W_{0,\tilde{\lambda}}$. By \cite[1.3.17]{kumar2002kac}, for every $w\in W^{\tilde{\lambda}}_0$ and $v\in W_{0,\tilde{\lambda}}$, we have $\ell(wv)=\ell(w)+\ell(v)$. In particular, for every $w\in W_0$, $wW_{0,\tilde{\lambda}}$ contains a unique minimal length element $\pr_{W^{\tilde{\lambda}}_0}(w)$\index{p@$\pr_{W^{\tilde{\lambda}}_0}$} of $W^{\tilde{\lambda}}_0$, and we call it the projection of $w$ on $W^{\tilde{\lambda}}_0$. Let $\lambda\in Y^{+}$. Let $v\in W_0$ be such that $\lambda=v(\lambda^{++})$. Denote by $w_\lambda$\index{w@$w_\lambda$} the projection of $v$ on $W_0^{\lambda^{++}}$. Then we have: \begin{equation}\label{e_definition_w_lambda}
    \lambda=w_\lambda(\lambda^{++})\text{ and }\ell(w_\lambda)<\min\{\ell(u)\mid u\in W_0\setminus\{w_\lambda\}, u(\lambda^{++})=\lambda\}\leq +\infty.
\end{equation}

By \cite[Proposition 3.12 c)]{kac1994infinite}  and \cite[Lemma 2.4]{gaussent2014spherical}, we have \begin{equation}\label{e_GR2.4}
    \lambda \leq_{Q^\vee} \lambda^{++}, x\leq_{Q^\vee_\R}x^{++}
\end{equation} for every $\lambda\in Y^{+}$ and $x\in \cT$. Moreover, by \cite[Lemma 4.8]{abdellatif2019completed}, if $\lambda\in Y$, then  we have: \begin{equation}\label{e_AH4.8}
    (W_0\cdot \lambda \text{ is upper bounded for $\leq_{Q^\vee})$ if and only if }(\lambda\in Y^+).
\end{equation}

\begin{Lemma}\label{l_Wlambda}
    Let $\tilde{\lambda}\in Y^{++}$, $w\in W^{\tilde{\lambda}}_0\setminus\{1\}$ and $i\in {I_A}$ be such that $r_iw<w$. Then $r_iw\in W^{\tilde{\lambda}}_0$.
\end{Lemma}

\begin{proof}
    If $r_iw\notin W^{\tilde{\lambda}}_0$, then we can write $r_iw(\tilde{\lambda})=v(\tilde{\lambda})$, with $\ell(v)<\ell(w)$. But then $w(\tilde{\lambda})=r_iv(\tilde{\lambda})$ and $\ell(r_iv)<\ell(w)$: a contradiction since $w\in W^{\tilde{\lambda}}_0$.
\end{proof}

The \textbf{extended affine Weyl group} is $W^a=W_0\ltimes Y$\index{w@$W^a$}. We also set $W^+=W_0\ltimes Y^+$\index{w@$W^+$}, the \textbf{affine semi-group.}  We use bold face letters to denote the elements of $W^+$.

\subsection{Split Kac-Moody group à la Tits}\label{ss_KM_groups}

In this subsection, we briefly recall some facts about split Kac-Moody groups over local fields. We describe  the Iwahori-Hecke algebra associated to such a group in the next subsection. As the computation rules and properties we give are sufficient for our purpose, very few backgrounds on Kac-Moody groups are necessary to understand the paper. We refer to \cite[7]{marquis2018introduction} for more information on the subject.

Let $\cS$ be a Kac-Moody root data. The Kac-Moody functor $\mathbf{G}=\mathbf{G}_\cS$ is a functor  from the category of fields to the category of groups, where $\cS$ is a Kac-Moody data as defined in \S\ref{ss_Standard_apartment}. It was introduced by Tits in \cite{tits1987uniqueness}. The functor $\mathbf{G}$ contains a split maximal torus $\bT$, isomorhic to $\cF\mapsto (\cF^\times)^{\mathrm{rank}_\Z(X)}$. Moreover $\mathbf{G}$ contains  a subgroup functor $x_\alpha$ isomorphic to $\mathbb{G}_a$ for every $\alpha\in \Phi$. If $\cF$ is a field, we have $\mathbf{G}(\cF)=\langle \bT(\cF),x_\alpha(\cF)\mid \alpha\in \Phi\rangle$.

Let $\cK$ be a non-Archimedean local field. Let $N$ be the normalizer of $T$ in $G$. We have $W_0\simeq N/T$ and $W^a\simeq N/\bT(\cO)$, where $\cO$ is the ring of  integers of $\cF$. For each $\bw\in W^a$, we choose an element $n_\bw\in N$ corresponding to $\bw$ for the bijection above. Let $G=\bG(\cK)$. In order to study $G$, Gaussent and Rousseau defined a masure $\I$ on which $G$ acts in \cite{gaussent2008kac} and \cite{rousseau2016groupes}. We have $\I=\bigcup_{g\in G}g.\A$. The Iwahori subgroup  $I$ is the fixator of some alcove of $\I$, see \cite[2.4.1 1)]{bardy2025twin} for a description of this fixator. A sub-semigroup $G^{\geq 0}$\index{g@$G^{\geq 0}$} of $G$ is defined (and denoted $G^+$) in \cite[1.5]{bardy2016iwahori}. By the Iwahori-Matsumoto decomposition  (\cite[1.11]{bardy2016iwahori}), we have $G^{\geq 0}=\bigsqcup_{\bw\in W^+}I n_\bw I.$ With our conventions, $G^{\geq 0}$ corresponds to the semigroup $G^-$ of \cite[1.5]{bardy2016iwahori}. The ``$\geq 0$'' refers to the Tits preorder (that we do not introduce here) on the masure $\I$ of $G$.

\subsection{Definition of the Iwahori-Hecke algebras associated with $G$ and $\cS$}\label{ss_Def_IH_algebras}

We now introduce the Iwahori-Hecke algebra that we study. It will be associated with a triple $(\cS,\cR,q)$ where $\cS$ is a Kac-Moody root data, $\cR$ is a commutative ring, and  $q\in \cR$ is a fixed element.

We first introduce the Iwahori-Hecke algebra of $\mathbf{G}(\cK)$, where $\cR$ is $\Z$ and $q$ is the residue cardinality of $\cK$. For $\wb\in W^+$, we denote by $T_{\wb}$\index{t@$T_\wb$} the indicator function of $I n_\wb I$.  Then the \textbf{Iwahori-Hecke algebra of }$G$ with coefficients in $\Z$ is the free $\Z$-module with basis $(T_{\wb})_{\wb\in W^+}$ equipped with the product defined by
\begin{equation}\label{e_product_indicator_functions}  
T_{\mathbf{v}}T_{\wb}=\sum_{\mathbf{u}\in W^+} a^{\mathbf{u}}_{\mathbf{v},\mathbf{w}}(\cK)T_\bu,
\end{equation} where 
\begin{equation}\label{e_structure_coefficients}
a^{\mathbf{u}}_{\mathbf{v},\mathbf{w}}(\cK)=|(I n_\mathbf{v} I \cap n_\mathbf{u} I n_\mathbf{w}^{-1} I)/I|
\end{equation} for $\mathbf{u},\mathbf{v},\mathbf{w}\in W^+$. The fact that such an algebra is well defined is \cite[5]{braverman2016iwahori} if $\bG$ is affine and \cite[Theorem 2.4]{bardy2016iwahori} in general.

By \cite[Theorem 3.22]{muthiah2018iwahori}, \cite[6.7]{bardy2016iwahori} or \cite{bardy2021structure}, for all $\bu,\bv,\bw\in W^+$, there exists $a_{\bv,\bw}^\bu(\bq)\in \Z[\bq]$, where $\bq$ is an indeterminate, such that $a_{\bv,\bw}^{\bu}(\cK)$ is the evaluation of $a_{\bv,\bw}^{\bu}$ at $q$, for all non-Archimedean local fields $\cK$ with residue cardinality $q$.

Moreover $a_{\bv,\bw}^{\bu}(\bq)$ only depends on $\bu,\bv,\bw$ and on the root generating data $\cS$. We return to the case of general $\cR$ and $q\in \cR$. The free $\cR$-module $\cH=\bigoplus_{\bw\in W^+} \cR T_{\bw}$ equipped with the product
\begin{equation}
  \label{eq:11}
T_{\mathbf{v}}T_{\wb}=\sum_{\mathbf{u}\in W^+} a^{\mathbf{u}}_{\mathbf{v},\mathbf{w}}(\cK)T_\bu,
\end{equation}
for $\mathbf{u},\mathbf{v},\mathbf{w}\in W^+$ 
is an associative algebra, called the Iwahori-Hecke algebra of $\cS$ with coefficients in $\cR$ with parameter $q$.

\subsection{Bernstein-Lusztig presentation}\label{ss_Bernstein_Lusztig_presentation_IH_algebras}
If $\bv,\bw\in W^+$, it is in general complicated to explicitly compute $T_{\bv}T_{\bw}$ or even to get information on its support. In \eqref{e_structure_coefficients}, the coefficients $a^{\bu}_{\bv,\bw}(q)$  are defined in \cite{bardy2016iwahori} as the cardinalities of certain sets of alcoves in a masure associated to $\cS$. They  are hard to compute and thus the convolution product is poorly understood in the $T$-basis. We refer to \cite{muthiah2024pursuing} for some results and conjectures on this subject. 

A way to make computations in $\cH$ is to use its Bernstein-Lusztig basis, which we describe below. 

From now onward, we make the following assumptions on $q \in \cR$.

\begin{Assumption}\label{a_assumption_ring}
We assume $q\in \cR^\times$. Recall the definition of $\htt$ in \eqref{e_definition_htt}.     Let $N\in \Z_{>0}$ be such that $\htt(Y)=\tfrac{1}{N}\Z$.

We further assume the existence of an $N$-th root $q^{1/N}$ of $q$ in $\cR$.  We fix $q^{1/N}$ once and for all.  Let $\delta:Q^\vee \rightarrow \cR^\times$ be defined by $\delta(x)=q^{2\htt(x)}$, for $x\in Q^\vee$.  If we   set $\delta^{1/2}(\lambda)=(q^{1/N})^{N \htt(\lambda)}$\index{d@$\delta^{1/2}$}, for $\lambda\in Y$, we have $(\delta^{1/2}(\lambda))^2=\delta(\lambda)$ for all $\lambda\in Q^\vee$ and thus the assumption of \cite[5.7]{bardy2016iwahori} is satisfied. Note that depending on the Kac-Moody matrix, $N$ can be arbitrarily large, see  \cite[Remark 2.13]{muthiah2024pursuing}.
\end{Assumption}

\begin{Definition}\label{d_BL_algebra}
  Let $^{BL}\cH$ be the free $\cR$-module with basis $(Z^\lambda T_w)_{\lambda\in Y,w\in W_0}$. For short, one sets $T_{w} = Z^{0}T_{w}$ for $w \in W_0$ and $Z^{\lambda} = Z^{\lambda}T_{1}$\index{z@$Z^\lambda$} for $\lambda \in Y$. The \textbf{Bernstein-Lusztig-Hecke algebra} $^{BL}\cH_{\cR}$ is the module $^{BL}\cH_{\cR}$ equipped with the unique product that turns it into an associative algebra and satisfies the following \textbf{Bernstein-Lusztig relations}, for all $\lambda,\mu\in Y$, for all $i\in I_A$ and $w\in W_0$:

\begin{itemize}
\item (BL1) $Z^{\lambda}  T_{w} = Z^{\lambda}T_{w}$;
\item (BL2) $T_{i}T_{w}=\left\{\begin{aligned} & T_{r_iw} &\mathrm{\ if\ }\ell(r_iw)=\ell(w)+1\\ & (q-1) 
T_{w}+q T_{r_i w} &\mathrm{\ if\ }\ell(r_iw)=\ell(w)-1 \end{aligned}\right . \ $;
\item (BL3) $Z^{\lambda}  Z^{\mu} = Z^{\lambda + \mu}$;
\item (BL4) $T_{i}Z^{\lambda}-Z^{r_i(\lambda)}T_{i} =(q-1)\frac{Z^\lambda-Z^{r_i(\lambda)}}{1-Z^{-\alpha_i^\vee}}$.
\end{itemize}

The existence and uniqueness of such a product is standard (see e.g. \cite[Theorem 6.2]{bardy2016iwahori}). 
\end{Definition}

\begin{Definition}
    Let $\HCW=\bigoplus_{w\in W_0} \cR T_w$. By (BL2), this is the Hecke algebra of the Coxeter group $(W_0,\{r_i\mid i\in I_A\})$ associated with the coefficient $q$.
\end{Definition}

\begin{Remark}\label{r_explicit_BL_relations}
\begin{enumerate}
\item\label{itPolynomiality_Bernstein_Lusztig} Let $i\in {I_A}$ and $\lambda\in Y$. Then $\frac{Z^\lambda-Z^{r_i(\lambda)}}{1-Z^{-\alpha_i^\vee}}\in \cR[Y]$. Indeed, $\frac{1-Z^{-\alpha_i(\lambda)\alpha_i^\vee}}{1-Z^{-\alpha_i^\vee}}=\sum_{j=0}^{\alpha_i(\lambda)-1}Z^{-j\alpha_i^\vee}$   if $\alpha_i(\lambda)\geq 0$ and $\frac{1-Z^{-\alpha_i(\lambda)\alpha_i^\vee}}{1-Z^{-\alpha_i^\vee}}= -\sum_{j=1}^{-\alpha_i(\lambda)}Z^{j\alpha_i^\vee}$  if $\alpha_i(\lambda)\leq 0,$ and thus \begin{equation}\label{e_computation_BL}
\frac{Z^\lambda-Z^{r_i(\lambda)}}{1-Z^{-\alpha_i^\vee}}=Z^\lambda \frac{1-Z^{-\alpha_i(\lambda)\alpha_i^\vee}}{1-Z^{-\alpha_i^\vee}}=\left\{\begin{aligned}&\sum_{\mu\in [\lambda,r_i(\lambda)[_{Q^\vee}}Z^\mu  &\mathrm{\ if\ }\alpha_i(\lambda)\geq 0\\ & \sum_{\mu\in ]\lambda,r_i(\lambda)]_{Q^\vee}}-Z^\mu  &\mathrm{\ if\ }\alpha_i(\lambda)\leq 0.\end{aligned}\right.
\end{equation}

\item From (BL4) we deduce that for all $i\in {I_A}$, $\lambda\in Y$, \[Z^\lambda T_i-T_iZ^{r_i(\lambda)}=(q-1)\frac{Z^\lambda-Z^{r_i(\lambda)}}{1-Z^{-\alpha_i^\vee}}.\]

\item By (BL4), the family $(T_wZ^\lambda)_{w\in W_0, \lambda\in Y}$ is also a basis of $\cH_\cR$.

\end{enumerate}

\end{Remark}

Let $^{BL}\cH^+=\bigoplus_{\lambda\in Y^+,w\in W_0} \cR T_w Z^\lambda \subset {^{BL}\cH}$. Then as $Y^+$ is stabilized by $W_0$, and as $Y^+=Y\cap \cT$, where $\cT$ is a convex cone, we have that $^{BL}\cH_{\cS,\cR}^+$ is a sub-algebra of $^{BL}\cH_{\cS,\cR}$.

 By \cite[5]{bardy2016iwahori}, we have the following theorem.

\begin{Theorem}\label{t_BL_relations}

There is a unique isomorphism
  \begin{equation}
    \label{eq:12}
\cH\overset{\sim}{\rightarrow} {^{BL}\cH^+}
  \end{equation}
such that $T_w \mapsto T_w$ for all $w \in W_0$, and $T_{\lambda} \mapsto \delta^{1/2}(\lambda) Z^\lambda$ for all $\lambda \in Y^{++}$.

\end{Theorem}

From now on, we identify $\cH$ and its image in ${{^{BL}\cH}^+}$ and we drop the notation ${^{BL}\cH^+}$. Note that when $G$ is reductive, we recover the usual Iwahori-Hecke algebra of $G$, since $Y\cap \T=Y$.

\begin{Remark}\label{r_invertibility_Z_T}
    Let $\lambda\in Y^+$. Then $Z^\lambda$ and $T_\lambda$ are invertible in  ${^{BL}}\cH$. Indeed,  $Z^\lambda Z^{-\lambda}=1$. By \cite[Corollary 4.3]{bardy2016iwahori}, if $w\in W_0$ is such that $\lambda=w(\lambda^{++})$, then we have $T_\lambda=T_w T_{\lambda^{++}} T_w^{-1}$ and by Theorem~\ref{t_BL_relations}, $T_{\lambda^{++}}\in \cR^\times Z^{\lambda^{++}}$, which proves that $T_\lambda$ is invertible in ${^{BL}}\cH$. 
\end{Remark}

\subsection{Left and right actions of $\HC_{W_0}$ on $\HC$ }\label{ss_left_right_actions_HCW_HC}

Let $\bw\in W^+$ and $i\in {I_A}$. Then explicit formulas for $T_{\bw}T_i$ and $T_iT_{\bw}$ are given in \cite[4]{bardy2016iwahori} and \cite[3]{muthiah2018iwahori}. We use them to study  $\cH$ as a left and right module over  $\cH_{W_0}$.  Note that we use the conventions of \cite{bardy2016iwahori} rather than the ones of \cite{muthiah2018iwahori}. As they slightly differ, we need to use the dictionary of the appendix to pass from  \cite{muthiah2018iwahori} results to ours. Roughly speaking, multiplication on the right in \cite{muthiah2018iwahori} setting corresponds to a multiplication on the left here.

For $i\in {I_A}$, we have \begin{equation}\label{e_inverse_T_i}
    T_i^{-1}=q^{-1}T_i-(1-q^{-1}).
\end{equation}

Let $\bw=\lambda\cdot w\in W^+$.  By \cite[Proposition 4.1]{bardy2016iwahori} or \cite[3]{muthiah2018iwahori}, we have the Iwahori-Matsumoto relations: \begin{align}
    T_{r_i\bw}&=\begin{cases}T_iT_{\bw} &\text{ if }\alpha_i(\lambda)>0 \text{ or }\alpha_i(\lambda)=0 \text{ and }\ell(r_iw)>\ell(w) \\ T_i^{-1}T_{\bw} \ &\text{ otherwise }   
    \end{cases}\label{e_IM_relations_l}\\
     T_{\bw r_i}&=\begin{cases}T_{\bw}T_i &\text{ if }w(\alpha_i)(\lambda)<0 \text{ or }\ell(wr_i)>\ell(w)  \\ T_{\bw}T_i^{-1} \ &\text{ otherwise}   \label{e_IM_relations_r}
    .\end{cases}
\end{align}

If $E$ is a subset of $Y^+$, one sets $\HC_E=\bigoplus_{\lambda\in E} \HC_{W_0}\Tb_{\lambda}$\index{h@$\cH_{E}$}.

\begin{Proposition}\label{p_HCW_bimodule}
\begin{enumerate}
\item Let $\lambda\in Y^+$. Then $\HCW\Tb_{\lambda}=\bigoplus_{w\in W_0} \cR \Tb_{w(\lambda)\cdot w}$ and  $T_\lambda\cH_{W_0}=\bigoplus_{w\in W_0} \cR T_{\lambda.w}$.

\item We have $\HC=\bigoplus_{\lambda\in Y^+}\HC_{W_0}\Tb_{\lambda}=\bigoplus_{\lambda\in Y^+} T_\lambda\cH_{W_0}$.

\item Let  $E\subset Y^+$. If $E$ is stable under the action of $W_0$, then $\cH_{E}=\bigoplus_{\lambda\in E} \cH_{W_0} T_\lambda$ is a two-sided $\cH_{W_0}$-submodule of $\cH$.

\end{enumerate}  
\end{Proposition}

\begin{proof} 
Let $v\in W_0$ and $i\in {I_A}$. By \eqref{e_inverse_T_i}, \eqref{e_IM_relations_l} and \eqref{e_IM_relations_r}, we have: \begin{equation}\label{e_left_multiplication_Ti}\Tb_i\Tb_{v(\lambda)\cdot v}\in \cR \Tb_{v(\lambda)\cdot  v}\oplus \cR \Tb_{r_iv(\lambda)\cdot r_iv}\subset \bigoplus_{w\in W_0} \cR \Tb_{w(\lambda)\cdot w},\end{equation} \begin{equation}\label{e_right_multiplication_T_i}
    T_{\lambda \cdot v}T_i\in \cR T_{\lambda\cdot v}\oplus \cR T_{\lambda\cdot vr_i}\subset \bigoplus_{w\in W_0} \cR T_{\lambda\cdot w},
\end{equation} \[\Tb_{r_iv(\lambda)\cdot r_iv}\in \{\Tb_i\Tb_{v(\lambda)\cdot v},\Tb_i^{-1}\Tb_{v(\lambda)\cdot v}\}\subset \cH_{W_0}T_{v(\lambda)\cdot v}\text{ and }T_{\lambda\cdot vr_i}\in \{T_{\lambda\cdot v}T_i,T_{\lambda\cdot v}T_i^{-1}\}\subset T_{\lambda\cdot v}\cH_{W_0}.\]  By induction on $\ell(v)$, we deduce (1) and then (2). 

Let now $E\subset Y^+$ be a $W_0$-stable subset of $Y^+$. In order to prove (3), it suffices to prove that $\cH_E$ is stable by right multiplication by $\cH_{W_0}$. As the algebra $\cH_{W_0}$ is generated the $T_i$, $i\in {I_A}$, it suffices to prove that $\cH_{E}$ is stable by right multiplication by $T_i$, for $i\in {I_A}$. Let $\lambda\in E$ and $i\in {I_A}$. Then by \eqref{e_right_multiplication_T_i} and \eqref{e_left_multiplication_Ti}, we have \[T_\lambda T_i\in \cR T_\lambda\oplus \cR T_{\lambda\cdot r_i}\subset \cR T_\lambda\oplus  \cH_{W_0} T_{r_i(\lambda)}\subset \bigoplus_{\mu\in W_0\cdot \lambda}\cH_{W_0} T_\mu\subset \cH_{E}.\]

\end{proof}

\begin{Definition}\label{d_Z_T_support}

We have $\cH=\bigoplus_{\lambda\in Y^+} \HCW   T_\lambda=\bigoplus_{\lambda\in Y^+} \HCW Z^\lambda$. Let $x\in \cH$. Write $x=\sum_{\lambda\in Y^+} h^T_\lambda T_\lambda=\sum_{\lambda\in Y^+} h^Z_\lambda Z^\lambda$, with $h^T_\lambda,h^Z_\lambda\in \HCW$, for $\lambda\in Y^+$. For $\lambda\in Y^{+}$, we set $\pr^T_\lambda(h)=h_\lambda^T T_\lambda$\index{p@$\pr^T_\lambda$} and  $\pr^Z_\lambda(h)=h_\lambda^Z Z^\lambda$\index{p@$\pr^Z_\lambda$}. We  set $\coeff^T_\lambda(h)=h^T_\lambda$\index{c@$\coeff^T_\lambda$} and $\coeff^Z_\lambda(h)=h^Z_\lambda$\index{c@$\coeff^Z_\lambda$}. This is well defined by Remark~\ref{r_invertibility_Z_T}. Note that we have $\pr_\lambda^T(h)=\coeff^T_\lambda(h)T_\lambda$ and $\pr_\lambda^Z(h)=\coeff^Z_\lambda(h)Z^\lambda$ for $\lambda\in Y^+$.

We then set $\supp^Z(x)=\{\lambda\in Y^+\mid \pr_\lambda^Z(x)\neq 0\}$ and $\supp^T(x)=\{\lambda\in Y^+\mid \pr_\lambda^T(x)\neq 0\}$\index{s@$\supp^Z,\supp^T$}. These are finite sets, which we call the \textbf{$Z$-support }and the \textbf{$T$-support} of $x$ respectively.
\end{Definition}

\section{From the $Z$-basis to the $T$-basis and vice-versa}\label{s_Z_basis_to_T_basis}

\subsection{Study of the support of $Z^\lambda T_w$ and $Z^\lambda T_{w}^{-1}$, for $\lambda\in Y^+$ and $w\in W_0$}\label{ss_study_support}

We now define inductively subsets $\sT_w(E),\overline{\sT_w}(E),\widetilde{\sT_w}(E)$ and $\widehat{\sT}_w(E)$ of $Y^+$, for $E\subset Y^+$ and $w\in W_0$ that enable to control the support of $Z^\lambda T_w$ and $Z^\lambda T_{w}^{-1}$, for $w\in W_0$ and $\lambda\in Y^+$. These sets will more generally be useful to obtain information on the support of elements of $\cH$.

\begin{Definition}\label{d_sT}

Let $\underline{W}^v={I_A}^{(\N)}$ be the set of finite words $(r_{i_1},\ldots,r_{i_k})$ such that $k\in {\Z_{\geq 0}}$ and $i_1,\ldots,i_k\in {I_A}$. If $w\in W_0$, we denote by $\Red(w)$\index{r@$\Red$} the set reduced decompositions of $w$, i.e the set of elements $(r_{i_1},\ldots,r_{i_k})\in \underline{W}^v$ such that $k=\ell(w)$ and $w=r_{i_1}\ldots r_{i_k}$.  For $\underline{m}\in \underline{{W_0}}$,  the \textbf{length} $\ell(\underline{m})$ of $\underline{m}$ is the number of terms of $\underline{m}$.

Let $\lambda\in Y$ and $i\in {I_A}$. One sets \[\sT_i(\lambda)=\begin{cases}\{r_i(\lambda)\} &\text{ if }\alpha_i(\lambda)\geq 0 \\ \{\lambda,r_i(\lambda)\} &\text{ if }\alpha_i(\lambda)<0   
    \end{cases}, \overline{\sT_i}(\lambda)=\sT_i(\lambda)\cup]\lambda,r_i(\lambda)[_{Q^\vee}=\begin{cases}]\lambda,r_i(\lambda)]_{Q^\vee} &\text{ if }\alpha_i(\lambda)\geq 0 \\ [\lambda,r_i(\lambda)]_{Q^\vee} &\text{ if }\alpha_i(\lambda)<0   
    \end{cases},\]  \[\tilde{\sT_i}(\lambda)=\begin{cases}[\lambda,r_i(\lambda)]_{Q^\vee}&\text{ if }\alpha_i(\lambda)\geq 0 \\ ]\lambda,r_i(\lambda)]_{Q^\vee} &\text{ if }\alpha_i(\lambda)<0   
    \end{cases} \text{ and } \widehat{\sT_i}(\lambda)=[\lambda,r_i(\lambda)]_{Q^\vee}.\]

 We write $\overrightarrow{\cdot}\in \{\emptyset, \overline{\cdot}, \widetilde{\cdot},\widehat{\cdot}\}$ to mean that $\overrightarrow{}$ denotes either no symbol or one of the three remaining symbols, which will be used as an overset symbol.

  If $E$ is a set and $\overrightarrow{\cdot}\in \{\emptyset, \overline{\cdot}, \widetilde{\cdot},\widehat{\cdot}\}$, one sets $\overrightarrow{\sT_i}(E)=\bigcup_{\lambda\in E} \overrightarrow{\sT_i}(\lambda)$, which defines a map $\overrightarrow{\sT_i}:\sP(Y)\rightarrow \sP(Y)$ (where $\sP(Y)$ is the  powerset of $Y$). 

Let $E$ be a subset of $Y$ and  $\underline{m}=(r_{i_1},\ldots,r_{i_k})\in \underline{{W_0}}$. We set  \[\overrightarrow{\sT_{\underline{m}}}(E)=\overrightarrow{\sT_{i_1}}\circ \ldots \circ \overrightarrow{\sT_{i_k}}(E)\text{ and }\overrightarrow{\sT_w}(E)=\bigcup_{\underline{m}\in \Red(w)}\overrightarrow{\sT_{\underline{m}}}(E).\]\index{t@$\overrightarrow{\sT_w}(E)$} If $\lambda\in Y$ and $\underline{m}\in \underline{{W_0}}$, we often write $\overrightarrow{\sT_{\underline{m}}}(\lambda)$ instead of $\overrightarrow{\sT_{\underline{m}}}(\{\lambda\})$. Note that since $\cT$ is a convex cone, we have $\overrightarrow{\sT_w}(E)\subset Y^+$ when $E\subset Y^+$. \end{Definition}

Note that for $\overrightarrow{\cdot}\in \{\emptyset,\overline{\cdot},\widetilde{\cdot},\widehat{\cdot}\}$ we are not claiming that for 
$\underline{m_1},\underline{m_2}\in \Red(w)$ that  $\overrightarrow{\sT_{\underline{m_1}}}(\lambda)$ and $\overrightarrow{\sT_{\underline{m_2}}}(\lambda)$ are equal.

\begin{Lemma}\label{l_sT_as_a_union}
  Let $\overrightarrow{\cdot}\in \{\emptyset,\overline{\cdot},\widetilde{\cdot},\widehat{\cdot}\}$.
  
 Let $\underline{m}\in \underline{W_0}$ and $E\subset Y$. Then we have $\overrightarrow{\sT_{\underline{m}}}(E)=\bigcup_{\lambda\in E}\overrightarrow{\sT_{\underline{m}}}(\lambda)$. 
\end{Lemma}

\begin{proof}
    We prove it by induction on $k=\ell(\underline{m})$. If $k\leq 1$, then this is clear by definition.   We assume that $k\geq 2$ and that for all $\underline{m'}\in \underline{W_0}$ such that $\ell(\underline{m'})\leq k-1$, we have $\overrightarrow{\sT_{\underline{m'}}}(E)=\bigcup_{\lambda\in E}\overrightarrow{\sT_{\underline{m'}}}(\lambda)$. Write $\underline{m}=(i_1,\ldots,i_k)$. Then we have \[\overrightarrow{\sT_{\underline{m}}}(E)=\overrightarrow{\sT_{i_1}}(\overrightarrow{\sT_{(i_2,\ldots,i_k)}}(E))=\overrightarrow{\sT_{i_1}}\left(\bigcup_{\lambda\in E}\overrightarrow{\sT_{(i_2,\ldots,i_k)}}(\lambda)\right)=\bigcup_{\lambda\in E}\overrightarrow{\sT_{i_1}}(\overrightarrow{\sT_{(i_2,\ldots,i_k)}}(\lambda))=\bigcup_{\lambda\in E} \overrightarrow{\sT_{\underline{m}}}(\lambda).\] Lemma follows.
\end{proof}

\begin{Lemma}\label{l_t}
\begin{enumerate}
    \item For every subset $E$ of $Y$, for every $w\in W_0$, we have: \begin{equation}\label{e_inclusions_sT}
    \sT_w(E)\subset \overline{\sT_w}(E), \widetilde{\sT_w}(E)\subset \widehat{\sT_w}(E)\subset \conv_{Q^\vee}(W_0\cdot E).
\end{equation}

 \item Let $\lambda\in Y$ and $w\in W_0$. Then $w(\lambda)\in \sT_w(\lambda)$. Moreover if $\lambda\in Y^{++}$, we have $\sT_w(\lambda)=\{w(\lambda)\}$. 
\end{enumerate}
\end{Lemma}

\begin{proof}
    (1) Except the last inclusion, it follows from the same inclusions for $w$ a simple reflection and $E$ a singleton. For the  last inclusion, take $\mu\in \conv_{Q^\vee}(W_0\cdot E)$ and $i\in {I_A}$. Then  we have $\widehat{\sT_i}(\mu)=[r_i(\mu),\mu]_{Q^\vee}\subset \conv_{Q^\vee}(W_0\cdot E)$, for all  $\mu\in \conv_{Q^\vee}(W_0\cdot E)$. Therefore, we have $\widehat{\sT}_i(\conv_{Q^\vee}(W_0\cdot E))\subset \conv_{Q^\vee}(W_0\cdot E)$. We deduce that $\widehat{\sT_w}(E)\subset \conv_{Q^\vee}(W_0\cdot E)$ for every $w\in W_0$ by induction on $\ell(w)$. 

    (2) is clear.
\end{proof}
 
\begin{Lemma}\label{lemBernstein_Lusztig_relation_inverse}
Let $\lambda\in Y^+$ and $i\in {I_A}$. Then: \begin{equation}\label{e_Z_lambda_Tinverse}
    Z^{\lambda}\Tb_i^{-1}=q^{-1}\Tb_iZ^{r_i(\lambda)}+(1-q^{-1})(\frac{Z^\lambda-Z^{r_i(\lambda)}}{1-Z^{-\alpha_i^\vee}}-Z^\lambda).
\end{equation} For all $w\in W_0$, we have 
\begin{equation}\label{eqSupport_Zlambda_Twinverse}
\supp^Z(\Zb^\lambda \Tb_w)\subset \widetilde{\sT_{w^{-1}}}(\lambda) \text{ and }\supp^Z(\Zb^\lambda \Tb_w^{-1})\subset \overline{\sT_w}(\lambda).
\end{equation}
\end{Lemma}

\begin{proof}
Let $\mu\in Y^+$. We have $\Tb_i^{-1}=q^{-1}\Tb_i+q^{-1}-1$. Moreover, 
\[Z^\mu \Tb_i=(q-1)\frac{Z^\mu-Z^{r_i(\mu)}}{1-Z^{-\alpha_i^\vee}}+\Tb_i  Z^{r_i(\mu)},\] and hence we have \eqref{e_Z_lambda_Tinverse}. By \eqref{e_computation_BL}, we deduce \begin{equation}\label{e_Zmu_Ti_inverse}
 \supp^Z(Z^\mu T_i)=\widetilde{\sT_i}(\mu) \text{ and }   \supp^Z(Z^\mu T_i^{-1})=\overline{\sT_i}(\mu).
\end{equation} 

Let $\lambda\in Y^+$. Let $w\in W_0$ be such that $\ell(w)\geq 2$. We assume that for all $v\in W_0$ such that $\ell(v)<\ell(w)$, we have $\supp^Z(Z^\lambda T_{v})\subset \widetilde{\sT_{v^{-1}}}(\lambda)$ and $\supp^Z(Z^\lambda T_{v}^{-1})\subset \overline{\sT_v}(\lambda)$. Let $i,j\in {I_A}$ be such that $\ell(wr_i)=\ell(r_jw)<\ell(w)$. Let $u=wr_i$ and $v=r_jw$. Then \begin{align*}
    \supp^Z(Z^\lambda T_{w})&=\supp^Z\left(Z^\lambda T_u T_i\right) \\
    &\subset \bigcup_{\mu\in \supp^Z(Z^\lambda T_u)}\supp^Z(Z^\mu T_i) \text{ by \eqref{e_Zmu_Ti_inverse}}\\
    &\subset \bigcup_{\mu\in \supp^Z(Z^\lambda T_u)} \widetilde{\sT_i}(\mu) \\
    &\subset \widetilde{\sT_i}(\widetilde{\sT_{u^{-1}}}(\lambda))\subset \widetilde{\sT_{w^{-1}}}(\lambda).
\end{align*}

 and

\begin{align*}
    \supp^Z(Z^\lambda (T_{w})^{-1})&=\supp^Z\left((Z^\lambda T_v^{-1}) T_j^{-1}\right)\\
    &\subset \bigcup_{\mu\in \supp^Z(Z^\lambda T_v^{-1})}\supp^Z(Z^\mu T_j^{-1})\text{ by \eqref{e_Zmu_Ti_inverse}}\\
    &\subset \bigcup_{\mu\in \supp^Z(Z^\lambda T_v^{-1})} \overline{\sT_j}(\mu) \\
    &\subset \overline{\sT_j}(\overline{\sT_v}(\lambda))\subset \overline{\sT_w}(\lambda).
\end{align*}

We deduce the lemma by induction.
\end{proof}

\begin{Remark}\label{r_comparison_T_R}

In \cite{bardy2016iwahori}, the authors associate an Iwahori-Hecke algebra $\cH$ to a group  $G$ acting strongly transitively on a thick masure of finite thickness (satisfying additional conditions, see \cite[1.4.3]{bardy2016iwahori}). This enables them to associate $\cH$ to an \emph{almost split} Kac-Moody group over a local field $G$.

This Hecke algebra also admits a (slighly more involved) Bernstein-Lusztig presentation \cite[Theorem 6.2]{bardy2016iwahori}. In particular, instead of a single $q$, these relations involve different $q_i$, for any $i\in I_A$. Despite this difference, we can also prove \eqref{eqSupport_Zlambda_Twinverse} in this framework and thus obtain the main results of this paper with slight modifications. Note however that in this framework, we do not know whether we can take the $q_i$ to be indeterminates, or elements of an arbitrary ring (see \cite[Conjecture 2]{bardy2021structure}).
\end{Remark}

\subsection{Depth and line segments}\label{ss_depth_line_segments}

\begin{Definition}
    Let $\lambda\in Y^+$. The depth with respect to $\lambda$ is the map $\mathrm{depth}_\lambda:(\lambda+Q^\vee)\cap Y^+\rightarrow \Z$\index{d@$\mathrm{depth}_\lambda$} defined by $\mu\mapsto \htt(\lambda^{++}-\mu^{++})$. 
\end{Definition}

\begin{Lemma}\label{lemGrowth_dominant_height_segments}
 Let $x\in \T$ and $w\in {W_0}$ be such that $w(x)\neq x$. Then for all $y \in ]x,w(x)[$, we have $y^{++}<_{Q^\vee_\R} x^{++}$, i.e $y^{++}-x^{++}\in Q^\vee_{\R,+}\setminus\{0\}$, where $Q^\vee_{\R,+}=\bigoplus_{i\in {I_A}}\R_{\geq 0}\alpha_i^\vee$. In particular, $\htt(x^{++}-y^{++})> 0$. Consequently if $x\in Y^+$, then $y^{++}-x^{++}\in (\bigoplus_{i\in {I_A}}\Z_{\geq 0} \alpha_i^\vee)\setminus\{0\}$ and  $\htt(x^{++}-y^{++})\in \Ne$.
 \end{Lemma}
 
 \begin{proof}
Write $y=tx+(1-t)w(x)$, where $t\in ]0,1[$. Let $v\in W_0$ be such that $v(y)=y^{++}$. By \eqref{e_GR2.4}, there exist $q^\vee,\tilde{q}^\vee\in Q^\vee_{\R,+}$ such that $v(x)=x^{++}-q^\vee$ and $vw(x)=x^{++}-\tilde{q}^\vee$. Then  \[y^{++}=v(y)=tv(x)+(1-t)vw(x)=x^{++}-(tq^\vee+(1-t)\tilde{q}^\vee),\] which proves the lemma.
 \end{proof}

\begin{Remark}
Fix $\lambda,\mu\in Y^+$ such that $\lambda-\mu\in Q^\vee$. Let $(\lambda_n)\in (Y^+)^\N$ be such $\lambda_0\in W_0\cdot\lambda$ and for all $n\in \N$, there exists $i_n\in {I_A}$ such that $\lambda_{n+1}\in ]\lambda_n,r_{i_n}(\lambda_n)[_{Q^\vee}$. Then by Lemma~\ref{lemGrowth_dominant_height_segments}, the sequence $(\mathrm{depth}_{\lambda}(\lambda_n))=(\htt(\lambda^{++}-\lambda_n^{++}))$ is a strictly increasing sequence of $\N$. Hence for $n> \mathrm{depth}_\lambda(\mu)$, we have $\lambda_n\neq \mu$ and consequently $\{n\in \N\mid \lambda_n=\mu\}$ is finite. Our proof of the main finiteness result of this paper (Lemma~\ref{lemFiniteness_support_TWhat}) is based on the same phenomenon but is more subtule since the sequences ($\lambda_n)$ we need to consider now satisfy either $\lambda_{n+1}\in ]\lambda_n,r_{i_n}(\lambda_n)]_{Q^\vee}$ or $\lambda_{n+1}\in [\lambda_n,r_{i_n}(\lambda_n)]_{Q^\vee}$, depending on $\lambda_n$ and $i_n$. 
\end{Remark}

\begin{Lemma}\label{lemDominant_height_segments}
Let $\lambda\in Y^{+}$ and $i\in {I_A}$. Suppose that $\alpha_i(\lambda)>0$. Then 
\[\lambda^{++}>_{Q^\vee} (\lambda-\alpha_i^\vee)^{++}>_{Q^\vee}\ldots >_{Q^\vee} (\lambda-\lfloor \frac{\alpha_i(\lambda)}{2}\rfloor\alpha_i^\vee)^{++},\] In particular for all $k\in \llbracket 0,\frac{\alpha_i(\lambda)}{2}\rrbracket$, we have $\htt(\lambda^{++}-(\lambda-k\alpha_i^\vee)^{++})\geq k$. 
\end{Lemma}

\begin{proof}
Let $k\in \llbracket 0,\lfloor \alpha_i(\lambda)/2\rfloor-1\rrbracket$. Then $\lambda-(k+1)\alpha_i^\vee\in ]\lambda-k\alpha_i^\vee,r_i(\lambda-k\alpha_i^\vee)[_{Q^\vee}$. By Lemma~\ref{lemGrowth_dominant_height_segments}, we deduce $(\lambda-(k+1)\alpha_i^\vee)^{++}<_{Q^\vee} (\lambda-k\alpha_i^\vee)^{++}$ and we get the lemma by induction.
\end{proof}

\subsection{Passage from the $Z$-basis to the $T$-basis, and vice-versa}\label{ss_Z_to_T_to_Z}

For $\lambda\in Y^+$, we do not know any close or recursive formula describing the $Z$-support of $T_\lambda$ or the $T$-support of $Z^\lambda$. In this subsection, we give information on these supports. We prove that the passage between the $T$-basis and the $Z$-basis satisfies a triangularity property, see Proposition~\ref{P_BL_support_IM_basis}.

Recall that if $E$ is a subset of $\A$,   $\conv_{Q^\vee}(E)$ denotes the set $\conv(E)\cap (E+Q^\vee)$.

 \begin{Lemma}\label{lemInequality_convex_hull_orbit}
 Let $\lambda\in Y^+$ and $\mu\in \conv_{Q^\vee}({W_0}\cdot\lambda)$. Then $\mu^{++}\leq_{Q^\vee} \lambda^{++}$. 
 \end{Lemma}
 
 \begin{proof}
 Write $\mu=\sum_{w\in {W_0}} a_w w(\lambda)$, with $(a_w)\in [0,1]^{({W_0})}$ such that $\sum_{w\in {W_0}} a_w=1$. Let $v\in {W_0}$ be such that $v(\mu)=\mu^{++}$. Then by \eqref{e_GR2.4}, we have: \[\mu^{++}=v(\mu)=\sum_{w\in {W_0}} a_w vw(\lambda)\leq_{Q^\vee} \sum_{w\in {W_0}} a_w \lambda^{++}=\lambda^{++}.\] 
 \end{proof}

\begin{Lemma}\label{lemT_support_Z_convex_hull}
Let $\lambda\in Y^+$. Then  $Z^\lambda\in \HC_{\conv_{Q^\vee}({W_0}\cdot \lambda)}$.
\end{Lemma}

\begin{proof}
    By  \cite[Theorem 5.5 (2)]{bardy2016iwahori} and Proposition~\ref{p_HCW_bimodule}, we have $Z^\lambda\in \sum_{\nu\in \conv_{Q^\vee}({W_0}\cdot \lambda)}\Tb_{\nu} \HC_{W_0}\subset \cH_{\conv_{Q^\vee}(W_0\cdot \lambda)}$ (by \cite[5.7]{bardy2016iwahori}, $Z^\lambda$ is an element of $\cR X^\lambda$).
\end{proof}

\begin{Definition}\label{d_dominant_preorder}
    We define $\leq^{++}$\index{ @$\leq^{++}$} on $Y^+$ by $\lambda\leq^{++}\mu$ if and only if $\lambda^{++}\leq_{Q^\vee}\mu^{++}$, for every $\lambda,\mu\in Y^+$. Note that this is a preorder but not an order since $\lambda\leq^{++}\mu\leq^{++}\lambda$ for all $\lambda,\mu\in Y^+$ such that $\mu\in W_0\cdot \lambda$.

\end{Definition}

For $\lambda\in Y^{++}$, set \begin{equation}\label{e_notation_bimodule}
    \cH_{ \leq^{++}\lambda}=\bigoplus_{\mu\in ]-\infty,\lambda]_{ \leq^{++}}}\HCW  T_{\mu} \text{ and }\cH_{<^{++}\lambda}=\bigoplus_{\mu\in ]-\infty,\lambda[_{ \leq^{++}}}\HCW* T_{\mu}.
\end{equation} These are two-sided $\HCW$-submodules of $\HC$ by Proposition~\ref{p_HCW_bimodule}.

For $\lambda\in Y^+$, let $\pi_\lambda:\cH_{ \leq^{++}\lambda}\twoheadrightarrow \cH_{ \leq^{++}\lambda}/\cH_{<^{++}\lambda}$\index{p@$\pi_\lambda$} be the natural projection. This is a morphism of $\HCW$-bi-modules by Proposition~\ref{p_HCW_bimodule}.

\begin{Lemma}\label{l_pi_lambda_Z_w_lambda}
Let $\tilde{\lambda}\in Y^{++}$ and $w\in W_0^{\tilde{\lambda}}$. Then we have: \[
    \pi_{\tilde{\lambda}}(Z^{w({\tilde{\lambda}})})=\pi_{\tilde{\lambda}}(q^{\ell(w)}T_{w^{-1}}^{-1}Z^{\tilde{\lambda}} T_w^{-1})=\pi_{\tilde{\lambda}}(q^{\ell(w)}\delta^{-1/2}({\tilde{\lambda}})T_{w^{-1}}^{-1}T_w^{-1} T_{w({\tilde{\lambda}})}).\]
\end{Lemma}

\begin{proof}
    
Let $w\in W_0^{\tilde{\lambda}}$. Assume that $w\neq 1$ and take $i\in {I_A}$ such that $v<w$, where $v=r_iw$. Then by Lemma~\ref{l_Wlambda}, $v\in W_0^{\tilde{\lambda}}$.  By Lemma~\ref{lemGrowth_dominant_height_segments}, \eqref{e_Z_lambda_Tinverse}  and \eqref{e_computation_BL}, we have $\pi_{{\tilde{\lambda}}}(\Zb^{v({\tilde{\lambda}})} T_{i}^{-1})=\pi_{\tilde{\lambda}}(q^{-1}T_{i} \Zb^{w({\tilde{\lambda}})})$ and hence $\pi_{\tilde{\lambda}}(Z^{w({\tilde{\lambda}})})= \pi_{\tilde{\lambda}}(qT_{i}^{-1}\Zb^{v({\tilde{\lambda}})})T_i^{-1}$. By induction, we deduce that $
    \pi_{\tilde{\lambda}}(Z^{w({\tilde{\lambda}})})=\pi_{\tilde{\lambda}}(q^{\ell(w)}T_{w^{-1}}^{-1}Z^{\tilde{\lambda}} T_w^{-1}).$ By Theorem~\ref{t_BL_relations} and \cite[Corollary 4.3]{bardy2016iwahori} we deduce the result.
    \end{proof}

Recall the definition of $w_\lambda$ in \eqref{e_definition_w_lambda}.

\begin{Proposition}\label{P_BL_support_IM_basis}
Let $\lambda\in Y^+$.  Then:\begin{enumerate}
\item Write $\lambda=w(\lambda^{++})$, with $w\in W_0$. Then $\supp^{Z}(\Tb_{\lambda})\subset \overline{\sT_{w}}(\{\lambda^{++}\})\subset \conv_{Q^\vee}(W_0\cdot \lambda)\subset ]-\infty,\lambda]_{\leq^{++}}$ and  $\supp^T(Z^\lambda)\subset \conv_{Q^\vee}(W_0\cdot \lambda) \subset ]-\infty,\lambda]_{\leq^{++}}$.

\item We have $\supp^T(Z^\lambda)\cap W_0\cdot \lambda=\{\lambda\}$, $\supp^Z(T_\lambda)\cap W_0\cdot \lambda=\{\lambda\}$, \[\pr^Z_\lambda(T_\lambda)=\left(q^{-\ell(w_\lambda)}\delta^{1/2}(\lambda^{++})T_{w_\lambda}T_{w_\lambda^{-1}}\right)Z^\lambda\in (\HCW)^\times Z^\lambda\] and \[\pr^T_\lambda(Z^\lambda)=\left(q^{\ell(w_\lambda)}\delta^{-1/2}(\lambda^{++})T_{w_\lambda^{-1}}^{-1} T_{w_{\lambda}}^{-1} \right)T_\lambda\in (\HCW)^\times T_\lambda.\]

\end{enumerate} 
\end{Proposition}

\begin{proof}

By \cite[Corollary 4.3]{bardy2016iwahori}, we have $T_\lambda=T_w T_{\lambda^{++}} T_{w}^{-1}$.  By Theorem~\ref{t_BL_relations}, we have  $\Tb_{\lambda^{++}}\in \cR^\times \Zb^{\lambda^{++}}$ and then $\supp^Z(\Tb_{\lambda^{++}})=\{\lambda^{++}\}$. Therefore $\supp^Z(T_\lambda)=\supp^Z(T_{\lambda^{++}} T_{w}^{-1})=\supp^Z(Z^{\lambda^{++}} T_w^{-1})$. By Lemma~\ref{lemBernstein_Lusztig_relation_inverse} we deduce 1). By \eqref{e_inclusions_sT} and Lemma~\ref{lemInequality_convex_hull_orbit}, we deduce that $\supp^Z(T_\lambda)\subset \conv_{Q^\vee}(W_0\cdot \lambda)\subset ]-\infty,\lambda]_{\leq^{++}}$.

By Lemma~\ref{lemT_support_Z_convex_hull} and Lemma~\ref{lemInequality_convex_hull_orbit}, we have $\supp^T(Z^\lambda)\subset \conv_{Q^\vee}(W_0\cdot \lambda)\subset  ]-\infty,\lambda]_{\leq^{++}}$, which proves (1). Then (2) follows from  Lemma~\ref{l_pi_lambda_Z_w_lambda}.
\end{proof}

We have \begin{equation}\label{e_injectivity_product}
    \forall h\in \HCW, \forall \lambda\in Y^+, hZ^\lambda=0\Rightarrow h=0.
\end{equation} Indeed in $^{BL}\cH$ (see Definition~\ref{d_BL_algebra}), we have $hZ^\lambda Z^{-\lambda}=h=0$ and thus $h=0$. We now prove a similar property for $T$. 

\begin{Corollary}\label{c_support_Tlambda_Tmu}
Let $\lambda,\mu\in Y^+$ and $h\in \HCW$. Then $\supp^T(T_\lambda h T_\mu)\subset \,\,]-\infty,\lambda^{++}+\mu^{++}]_{\leq^{++}}$. 
\end{Corollary}

\begin{proof}
By Proposition~\ref{P_BL_support_IM_basis}, we can write \[T_\lambda=\sum_{\nu\in ]-\infty,\lambda^{++}]_{\leq^{++}}} h_\nu Z^\nu\text{ and }T_\mu=\sum_{\tau\in ]-\infty,\mu^{++}]_{\leq^{++}}}h_\tau' Z^\tau,\] with $(h_\nu),(h'_\tau)\in (\HCW)^{(Y^+)}$. 
By Lemma~\ref{lemBernstein_Lusztig_relation_inverse}, $\supp^Z(Z^\nu hh_\tau')\subset ]-\infty,\nu^{++}]_{\leq^{++}}$ for $\nu,\tau \in Y^+$.
Using \cite[Lemma 4.17 (1)]{abdellatif2019completed} we deduce $\supp^Z(T_\lambda hT_\mu)\subset ]-\infty,\lambda^{++}+\mu^{++}]_{\leq^{++}}$ and thus   by Proposition~\ref{P_BL_support_IM_basis}, $\supp(T_\lambda hT_\mu)\subset ]-\infty,\lambda^{++}+\mu^{++}]_{\leq^{++}}$. 
\end{proof}

\subsection{Equivalence of two finiteness properties}\label{ss_equivalence_finitenesss}

We want to prove that every $T$-series $\sum_{\lambda\in Y^+} h_\lambda T_\lambda$ (resp. $Z$-series $\sum_{\lambda\in Y^+} h_{\lambda} Z_\lambda$), with $(h_\lambda)\in (\HCW)^{Y^+}$ having Weyl almost finite support can be decomposed as a $Z$-series $\sum_{\lambda\in Y^+} h_\lambda Z_\lambda$ (resp. $\sum_{\lambda\in Y^+} h_\lambda T_\lambda$). To that end, we need to prove that for $\tilde{\lambda}\in Y^{++},\mu\in Y^+$, $\{\lambda\in W_0\cdot \tilde{\lambda}
\mid \mu\in \supp^Z(T_\lambda)\}$  (resp. $\{\lambda\in W_0\cdot \tilde{\lambda}
\mid \mu\in \supp^T(Z_\lambda)\}$) is finite. We prove that these two properties are actually equivalent, see Proposition~\ref{P_equivalence_finiteness}. Thus we only need to prove one of them, which we do in Lemma~\ref{lemFiniteness_support_TWhat}.

\begin{Proposition}\label{P_equivalence_finiteness}
The following two properties are equivalent:\begin{enumerate}
\item For every $\tilde{\lambda}\in Y^{++}$, $\mu\in Y^+$,  the set $\{\lambda\in W_0\cdot \tilde{\lambda}\mid \mu\in \supp^Z(\Tb_\lambda)\}$ is finite,

\item For every $\tilde{\lambda}\in Y^{++}$, $\mu\in Y^+$, the set $\{\lambda\in W_0 \cdot \tilde{\lambda}\mid \mu\in \supp^T(\Zb^\lambda)\}$ is finite.
\end{enumerate}
\end{Proposition}

\begin{proof}
For $\mu\in Y^+$, set $S^Z(T,\mu)=\{\lambda\in Y^+\mid \mu\in \supp^Z(T_\lambda)\}$ and $S^T(Z,\mu)=\{\lambda\in Y^+\mid \mu\in \supp^T(Z^\lambda)\}$. 

    Let $\mu\in Y^+$. By Proposition~\ref{P_BL_support_IM_basis}, \begin{equation}\label{e_sets_E}
        S^Z(T,\mu)\cup S^T(Z,\mu)\subset \{\mu\}\cup ]\mu,+\infty[_{\leq^{++}}.
    \end{equation} 
    
 Let $\lambda,\mu\in Y^{+}$. Then by Proposition~\ref{P_BL_support_IM_basis}, we can write \[Z^\lambda=hT_\lambda+\sum_{\nu\in \supp^T(Z^\lambda)\setminus \{\lambda\}}h_\nu T_\nu,\] with $h\in (\HCW)^\times$ and $h_\nu\in \HCW$ for all $\nu\in \supp^T(Z^\lambda)\setminus \{\lambda\}$. Hence \[T_\lambda=h^{-1} Z^\lambda-\sum_{\nu\in \supp^T(Z_\lambda)\setminus\{\lambda\}} h^{-1}h_\nu T_\nu.\] Therefore \begin{align*}\lambda\in S^T(Z,\mu)&\Rightarrow \mu\in \supp^Z(Z^\lambda)\cup \bigcup_{\nu\in \supp^T(Z^\lambda)\setminus \{\lambda\}}\supp^Z(T_\nu)=\{\lambda\}\cup \bigcup_{\nu\in \supp^T(Z^\lambda)\setminus \{\lambda\}}\supp^Z(T_\nu)\\
    &\Rightarrow  \exists \nu\in S^Z(T,\mu)\mid \lambda\in S^T(Z,\nu).
    \end{align*} Using \eqref{e_sets_E} we deduce \begin{equation}\label{e_recursive_description_E}
        S^T(Z,\mu)\subset \bigcup_{\nu\in S^Z(T,\mu)}S^T(Z,\nu)\subset S^Z(T,\mu)\cup \bigcup_{\nu\in S^Z(T,\mu)\cap ]\mu,+\infty[_{\leq^{++}}}S^T(Z,\nu).
    \end{equation}

We assume (1).     For $n\in \N$, set $\sP(n):$ ``for all $\mu\in Y^+$, for all $\tilde{\lambda}\in Y^{++}\cap [\mu,+\infty[_{\leq^{++}}$, if $\htt(\tilde{\lambda}-\mu^{++})\leq n$, then $S^T(Z,\mu)\cap W_0\cdot \tilde{\lambda}$ is finite.'' 

If $\tilde{\lambda}\in Y^{++}$ and $\mu\in W_0\cdot \tilde{\lambda}$, then according to \eqref{e_sets_E} we have $S^T(Z,\mu)\cap W_0\cdot \tilde{\lambda}=\{\mu\}$ and thus $\sP(0)$ is true. Let $n\in \N$ be such that $\sP(n)$ is true. Let $\tilde{\lambda}\in Y^{++}$ and $\mu\in Y^+$ be such that $\htt(\tilde{\lambda}-\mu)\geq 0$.  Using \eqref{e_recursive_description_E} we have: \begin{equation}\label{e_recursive_description_E_orbit}    
S^T(Z,\mu)\cap W_0\cdot \tilde{\lambda}\subset \left(S^Z(T,\mu)\cap W_0\cdot \tilde{\lambda}\right) \cup \bigcup_{\nu\in S^Z(T,\mu)\cap ]\mu,\tilde{\lambda}[_{\leq^{++}}} \left(S^T(Z,\nu)\cap W_0\cdot \tilde{\lambda}\right).\end{equation}
    
    Now  $]\mu,\tilde{\lambda}[_{\leq^{++}}=\bigcup_{\tilde{\nu}\in Y^{++}\cap ]\mu^{++},\tilde{\lambda}[_{\leq^\vee}} W_0\cdot \tilde\nu$ and thus $]\mu,\tilde{\lambda}[_{\leq^{++}}$ is  a finite union of $W_0$-orbits. Thus by (1), $]\mu,\tilde{\lambda}[_{\leq^{++}}\cap S^Z(T,\mu)$ is finite.  Moreover if $\nu\in ]\mu,\tilde{\lambda}[_{\leq^{++}}$ we have $\htt(\tilde{\lambda}-\nu^{++})\leq n$ and thus $S^T(Z,\nu)\cap W_0\cdot \tilde{\lambda}$ is finite for every $\nu\in ]\mu,\tilde{\lambda}[_{\leq^{++}}$, by $\sP(n)$. Therefore the right hand side of \eqref{e_recursive_description_E_orbit} is finite and hence $S^T(Z,\mu)\cap W_0\cdot \tilde{\lambda}$ is finite. Consequently $\sP(n+1)$ is true and hence (2) is true. We proved that (1) implies (2). To obtain the other sense, it suffices to switch $Z$ and $T$ in the proof above.
\end{proof}

\section{Completed Iwahori-Hecke algebra:  Bernstein-Lusztig presentation}\label{s_Z_presentation}

When $\cH$ is associated with a reductive group and $W_0$ is finite, the center of $\cH$ is isomorphic to $\cR[Y]^{W_0}$ (see \cite[Theorem 8.1]{lusztig1983singularities}). By the Satake isomorphism, it is therefore isomorphic to the spherical Hecke algebra.  When $\cH$ is associated with a non-reductive Kac-Moody group, the center is still isomorphic to  $\cR[Y]^{W_0}$, but the invariant ring for the infinite group $W_0$ is very small and can even be trivial (see \cite[Lemma 4.31]{abdellatif2019completed}). In particular, it is not isomorphic to the spherical Hecke algebra. 

We introduce here a new \emph{completed Iwahori-Hecke algebra} $\HBL$. We prove can be interpretted ast a set of $I$-biinvariant functions on $G^{\geq 0}$ (Theorem~\ref{c_T_version_completed_algebra}) satisfying an explicit support condition. This algebra has a large center isomorphic to Looijenga's invariant ring and therefore the spherical Hecke algebra. In this section, we define and study $\HBL$ using the Bernstein-Lusztig generators. We introduce \emph{Weyl almost finite sets} in \S\ref{ss_WAF_sets}.  We define $\HBL$ and prove that it is an algebra in \S\ref{ss_definition_BL_completed} and we determine its center in \S\ref{ss_center}.

\subsection{Weyl almost finite sets and Looijenga algebra}\label{ss_WAF_sets}

In order to define the support condition of the elements of $\HBL$, we use the notion of ``Weyl almost finite sets'', that we introduce here.

Recall that if $x\in \cT$, we denote by $x^{++}$ the unique element of $W_0\cdot x\cap \cT$. If $E$ is a subset of $Y^+$, we denote by $E^{++}$ the set $\{\lambda^{++}\mid \lambda\in E\}$.

\begin{Definition}\label{d_Weyl_almost_finiteness}
  Let $E\subset Y$. We call $E$ \textbf{almost finite} if there exists a finite set $J\subset Y$ such that $E\subset J-Q^\vee_+$
 
  We say that  $E$ is \textbf{Weyl almost finite} if $W_0\cdot E=\{w(\lambda)\mid w\in W_0,\lambda\in E\}$ is almost finite. 
\end{Definition}

By \eqref{e_AH4.8}, a Weyl almost finite subset of $Y$  is necessarily contained in $Y^+$. By  \eqref{e_GR2.4}, a subset $E$ of $Y^+$ is Weyl almost finite if and only if $E^{++}$ is almost finite.

\begin{Proposition}\label{p_Weyl_almost_finite_sets}
    Let $E\subset Y^+$ be non-empty. Then $E$ is Weyl almost finite if and only the following two conditions are satisfied:\begin{enumerate}
        \item the set $M$ of maximal elements (for $\leq_{Q^\vee}$) of $W_0\cdot E$ is finite and contained in $E^{++}$,

        \item $W_0\cdot E\subset \bigcup_{\nu\in M} \nu-Q^\vee_+$. 
    \end{enumerate}
\end{Proposition}

\begin{proof}
    If a set $E$ satisfies (1) and (2), then it is clearly Weyl almost finite. Conversely, let $E$ be a Weyl almost finite subset of $Y$. Let $E'=W_0 \cdot E$. Then $E'$ is Weyl almost finite. By \cite[Lemma 4.4]{abdellatif2019completed} applied with $F=E^{++}\subset E'$, there exists a finite subset $J$ of $E^{++}$ such that $E^{++}\subset \bigcup_{j\in J} j-Q^\vee_+$. By \eqref{e_GR2.4}   and by definition, the set $M$ of maximal elements of $W_0\cdot E$ is contained in $J$ and we can replace $J$ by $M$.   
    By \eqref{e_GR2.4},  we deduce $W_0\cdot E=W_0\cdot E^{++}\subset \bigcup_{\nu\in F} \nu-Q^\vee_+$ and we get the proposition.
\end{proof}

\begin{Example}\label{ex_Weyl_almost_finite_sets} \begin{enumerate}
\item Let $\lambda\in Y$. Then $\{\lambda\}$ is almost finite. By \cite[Lemma 4.8]{abdellatif2019completed}, $\{\lambda\}$ is Weyl almost finite if and only if $\lambda\in Y^+$. 

\item Suppose that $Y$ is associated with an affine Kac-Moody matrix. In this case, there exists $\delta\in X$ (the dual of $Y$) such that $Y^+=\bigcap_{i\in {I_A}} \ker(\alpha_i)\cup \delta^{-1}(\R_{>0})$ and $\delta(\alpha_i^\vee)=0$, for all $i\in {I_A}$, by \cite[Proposition 5.8 b)]{kac1994infinite}. Let $\lambda\in Y^{+}$ and $i\in {I_A}$. Then $\lambda-{\Z_{\geq 0}} \alpha_i^\vee$ is almost finite but $r_i\cdot (\lambda-{\Z_{\geq 0}} \alpha_i^\vee)=r_i(\lambda)+{\Z_{\geq 0}} \alpha_i^\vee$ is not almost finite. Thus $\lambda-{\Z_{\geq 0}} \alpha_i^\vee$ is an almost finite subset of $Y^+$ which is not Weyl almost finite. 

\item Suppose that $Y$ is associated with an indefinite size $2$ Kac-Moody matrix. By \cite[4.2.2]{abdellatif2019completed}, every subset of $Y^+$ is almost finite and thus the Weyl almost finite subsets of $Y$ are exactly the subsets of $Y^+$.

\item Suppose that $Y$ is associated with a Cartan matrix. Then by \cite[Lemma 3.16]{abdellatif2022erratum}, the Weyl almost finite subsets of $Y$ are exactly the finite subsets of $Y$ (since Weyl almost finite subsets of $Y$ are in particular $W_0$-almost finite in the sense of \cite[Definition 3.1]{abdellatif2022erratum}). 

\end{enumerate}
\end{Example}

\begin{Lemma}\label{l_sum_WAF_sets}
    Let $S_1,S_2$  be two Weyl almost finite subsets of $Y^+$. Then $S_1+S_2$ is Weyl almost finite.
\end{Lemma}

\begin{proof}
   Let $S=W_0\cdot S_1+W_0\cdot S_2$. As $S_1+S_2\subset S$, it suffices to prove that $S$ is Weyl almost finite. Let $M_1$ and $M_2$ be the set of maximal elements of $S_1^{++}$ and $ S_2^{++}$ respectively. Let $\lambda\in S$. Write $\lambda=w_1(\lambda_1)+w_2(\lambda_2)$, with $w_1,w_2\in W_0$ and $\lambda_1\in S_1,\lambda_2\in  S_2$. Then by Proposition~\ref{p_Weyl_almost_finite_sets}, there exist $m_1\in M_1$ and $m_2\in M_2$ such that $w_1(\lambda_1)\leq_{Q^\vee}m_1$ and $w_2(\lambda_2)\leq_{Q^\vee}m_2$. Therefore by \cite[Lemma 4.17 (1)]{abdellatif2019completed}, $\lambda=w_1(\lambda_1)+w_2(\lambda_2)\leq_{Q^\vee} m_1+m_2$ and hence $S$ is Weyl almost finite
\end{proof}

\begin{Definition}\label{d_Looijenga_algebra}

Let $\cR[\![Y]\!]$ be the \textbf{Looijenga algebra over $\cR$}, i.e $\RY$\index{r@$\RY$} is the set of infinite series $\sum_{\lambda\in Y} a_\lambda Z^\lambda$ such that $(a_\lambda)\in \cR^{Y}$ has   almost finite support. This is an algebra for the product \[\sum_{\lambda\in Y}a_\lambda Z^\lambda.\sum_{\mu\in Y} b_\mu Z^\mu=\sum_{\nu\in Y} \left(\sum_{\lambda,\mu\in Y\mid \lambda+\mu=\nu} a_\lambda b_\mu \right)Z^\nu,\] for $a=\sum_{\lambda\in Y} a_\lambda Z^\lambda$, $b=\sum_{\mu\in Y} b_\mu Z^\mu\in \cR[\![Y]\!]$. This product is well defined since for all $\nu\in Y$, there exist only finitely elements $(\lambda,\mu)$ of $\supp(a)\times \supp(b)$ such that $\lambda+\mu=\nu$, by the almost finiteness condition.

  Let $\RY$ be the subalgebra of $\cR[\![Y]\!]$ consisting of the elements having almost finite support.  Let   $\RY_{\WAF}$\index{r@$\RY_{\WAF}$} be the subset of $\RY$ consisting of the elements having Weyl almost finite support. It is a subalgebra of $\RY$ by Lemma~\ref{l_sum_WAF_sets}.

  The group $W_0$ acts on $\RY_{\WAF}$ via $w\cdot \sum_{\lambda\in Y^+} c_\lambda Z^\lambda=\sum_{\lambda\in Y^+} c_\lambda Z^{w(\lambda)}$, for all $\sum_{\lambda\in Y^{+}} c_\lambda Z^\lambda\in \RY_{\WAF}$ and $w\in W_0$. We denote by $\RY^{W_0}$\index{r@$\RY^{W_0}$} the subalgebra of $\RY_{\WAF}$ consisting of the elements of $\RY_{\WAF}$ which  are invariant under the action of $W_0$. We call this ring \textbf{Looijenga's invariant ring}.
\end{Definition}

\subsection{Definition of the completed Iwahori-Hecke algebra via the $Z$-basis }\label{ss_definition_BL_completed}

We now define a new completed Iwahori-Hecke algebra $\HBL$ different from the completed algebra $\widetilde{\cH}$ constructed in \cite{abdellatif2019completed} and \cite{abdellatif2022erratum}. 

We denote by $(\HCW)^{Y^+}_{\WAF}$ the set of families $(h_\lambda)_{\lambda\in Y^+}$ whose support $\{\lambda\in Y^+\mid h_\lambda\neq 0\}$ is Weyl almost finite. Let $\HBL$  be the set of formal series of the form $\sum_{\lambda\in Y^+} h_\lambda *\Zb^{\lambda}$, for $(h_\lambda)\in (\HC_{W_0})^{Y^+}_\WAF$. 

 For $\lambda\in Y^+$ and $x=\sum_{\mu\in Y^+} h_\mu \Zb^\mu\in \HBL$ we set $\coeff^Z_\lambda(x)=h_\lambda$ and $\supp^Z(x)=\{\mu\in Y^+\mid h_\mu\neq 0\}$.

\begin{Remark}\label{r_non_comparability}

When $\cS$ is associated with a Cartan matrix, then $\cH=\HBL=\widetilde{\cH}$, by Example~\ref{ex_Weyl_almost_finite_sets} (4)  and \cite[Lemma 3.16]{abdellatif2022erratum}. 

\end{Remark}

\begin{Definition}\label{d_Summable_family}
Let $J$ be a set. A family $(a_j)_{j\in J}\in (\HBL)^J$ is {\it summable} when the following properties hold.
\begin{itemize}
\item For all $\lambda \in Y^{+}$, the set $\{j \in J \mid \coeff_\lambda^Z(a_{j})\neq 0\}$ is finite.
\item $\displaystyle \bigcup_{j \in J} \supp^Z(a_{j})$ is Weyl almost finite.
\end{itemize}
If $(a_{j})_{j \in J} \in (\HBL)^{J}$ is a summable family, we define $\displaystyle \sum_{j \in J} a_{j} \in \HBL$ by 
\[\displaystyle \sum_{j \in J} a_{j} = \sum_{\lambda \in Y^{+}} h_{\lambda} Z^{\lambda} \ , \text{ with } h_\lambda = \displaystyle \sum_{j \in J} \coeff^Z_\lambda(a_{j})\in \HCW \text{ for all } \lambda\in Y^{+} .\] 
\end{Definition}

By Lemma~\ref{lemBernstein_Lusztig_relation_inverse}, we have the following lemma.

\begin{Lemma}\label{lemCommutation_relation} 
Let $\lambda\in Y^+$ and $h\in \HC_{W_0}$. Then \[\supp^Z(\Zb^\lambda h)\subset \bigcup_{u\in \supp(h)} \widetilde{\sT_{u^{-1}}}(\lambda)\subset \bigcup_{u\in \supp(h)} R_{u^{-1}}(\lambda),\] where $\supp(h)=\{w\in W_0\mid a_w\neq 0\}$, if $h=\sum_{w\in W_0}a_w T_w$. 
\end{Lemma}

\begin{Lemma}\label{l_summable_families}
    Let $S$ be a Weyl almost finite subset of $Y^+$. Let $\nu\in Y^+$. Then: \[E_2=E_2(\nu)=\{\mu\in S\mid \exists \lambda\ni S,w\in  W_0, \nu\in \supp^Z(Z^\lambda T_w Z^\mu)\}\] is finite. Choose a finite subset $W(\mu)$ of $W_0$, for every $\mu\in S$. Then 
    \[E=\{(\lambda,\mu)\in S^2\mid \exists w\in W(\mu), \nu\in \supp^Z(Z^\lambda T_w Z^\mu)\}\] is finite. Moreover   if $S'=\{\nu\in Y^+\mid E_2(\nu)\neq \emptyset\}$, then $S'$ is Weyl almost finite. 
\end{Lemma}

\begin{proof}
Let $M\subset Y^{++}$ be the set of maximal elements of $S^{++}$.  Let $\nu\in Y^+$ and $\mu\in E_2=E_2(\nu)$.  Then by Lemma~\ref{lemCommutation_relation}, there exist $\lambda\in S$, $w\in W_0$  and $\lambda'\in R_{w^{-1}}(\lambda)$ such that \begin{equation}\label{e_relation_lambda'_mu_nu}
    \lambda'+\mu=\nu.
\end{equation} By \cite[Lemma 4.17 (2)]{abdellatif2019completed} and Proposition~\ref{p_Weyl_almost_finite_sets}, $\lambda'\leq_{Q^\vee}\lambda^{++}\leq_{Q^\vee} m_1$ for some $m_1\in M$. Let $m_2\in M$ be such that $\mu\leq_{Q^\vee}m_2$.  We have $m_2\geq_{Q^\vee} \mu=\nu-\lambda'\geq_{Q^\vee}\nu-m_1$ and hence $E_2\subset \bigcup_{m_1,m_2\in M}[\nu-m_1,m_2]_{\leq_{Q^\vee}}$, which is finite. 
    
   Let $W=\bigcup_{\mu\in E_2} W(\mu)$. Let $(\lambda,\mu)\in E$. Then $\mu\in E_2$ and $\nu\in \supp^Z(Z^\lambda T_w Z^\mu)$, for some $w\in W$. By Lemma~\ref{lemCommutation_relation}, there exists $\lambda'\in R_{w^{-1}}(\lambda)$ such that $\lambda'=\nu-\mu$. By \cite[Lemma 3.3]{abdellatif2022erratum} we deduce that $\lambda$ can take only finitely many values, which proves the finiteness of $E$. 

   Now let $\nu\in Y^+$ be such that $E_2(\nu)\neq \emptyset$. Then by \eqref{e_relation_lambda'_mu_nu} and \cite[Lemma 4.17 (2)]{abdellatif2019completed}, we have $\nu^{++}\leq_{Q^\vee}m_1+m_2$, for some $(m_1,m_2)\in M^2$ and thus $S'$ is Weyl almost finite.
\end{proof}

\begin{Theorem}\label{thmConvolution_summable_family}
Let $(a_j)_{j\in J}\in \HC^J$ and $(b_k)_{k\in K}\in \HC^K$ be two families. Assume that they are summable in $\HBL$. Then: \begin{enumerate}
\item $(a_j b_k)_{j\in J,k\in K}$ is  summable in $\HBL$,

\item $\sum_{j\in J,k\in  K}a_j b_k$ only depends on the two elements $\sum_{j\in J}a_j$ and $\sum_{k\in K} b_k$ of $\HBL$.
\end{enumerate} 
\end{Theorem}

\begin{proof}
    Let $S=\bigcup_{j\in J} \supp^Z(a_j)\cup \bigcup_{k\in K} \supp^Z(b_k)\subset Y^+$.  Then $S$ is Weyl almost finite by assumption. For $j\in J,k\in K$, write $a_j=\sum_{\lambda\in S} h_{j,\lambda} Z^\lambda$ and $b_k=\sum_{\mu\in S}\tilde{h}_{k,\mu} Z^\mu$, where $(h_{j,k}),(\tilde{h}_{k,\mu})\in (\cH_{W_0})^{(Y^+)}$.

    For $\mu\in Y^+$, set $W(\mu)=\bigcup_{k\in K}\supp(\tilde{h}_{k,\mu})\subset W_0$. Then $W(\mu)$ is finite.   Let $\nu\in Y^+$. Define $E_2=E_2(\nu)$ and $E$ be as  in Lemma~\ref{l_summable_families}.  Let $j\in J,k\in K$ and assume that 
    $\nu\in\supp^Z(a_j b_k)$. Then there exists $(\lambda,\mu)\in (Y^+)^2$ such that $\nu\in \supp^Z(h_{j,\lambda}Z^\lambda \tilde{h}_{k,\mu}Z^\mu)$. 
    Then $\mu\in E_2$ and there exists $w\in W(\mu)$ such that $\nu\in \supp^Z(Z^\lambda T_w Z^\mu)$. 
    Consequently there exists $(\lambda,\mu)\in E$ such that $h_{j,\lambda}\neq 0$ and $\tilde{h}_{k,\mu}\neq 0$. Therefore $j\in J_\nu,k\in K_\nu$, where  \[J_\nu=\{\tilde{j}\in J \mid \exists (\lambda,\mu)\in E, h_{\tilde{j},\lambda}\neq 0\}\text{ and }K_\nu=\{\tilde{k}\in K\mid \exists(\lambda,\mu)\in E,\tilde{h}_{\tilde{k},\mu}\neq 0\}.\] These sets are finite  since $E$ is finite and thanks to the summability assumption. Moreover \[\bigcup_{j\in J,k\in  K } \supp^Z(a_j b_k)\subset \{\nu\in Y^+\mid E_2(\nu)\neq \emptyset\},\] and the right-hand side is  Weyl almost finite by Lemma~\ref{l_summable_families}. Therefore $(a_j b_k)_{j\in J,k\in  K}$ is summable.

It remains to prove that $\sum_{j\in J,k\in K}a_j b_k$ only depends on $\sum_{j\in J}a_j$ and $\sum_{k\in K} b_k$.  For  $w\in W_0$ and $\lambda\in {Y^{+}}$, write  $\Zb^\lambda \Tb_w=\sum_{\lambda'\in {Y^{+}}} x_{\lambda',w}  \Zb^{\lambda'}$, where  $(x_{\lambda',w})\in (\HCW)^{({Y^{+}})}$.
 For $k\in K$, write $b_k=\sum_{w\in W_0,\mu\in Y^+}b_{k,w,\mu} \Tb_w \Zb^{\mu}$, where $b_{k,w,\mu}\in \cR$ for all $w\in W_0$ and $\mu\in Y^+$. Let $\nu\in Y^+$ and $j\in J,k\in K$. Then \[\coeff^Z_\nu(a_j b_k)=\sum_{\lambda,\lambda',\mu\in Y^+\mid\lambda'+\mu=\nu,w\in W_0}b_{k,w,\mu}h_{j,\lambda} x_{\lambda',w}.\] Thus
 \[\begin{aligned}\coeff^Z_\nu(\sum_{(j,k)\in J\times K}a_j b_k)&=\sum_{(j,k)\in J_\nu\times K_\nu}\coeff^Z_\nu(a_j b_k) \\
 &=\sum_{(j,k)\in J_\nu\times K_\nu}\sum_{(\lambda,\mu)\in E,\lambda'\in {Y^{+}}\mid\lambda'+\mu=\nu,w\in W_0}b_{k,w,\mu}h_{j,\lambda} x_{\lambda',w}\\
 &=\sum_{(\lambda,\mu)\in E,\lambda'\in {Y^{+}}\mid\lambda'+\mu=\nu,w\in W_0}\sum_{k\in K_\nu}b_{k,w,\mu}\sum_{j\in J_\nu}h_{j,\lambda} x_{\lambda',w}\\
 &=\sum_{(\lambda,\mu)\in E,\lambda'\in {Y^{+}},w\in W_0}\coeff^{\Tb_w,Z}_{\mu}(\sum_{k\in K} b_k)\coeff^Z_\lambda(\sum_{j\in J} a_j) x_{\lambda',w},\end{aligned}\] where $\coeff_{\mu}^{\Tb_w,Z}$ is the coefficient in front of $\Tb_w \Zb^\lambda$ in the $\cR$-basis $(\Tb_u \Zb^\nu)_{u\in W_0,\nu\in Y^+}$ of $\cH$, i.e $\coeff_{\mu}^{\Tb_w,Z}(\sum_{\tau\in Y^+} h'_\tau Z^\tau)=c_{w,\lambda}$ if $\sum_{\tau\in Y^+} h_\tau' Z^\lambda\in \HBL$, with $h_\lambda'=\sum_{v\in W_0}c_{v,\lambda}T_v\in \HCW$. This proves (2) and completes the proof of the theorem. 
\end{proof}

\begin{Definition}\label{d_convolution_product}
Recall that  $\HBL$ is the set of formal series of the form $\sum_{\lambda\in Y^+} h_\lambda Z^\lambda$, with $(h_\lambda)\in (\HCW)^{Y^+}_{\WAF}$. We equip $\HBL$ with a convolution $ $ by setting \[(\sum_{\lambda\in Y^+}h_\lambda \Zb^{\lambda}) (\sum_{\mu\in Y^+} \tilde{h}_\mu \Zb^\mu)=\sum_{\lambda,\mu\in Y^+} \left((h_\lambda \Zb^\lambda) (\tilde{h}_\mu \Zb^\mu)\right),\] for all $\sum_{\lambda\in Y^+} h_\lambda \Zb^\lambda,\sum_{\mu\in Y^+} \tilde{h}_\mu Z^\mu\in \HBL$. This is well defined by Theorem~\ref{thmConvolution_summable_family} applied with the families $(h_\lambda Z^\lambda)_{\lambda\in Y^+}$, $(\tilde{h}_\mu Z^\mu)_{\mu\in Y^+}$. This equips $\HBL$ with the structure of an associative algebra over $\cR$ (the associativity of $ $ is a consequences of the associativity of $ $ on  $\HC$ and of Theorem~\ref{thmConvolution_summable_family}).

\end{Definition}

\begin{Proposition}\label{propDescription_HBL_H_basis}

Let $(\theta_w)_{w\in W_0}\in (\RY_{\WAF})^{W_0}$ be a family satisfying the following two conditions:
\begin{enumerate}
\item $\bigcup_{w\in W_0}\supp^Z(\theta_w)$ is Weyl almost finite,

\item for all $\lambda\in Y^+$, $\{w\in W_0\mid\lambda\in \supp^Z(\theta_w)\}$ is finite.
\end{enumerate} 
\end{Proposition}

Then $(\Tb_w \theta_w)_{w\in W_0}$ is summable in $\HBL$. Moreover, for every $a\in \HBL$, there exists a unique $(\theta_w)_{w\in W_0}\in (\RY_{\WAF})^{W_0}$ satisfying the conditions above and such that $a=\sum_{w\in W_0} \Tb_w \theta_w$.
\begin{proof}
Let $(\theta_w)_{w\in W_0}\in (\RY_{\WAF})^{W_0}$ be a family satisfying (1) and (2). Let $\lambda\in Y^+$ and $w\in W_0$. Then $\coeff_\lambda^Z(\Tb_w \theta_w)\neq 0$ if and only if $\lambda\in \supp^Z(\theta_w)$. Therefore 
$\{w\in W_0\mid\coeff_\lambda^Z(\Tb_w \theta_w)\neq 0\}$ is finite. Moreover $\bigcup_{w\in W_0} \supp^Z(\Tb_w \theta_w)=\bigcup_{w\in W_0} \supp^{Z}(\theta_w)$ is Weyl almost finite by (1), which proves that $(\Tb_w \theta_w)$ is summable.

Let $a=\sum_{\lambda\in Y^+}h_\lambda \Zb^\lambda\in \HBL$. For $\lambda\in Y^+$, write $h_\lambda=\sum_{w\in W_0,\lambda\in Y^+} x_{w,\lambda} \Tb_w$, where $(x_{w,\lambda})_{w\in W_0}\in (\cR)^{(Y^+)}$. For $w\in W_0$, set  $\theta_w=\sum_{\lambda\in Y^+} x_{w,\lambda} \Zb^\lambda$. Then $\supp^Z(\theta_w)\subset \{\lambda\in Y^+\mid\ h_\lambda\neq 0\}$ and thus $\bigcup_{w\in W_0} \supp^Z (\theta_w)$  is Weyl almost finite. In particular, $\theta_w$ is a well defined element of $\RY_{\WAF}$. Let $\lambda\in Y^+$. Then $\coeff_\lambda^Z(\Tb_w \theta_w)= x_{w,\lambda} \Tb_w$. Moreover, $\lambda\in \supp^Z(\theta_w)$ if and only if  $x_{w,\lambda}\neq 0$ if and only if $w\in \supp(h_\lambda)$, hence $\{w\in W_0\mid\lambda\in \supp^Z(\theta_w)\}$ is finite, which proves that $(\theta_w)$ satisfies (1) and (2). As $\coeff_\lambda^Z(\sum_{w\in W_0} \Tb_w \theta_w)=\sum_{w\in W_0} x_{w,\lambda} \Tb_w=h_\lambda$, we deduce that $a=\sum_{w\in W_0} \Tb_w \theta_w$.

Let $(\theta_w)\in (\HCW)^{W_0}$ be a family satisfying (1) and (2) and such that $\sum_{w\in W_0} \Tb_w \theta_w=0$. Let $\lambda\in Y^+$. Then \[\coeff_\lambda^Z(\sum_{w\in W_0} \Tb_w \theta_w)=\sum_{w\in W_0\mid \lambda\in \supp(\theta_w)}\coeff_\lambda^Z(\theta_w) \Tb_w=0\] and as $(\Tb_v)_{v\in W_0}$ is an $\cR$-basis of $\HCW$, we deduce that $\theta_w=0$, which completes the proof of the proposition.
\end{proof}

\begin{Definition}\label{d_coeff_w_Z}
Let $h\in \HBL$. Using Proposition~\ref{propDescription_HBL_H_basis}, we write $h=\sum_{w\in W_0} T_w\theta_w$, with $(\theta_w)\in (\RY_{\WAF})^{W_0}$. For $w\in W_0$, we set $\coeff^Z_w(h)=\theta_w$\index{c@$\coeff^Z_w$}, which is well defined by the uniqueness in Proposition~\ref{propDescription_HBL_H_basis}.
    
\end{Definition}

\subsection{Center of $\HBL$}\label{ss_center}

We now prove that the center of $\HBL$ is $\RY^{W_0}$ (where $\RY^{W_0}$ is defined in Definition~\ref{d_Looijenga_algebra}).

In order to use a result from \cite{abdellatif2019completed}, we use the following remark. 
\begin{Remark}\label{r_Rw}
 In \cite[4.4]{abdellatif2019completed}, a set $R_w(E)$ is defined, for $w\in W_0$ and $E\subset Y$. It is defined with the same process as in Definition~\ref{d_sT}, with $R_i(E)=\conv_{Q^\vee}(E\cup r_i(E))$, for $i\in {I_A}$. If $E\subset Y$,  then $\widehat{\sT_i}(\lambda)\subset R_i(E)$, for all $\lambda\in E$, and thus $\widehat{\sT_w}(E)\subset R_w(E)$ for every $w\in W_0$. 
 \end{Remark}

\begin{Lemma}\label{lemCommutation_relation_Zlambda_Hw}
Let $\lambda\in Y$ and $w\in W_0$. Then there exists $(\theta_{v,w,\lambda})\in (\cR[Y^+])^{[1,w]}$ (where $[1,w]=\{v\in W_0\mid v\leq w\}$) such that \[\Zb^\lambda \Tb_w=\sum_{v\in [1,w]} \Tb_v \theta_{v,w,\lambda}\] and \begin{enumerate}
\item $\theta_{v,w,\lambda}\in \cR[Y]$, for all $v\in[1,w]$,

\item $\theta_{w,w,\lambda}=\Zb^{w^{-1}(\lambda)}$,

\item for all $v\in [1,w]$, we have $ \supp^Z(\theta_{v,w,\lambda})\subset \widetilde{\sT_{w^{-1}}}(\lambda)$. 

\end{enumerate}
\end{Lemma}

\begin{proof}
By   \cite[Lemma 2.8]{hebert2022principal} (here $\cR$ is  not necessarily a field, but this does not change the proof) and \eqref{eqSupport_Zlambda_Twinverse}, we have  \[\Zb^\lambda \Tb_w-\Tb_w \Zb^{w^{-1}(\lambda)}\in \left(\bigoplus_{v<w} \Tb_w \cR[Y]\right)\cap \left(\bigoplus_{\mu\in \widetilde{\sT_{w^{-1}}}(\lambda)} \HCW Z^\mu\right),\] which proves the lemma.
\end{proof}

The following theorem  is similar to \cite[Theorem 4.30]{abdellatif2019completed}. Our proof is however different, since we do not know how to define $Z^{-\mu} h Z^\mu$, for $\mu\in Y$ and $h\in \HBL$, contrary to what is done in \cite[4.5]{abdellatif2019completed}.

\begin{Theorem}\label{t_center}
The center $\sZ(\HBL)$ of $\HBL$ is $\RY^{W_0}$.
\end{Theorem}

\begin{proof}
Let $z\in \sZ(\HBL)\setminus\{0\}$. Using Proposition~\ref{propDescription_HBL_H_basis}, we write $z=\sum_{w\in W_0} \Tb_w \theta_w$, where $\theta_w\in\RY_{\WAF}$, for all $w\in W_0$. For $w\in W_0$, set $S_w=\bigcup_{v\in W_0\mid v\geq w} \supp (\theta_w)\subset Y^+$.  Let $w\in W_0$ be such that $\theta_w\neq 0$. Let $\mu$ be a maximal element of $S_w$ for $\leq_{Q^\vee}$. By Proposition~\ref{propDescription_HBL_H_basis}, $\{v\in W_0\mid v\geq w\text{ and }\mu\in \supp(\theta_v)\}$ is finite. We choose a maximal element $w'$ of this set (for the Bruhat order). 

Let $\lambda\in Y\cap C^v_f$.  Then by Theorem~\ref{thmConvolution_summable_family}, we have \begin{equation}\label{eqCommutation_z}\Zb^\lambda z=z \Zb^\lambda=\sum_{u\in W_0} \Tb_u \theta_u \Zb^\lambda.\end{equation}  On the other side, with the notation of Lemma~\ref{lemCommutation_relation_Zlambda_Hw}, \[\begin{aligned}\Zb^\lambda z &= \Zb^\lambda \sum_{u\in W_0} \Tb_u \theta_u\\ 
&= \sum_{u\in W_0}\sum_{u'\in [1,u]}\Tb_{u'} \theta_{u',u,\lambda}\theta_u\\
&= \sum_{u'\in W_0} \Tb_{u'} \sum_{u\geq u'}\theta_{u',u,\lambda}\theta_u.\end{aligned}\]

Therefore, with the notation of Definition~\ref{d_coeff_w_Z}, we have $\coeff^Z_{w'}(Z^\lambda z)=\sum_{w''\geq w'}\theta_{w',w'',\lambda}\theta_{w''}$. 

Using \eqref{eqCommutation_z}, we deduce:  \begin{equation}\label{e_lambda_mu_support_1}
    \lambda+\mu\in \supp (\sum_{w''\geq w'}\theta_{w',w'',\lambda}\theta_{w''}).
\end{equation}

Let $w''\in W_0$ be such that $w''\geq w'$. Let $\nu\in  \supp(\theta_{w',w'',\lambda}\theta_{w''})$. Then we can write $\nu=\lambda'+\mu'$, where $\lambda'\in \supp(\theta_{w',w'',\lambda})$ and $\mu'\in \supp(\theta_{w''})$.  Then by Lemma~\ref{lemCommutation_relation_Zlambda_Hw}, Remark~\ref{r_Rw} and \cite[Lemma 4.17 (2)]{abdellatif2019completed}, we have $\lambda'\leq_{Q^\vee}\lambda$. Suppose $\nu=\lambda+\mu$. Then $\mu'=(\lambda-\lambda')+\mu\geq_{Q^\vee} \mu$. By assumption on $\mu$, we deduce that $\mu'=\mu$. By assumption on $w'$ we deduce $w''=w'$. Thus we proved that for every $w\in ]w',\infty[$, we have: \[\lambda+\mu\notin \supp(\theta_{w',w'',\lambda}\theta_{w''}).\] Using \eqref{e_lambda_mu_support_1} we deduce that:
 \[\lambda+\mu\in \supp(\theta_{w',w',\lambda}\theta_{w'})=\supp(\Zb^{w'^{-1}(\lambda)}\theta_{w'})=w'^{-1}(\lambda)+\supp(\theta_{w'}).\] Thus there exists $\mu'\in \supp(\theta_{w'})$ such that $w'^{-1}(\lambda)+\mu'=\lambda+\mu$. Consequently, $\mu'=\lambda-w'^{-1}(\lambda)+\mu\geq_{Q^\vee}\mu$ and thus $\mu'=\mu$, which proves that $w'=1$, since $\lambda$ is regular. As $w\leq w'$ we deduce that $w=1$. Therefore $z\in \RY_{\WAF}$. 

Write $z=\sum_{\lambda\in Y^+} c_\lambda \Zb^\lambda$, where $(c_\lambda)\in \cR^Y$. Let $w\in W_0$. By Lemma~\ref{lemCommutation_relation_Zlambda_Hw} we have \begin{equation}\label{eqCommutation_z_T_w}
z \Tb_w=\sum_{\lambda\in Y^+} c_\lambda \Tb_w \Zb^{w^{-1}(\lambda)}+h'=\Tb_w z=\sum_{\lambda\in Y^+} c_\lambda \Tb_w \Zb^\lambda= \Tb_w\sum_{\lambda\in Y^+}  c_\lambda\Zb^\lambda,
\end{equation} for some $h'\in \bigoplus_{v\in [1,w[ } \Tb_v  \RY_\WAF$. Therefore  
\[\coeff^Z_w(zT_w)=\sum_{\lambda\in Y^+} c_\lambda \Zb^{w^{-1}(\lambda)}=\sum_{\lambda\in Y^+} c_\lambda \Zb^\lambda.\] Consequently $z\in \RY^{W_0}$. Hence $\sZ(\HBL)\subset \RY^{W_0}$.

Let now $z=\sum_{\lambda\in Y^+} a_\lambda\Zb^{\lambda}\in \RY^{W_0}$. Let $i\in {I_A}$. Write $z=x+y$, where $x=\sum_{\lambda\in Y^+\cap \ker(\alpha_i)} a_\lambda \Zb^\lambda$ and $y=\sum_{\lambda\in Y^+\mid\alpha_i(\lambda)>0} a_\lambda (\Zb^{\lambda}+\Zb^{r_i(\lambda)})$. Then $x$ and $y$ commute with $\Tb_i$ and thus $z$ commutes with $\Tb_i$. Consequently $z$ commutes with any element of $\HCW$. Let $b\in \HBL$. Write $b=\sum_{w\in W_0} \Tb_w \theta_w$, with $(\theta_w)\in (\RY_{\WAF})^{W_0}$ as in Proposition~\ref{propDescription_HBL_H_basis}. Then $(\Tb_w \theta_w)_{w\in W_0}$ is  summable in $\HBL$ and thus by Theorem~\ref{thmConvolution_summable_family}, $(z \Tb_w \theta_w)$ and $(\Tb_w \theta_w z)$ are summable in $\HBL$ and we have: 
 \[z (\sum_{w\in W_0} \Tb_w \theta_w)=\sum_{w\in W_0} (z \Tb_w \theta_w)=\sum_{w\in W_0} (\Tb_w \theta_w z)=(\sum_{w\in W_0} \Tb_w \theta_w) z,\] which proves that $\RY^{W_0}\subset\sZ(\HBL)$ and completes the proof of the theorem.

\end{proof}

\section{Finiteness results for changing bases between $Z$-basis and $T$-basis}\label{s_finiteness_results}

Recall that $(\HCW)^{Y^+}_{\WAF}$ is the set of families of $(\HCW)^{Y^+}$ whose support is Weyl almost finite. In Definition~\ref{d_sT}, we introduced, for every $w\in W_0$, a map $\overline{\sT_w}$ sending subsets of $Y^+$ on subsets of $Y^+$ and we proved that $\supp^Z(T_{w(\tilde{\lambda})})$ is contained in  $\overline{\sT_w}(\tilde{\lambda})$, for $\tilde{\lambda}\in Y^{++}$ and $w\in W_0$ (see Proposition~\ref{P_BL_support_IM_basis}). In this technical section, we prove  a finiteness result concerning the $\overline{\sT_w}$ (see Lemma~\ref{lemFiniteness_support_TWhat}). This enables us to express series of the form $\sum_{\lambda\in Y^+}h_\lambda T_\lambda$, with $(h_\lambda)\in (\HCW)^{Y^+}_{\WAF}$   as a series $\sum_{\lambda\in Y^+}h_\lambda' Z^\lambda$, with $(h_\lambda')\in (\HCW)^{Y^+}_{\WAF}$, and conversely (see Theorem~\ref{t_finiteness_coefficients}). This is the main technical achievement of this paper.

\subsection{Quotients of stabilizers on  a line segment}\label{ss_Quotient_stabilizers}

Recall that if  $E\subset \A$ is non-empty, we denote by $W_{0,E}$ its fixator.

In our proof of Lemma~\ref{lemFiniteness_support_TWhat}, we use the following property: if $\lambda,\lambda'\in Y^+$ and $\nu\in ]\lambda,\lambda'[_{Q^\vee}$, we have: \begin{equation}\label{e_expected_property}
   \text{if }E\subset W_0^\nu\text{ is finite, then }E\cdot \lambda\text{ is finite.} 
\end{equation} We prove this property in this subsection (see Corollary~\ref{c_finiteness_property_segments}). This leads us to studying the finiteness of $W_{0,\nu}/W_{0,[\lambda,\lambda']}$ for $\lambda,\lambda'\in \T$ and $\nu\in ]\lambda,\lambda'[$. 

Let $\lambda,\lambda'\in \T$, $\lambda\neq \lambda'$ and $\nu\in ]\lambda,\lambda'[$. In general,  $W_{0,\nu}$ has no reason to be smaller than $W_{0,\lambda}$ or $W_{0,\lambda'}$. For example if $i\in {I_A}$, $\lambda\in C^v_f$, $\lambda'\in r_i(C^v_f)$ and $\nu\in ]\lambda,\lambda'[$ is the element satisfying $\alpha_i(\nu)=0$, then $W_{0,\lambda}=W_{0,\lambda'}=\{1\}$ and $\{1,r_i\}\subset W_{0,\nu}$. The aim of this subsection is to prove that nevertheless, $W_{0,\nu}$ is not ``much bigger'' than $W_{0,\lambda}$ (or $W_{0,\lambda'}$) in the sense that it is contained in a finite union of translates of $W_{0,\lambda}$. Indeed, let $D=\lambda+\R(\lambda-\lambda')$. We have $W_{0,\lambda}\cap W_{0,\lambda'}=W_{0,D}$. We prove in Proposition~\ref{propFiniteness_quotients_stabilizers} that $W_{0,\nu}/W_{0,D}$ is finite and hence $W_{0,\nu}=\bigsqcup_{vW_{0,D}\in W_{0,\nu}/W_{0,D}} v W_{0,D}\subset  \bigcup_{vW_{0,D}\in W_{0,\nu}/W_{0,D}} v W_{0,\lambda}$.

\begin{Remark}
\begin{enumerate}
    \item In the case where $\cS$ is associated with an affine or a size $2$ Kac-Moody matrix, we can avoid this subsection, using the following reasoning. In this case, we have $\cT=\mathring{\cT}\sqcup \A_{in}$, where $\mathring{\cT}$ is the interior of the Tits cone and  $\A_{in}=\bigcap_{i\in {I_A}}\ker(\alpha_i)$ (indeed in this case, $W_0$ is the only infinite standard parabolic subgroup of $W_0$ and hence we have the claim by \cite[Proposition 3.12 f]{kac1994infinite}). Let $\lambda\in \cT$. If $\lambda\in \A_{in}$ then $W_0^{\lambda}=\{1\}$ and therefore \eqref{e_expected_property} is vacuous.  If $\lambda\in \mathring{\cT}$, then for $w\in W_0$ and  $\nu \in ]\lambda,w(\lambda)[$, $W_{0,\nu}$ is finite since $\mathring{\cT}$ is a convex cone. Therefore for every infinite subset $E$ of $W_0$, $E\cdot \nu$ is infinite, which proves \eqref{e_expected_property}.

    \item  In general however, $Y^+$ can contain an element $\lambda$ which has an infinite $W_0$-orbit and an infinite stabilizer. This is for example the case if $\cS$ is associated with a Kac-Moody matrix of size greater than $2$, with all non-diagonal coefficients with absolute value greater than $2$, by \cite[Proposition 1.3.21]{kumar2002kac}. Then if $i\in {I_A}$ and $\lambda\in Y^{+}$ satisfies $\alpha_i(\lambda)>0$ and $\alpha_j(\lambda)=0$ for every $j\in {I_A}\setminus\{i\}$, we have that $W_{0,\lambda}=\langle r_j\mid j\in {I_A}\setminus \{i\}\rangle$ is infinite and $W_0^{\lambda}$ is infinite. In that case, \eqref{e_expected_property} is not obvious.
\end{enumerate}

\end{Remark}

\begin{Proposition}\label{propFiniteness_quotients_stabilizers}
Let $\lambda,\nu\in \T$, $\lambda\neq \nu$. Then $W_{0,\nu}/(W_{0,\lambda}\cap W_{0,\nu})$ is finite if and only if there exists $\lambda'\in \T$ such that $\nu\in ]\lambda,\lambda'[$.
\end{Proposition}

\begin{proof}
Let $\lambda,\lambda'\in \T$, $\lambda'\neq \lambda$. Take $\nu\in ]\lambda,\lambda'[$. Let us prove that $W_{0,\nu}/(W_{0,\lambda}\cap W_{0,\nu})$ is finite. 
By \cite[5.1.5 Théorème (iv)]{remy2002groupes}, $[\lambda,\lambda']$ is contained in a finite number of vectorial faces. Let $\tau:[0,1]\rightarrow [\lambda,\lambda']$ be the affine parametrization such that $\tau(0)=\lambda$ and $\tau(1)=\lambda'$. Let $n\in  \N$, $t_0=0<t_1<\ldots <t_n =1$ be such that for all $i\in \llbracket 0,n-1\rrbracket$, $]\tau(t_i),\tau(t_{i+1})[$ is contained in a vectorial face $F_i^v$ and $\tau(t_i),\tau(t_{i+1})\notin F_i^v$. For $i\in \llbracket 0,n\rrbracket$, set $\lambda_i=\tau(t_i)$. Let $\mu\in ]\lambda,\lambda'[\setminus\{\lambda_i\mid i\in \llbracket 0,n\rrbracket\}$. Then $W_{0,\mu}=W_{0,[\lambda,\lambda']}$. Indeed, let $i\in \llbracket 0,n-1\rrbracket$ be such that $\mu\in F_i$. Then $W_{0,F_i}=W_{0,\mu}$. Thus every $w\in W_{0,\mu}$ fixes a neighborhood of $\mu$ in $[\lambda,\lambda']$, and thus fixes $[\lambda,\lambda']$. In particular, $W_{0,\mu}/W_{0,[\lambda,\lambda']}$ is finite. Thus we only need to prove that $W_{0,\lambda_i}/W_{0,[\lambda,\lambda']}$ is finite for every $i\in \llbracket 1,n-1\rrbracket$.

Let $i\in \llbracket 1,n-1\rrbracket$. Choose $\mu\in F_{i-1}^v$. Then maybe considering $w\cdot F^v_{i-1}$, for some $w\in W_0$, we may assume that $F_{i-1}^v\subset \overline{C^v_f}$. As $\lambda_i\in \overline{F^v_i}$, we also have $\lambda_i\in \overline{C^v_f}$.   Set $J_A=\{j\in I_A\mid \alpha_j(\mu)=0\}$, $J_i=\{j\in J_A\mid \alpha_j(\lambda_i)=0\}$ and $\tilde{J_i}=J_A\sqcup J_i$. Then by Subsection~\ref{sss_parabolic_subgroups}, we have $W_{0,\lambda_i}=W_{J_i}=\langle r_j\mid j\in \tilde{J_i}\rangle$ and $W_{0,\mu}=W_{J_A}=\langle r_j\mid j\in J_A\rangle$. Set $W_0^{J_i}=\{w\in W_{0,\lambda_i}\mid wr_j>w, \forall j\in J_i\}$. By \cite[Proposition 2.4.4]{bjorner2005combinatorics}, for every $w\in W_{0,\lambda_i}$, there exists a unique $w^{J_i}\in W_0^{J_i},w_{J_i}\in W_{J_i}$ such that $w=w^{J_i}w_{J_i}$. Moreover, $\ell(w)=\ell(w^{J_i})+\ell(w_{J_i})$.

 Let  $w\in W_0^{J_i}\subset W_{0,\lambda_i}$. Write \[w=w_{k+1}r_{j_k}w_{k}\ldots r_{j_2}w_2 r_{j_1} w_1,\] with $w_m\in W_{J_A}$ for $m\in \llbracket 1,k+1\rrbracket$, $j_m\in J_A$, for $m\in \llbracket 1,k\rrbracket$ and $\ell(w)=\sum_{m=1}^{k+1} \ell(w_i)+k$. Note that $w_1=1$, by \cite[Lemma 2.4.3]{bjorner2005combinatorics}. For $m\in \llbracket 1,k\rrbracket$, set \[t_m= (w_m r_{j_{m-1}} \ldots r_{j_1} w_1)^{-1} r_{j_m} r_{j_{m-1}} \ldots r_{j_1} w_1.\] 
 Let $m\in \llbracket 1,k\rrbracket$. Suppose $t_m\in W_{J_A}$. Set \[w'=w_{k+1}r_{j_k}\ldots \hat{r}_{j_m} \ldots r_{j_1}w_1=w_{k+1}r_{j_k} \ldots r_{j_{m+1}} w_{m+1} w_m r_{j_{m-1}}\ldots r_{j_{1}}w_1.\] Then $w =w't_m$. Write $w'=w'^{J_i}w'_{J_i}$, with $w'^{J_i}\in W_0^{J_i},w'_{J_i}\in  W_{J_i}$. Then $w=w'^{J_i}.(w'_{J_i}t_m)$. By  \cite[Proposition 2.4.4 (i)]{bjorner2005combinatorics}, $w=w'^{J_i}$. But $\ell(w'^{J_i})\leq \sum_{m=1}^{k+1}\ell(w_i)+k<\ell(w)$: a contradiction. Therefore $R\subset W_{0,\lambda_i}\setminus W_{J_A}$.

Let $m\in \llbracket 1,k\rrbracket$. Denote by $\alpha_{t_m}\in \Phi_+$ the root such that $t_m$ fixes $\alpha_{t_m}^{-1}(\{0\})$. Then $\alpha_{t_m}(\lambda_i)=0$ and $\alpha_{t_m}(\mu)\neq 0$. In particular, $\alpha_{t_m}(\lambda)\alpha_{t_m}(\lambda')<0$. Therefore \[\alpha_{t_m}\in \Phi(\lambda,\lambda'),\] where \[\Phi(\lambda,\lambda')=\{\alpha\in \Phi_+\mid \alpha(\lambda)<0\}\cup \{\alpha\in \Phi_+\mid\alpha(\lambda')<0\}.\] By \cite[Proposition 3.12 c)]{kac1994infinite}, $K=|\Phi(\lambda,\lambda')|$ is finite.   By \cite[Lemma 1.3.1]{bjorner2005combinatorics}, $|\{t_m\mid m\in \llbracket 1,k\rrbracket\}|=k$ and thus we have \[k\leq K.\] We have \[w= w_{k+1} r_{j_k} w_{k+1}^{-1}. (w_{k+1}w_k) r_{j{k-1}} (w_{k+1}w_k)^{-1}\ldots (w_{k+1}w_k\ldots w_2)r_{j_1} (w_{k+1}\ldots w_2)^{-1}.(w_{k+1}\ldots w_1).\] Consequently we can write \[w(\mu)= \tilde{r}_k\tilde{r}_{k-1}\ldots \tilde{r}_{1}(\mu),\] where for $m\in \llbracket 1,k\rrbracket$, $\tilde{r}_m\in \{r_{\alpha}\mid\alpha \in \Phi(\lambda,\lambda')\}$. This proves that $W_0^{J_i}\cdot \mu$ is finite, i.e $W_{0,\lambda_i}/W_{0,\mu}$ is finite, which proves the first implication.

\medskip

Conversely, let $\lambda,\nu\in \T$ be such that $W_{0,\nu}/(W_{0,\lambda}\cap W_{0,\nu})$ is finite. 
 By \cite[5.1.5 Théorème (iv)]{remy2002groupes}, $[\nu,\lambda]$ is contained in a finite number of vectorial faces. As $W_0$ acts by affine maps on $\A$, we have $W_{0,\lambda_1}\cap W_{0,\nu}=W_{0,[\lambda,\nu]}$ for all $\lambda_1\in ]\nu,\lambda]$. Therefore, up to replacing $\lambda$ by an element of $]\nu,\lambda]$ close enough to $\nu$, we may assume that $\lambda$ belongs to a vectorial face $F^v$ containing $\nu$ in its closure. Up to translating $F^v$ by an element of $W_0$, we may also assume $F^v\subset \overline{C^v_f}$. Let $J_A\subset {I_A}$ be such that $F^v=F^v(J_A)=\{x\in \A\mid \alpha_i(x)=0, \forall i\in {I_A}\setminus J_A, \alpha_i(x)=0,\forall i\in {I_A}\setminus J_A\}$. Let $K_A=\{i\in {I_A}\mid \alpha_i(\nu)>0\}$. Then $K_A\supset J_A$ since $\nu$ belongs to $\overline{F^v}$.  Therefore $W_{0,\nu}=\langle r_i\mid i\in K_A\rangle\supset W_{0,\lambda}=\langle r_i\mid i\in J_A\rangle$. We define $\leq_{K_A}$ on $\A$ by $x\leq_{K_A} y$ if and only if $y-x\in \bigoplus_{i\in {K_A}} \R_{\geq 0}\alpha_i^\vee$, for $x,y\in \A$.

  Let $t\in \R_{>0}$ and consider $x_t=\nu+t(\nu-\lambda)$.  By assumption, $W_{0,\nu}\cdot x_t$ is finite. Let $w\in W_{0,\nu}$ be such that $w(x_t)$ is maximal in $W_{0,\nu} \cdot x_t$ for $\leq_{K_A}$. Let $j\in {K_A}$. Then $r_j.w(x_t)\leq_{K_A} x_t$ and thus $\alpha_j(w(x_t))\leq 0$. Let $j\in {I_A}\setminus {K_A}$. Then $\alpha_j(w(x_t))=\alpha_j(\nu+tw\cdot (\nu-\lambda))=\alpha_j(\nu)+tw.\alpha_j(\nu-\lambda)$ and for $t$ small enough, we have $\alpha_j(w(x_t))>0$. Therefore $w(x_t)\in \overline{C^v_f}$ and thus $x_t\in \T$. As $\nu\in ]\lambda,x_t[$, we get the lemma. 

\end{proof}

\begin{Corollary}\label{c_finiteness_property_segments}
    Let $\lambda,{\lambda'}\in \cT$, $\lambda\neq {\lambda'}$ and $\nu\in ]\lambda,{\lambda'}[$. Let $E\subset W_0$ be such that $E\cdot \nu$ is finite. Then $|E\cdot \lambda|\leq |E\cdot \nu|.|W_{0,\nu}/W_{0,[\lambda,{\lambda'}]}|$. In particular, $E\cdot \lambda$ is finite.
\end{Corollary}

\begin{proof}
  Let $F'\subset W_{0,\nu}$ be a set of representatives of $W_{0,\nu}/W_{0,[\lambda,{\lambda'}]}$. Then $F'$ is finite by Propositio n~\ref{propFiniteness_quotients_stabilizers}.  Let $F\subset W_0$ be finite and such that $E\cdot \nu=\{w(\nu)\mid w\in F\}$.  We have $E\subset \bigcup_{w\in F} w W_{0,\nu}=\bigcup_{w\in F,w'\in F'}ww'W_{0,[\lambda,{\lambda'}]}$ and thus $E\cdot \lambda=\{ww'(\lambda)\mid w\in F,w')\in F'\}$. 
\end{proof}

\subsection{Inclusion of  $\sT_w$ in $\sS_w$}\label{ss_inclusion_Tw_Sw}

In this subsection, we prove that if $w\in W_0$ and $\tau\in Y^+$, then $\sT_w(\tau)$ is contained in a set $\sS_w(\tau)$ that we introduce below. We deduce a finiteness result for $\sT_w$, see Proposition~\ref{p_Inclusion_TCCw_Sw} and Corollary~\ref{corFiniteness_w_mu_in_Tw_lambda}.

\subsubsection{The set $\sS_{w}$}

Fix $\lambda\in Y^+$.
For any $w \in W_0$, define:
\begin{align}
  \sS_w(\lambda) = \left\{ v(\lambda) \suchthat v \leq w \textand v(\lambda) \leq_{Q^\vee} w(\lambda) \right\}
\end{align}
It is easy to see that if $\tilde{\lambda}\in Y^{++}$, then we have:
\begin{align}
 \sS_w(\tilde{\lambda}) = \{ w(\tilde{\lambda}) \}
\end{align}
Directly from the definition, we also see that every element of $\sS_w(\tau)$ is less than or equal to $w(\tau)$ in the dominance order.

\begin{Proposition}\label{prop_finiteness_property_Sw}
  Fix $\lambda\in Y^+$, and $\mu\in Y$. Then the set $H=\{w\in W_0\mid \mu\in \sS_w(\lambda)\}$ is stable by right multiplication by $W_{0,\lambda}$ and $H/W_{0,\lambda}$ is finite.
\end{Proposition}

\begin{proof}
    Let $w\in H$ and $v\in W_{0,\lambda}$. Then by definition, there exists $v\in [1,w]$ such that $v(\lambda)=\mu$. Let $u\in W_{0,\tau}$. Then by \cite[Proposition 2.5.1]{bjorner2005combinatorics}, if $v'=\pr_{W_0^{\lambda}}(v)$, then we have $v'\leq \pr_{W_0^{\lambda}}(w)=\pr_{W_0^{\lambda}}(wu)$. By \cite[Proposition 2.4.4]{bjorner2005combinatorics}, we have $\pr_{W_0^{\lambda}}(wu)\leq wu$, which proves that $v'\leq wu$. Then $\mu=v'(\lambda)$ with $v'\in [1,wu]$, $v'(\lambda)\leq_{Q^\vee} wu(\lambda)$ and  hence $\mu\in \sS_{wu}(\lambda)$. Therefore  $wu\in H$ and we have the desired stability.

    Now if $w\in H$, we have $\mu\leq_{Q^\vee} w(\lambda)\leq_{Q^\vee} \lambda^{++}$, by definition and by \eqref{e_GR2.4}. Therefore $w(\lambda)$ belongs to the finite set $[\mu,\lambda^{++}]_{\leq_{Q^\vee}}$, which proves the lemma.
\end{proof}

\subsubsection{Facts about $\sT_{w}$}

\begin{Lemma}
Let $\lambda\in Y$, $w\in W_0$ and $\mu\in \sT_w(\lambda)$. Then there exists $v\in [1,w]$ such that $\mu=v(\lambda)$.
\end{Lemma}

\begin{proof}
    Let $\underline{m}=(i_1,\ldots,i_k)$ be a reduced writing of $w$. We have $\mu\in \sT_{\underline{m}}(\lambda)\subset \sT_{w}(\lambda)$.  By definition of $\sT_{\underline{m}}(\lambda)$, there exists $(s_1,\ldots,s_k)\in \prod_{j=1}^k \{1,r_{i_j}\}$ such that $\mu=s_1\ldots s_k(\lambda)$. Set $v=s_1\ldots s_k$. Then $v\leq w$ by \cite[Corollary 2.2.3]{bjorner2005combinatorics} and $\mu=v(\lambda)$. 
\end{proof}

\begin{Lemma}\label{lemInequality_TCC_w}
Let $w\in W_0$ and $\lambda\in Y^+$. Then for all  $\mu\in \sT_{w}(\lambda)$, we have  $\mu \leq_{Q^\vee} w(\tau)$.
\end{Lemma}

\begin{proof}
We proceed by induction on $\ell(w)$. If $\ell(w)\leq 1$, the result is clear so we assume $\ell(w)\geq 2$. Let $\underline{m}$ be a reduced writing of $w$. Write $\underline{m} = (\underline{m'}, r_i)$ where $i\in I_A$ and  $\ell(w) = \ell(\underline{m'}) + 1$. Let $w'$ be the element of $W_0$ corresponding to $\underline{m'}$.  Our assumption is equivalent to the fact that  $w(\alpha_i^\vee)$ is a negative root.  

We compute:
\begin{align}
  \sT_{(\underline{m'},i)}(\lambda) = \sT_{\underline{m'}} (\sT_{i} (\lambda)).
\end{align}
There are two cases.

Case 1: The easier case of $\langle \lambda, \alpha_i \rangle \geq 0$. In this case, $\sT_{i}(\lambda) = \left\{ r_i(\lambda) \right\}$. So
$\sT_{(\underline{m'},i)}(\lambda) = \sT_{\underline{m'}}(r_i(\lambda))$. By induction, all elements of this set are less than or equal to $w'(r_i(\lambda)) = w(\lambda)$.

Case 2: The case of $\langle \lambda, \alpha_i \rangle < 0$. We have  $\sT_{i}(\lambda) = \left\{ r_i(\lambda), \lambda \right\}$, so by Lemma~\ref{l_sT_as_a_union}, we have:
\begin{align}
\sT_{(\underline{m'},r_i)}(\lambda) = \sT_{\underline{m'}}(r_i(\lambda)) \cup \sT_{\underline{m'}}(\lambda).
\end{align}
Every element of $\sT_{\underline{m'}}(r_i(\lambda))$ is less than or equal to $w(\lambda)$ by induction. Every element of $\sT_{\underline{m'}}(\lambda)$ is less than or equal to $w'(\lambda)$ by induction.

Let $c = -\langle \lambda, \alpha_i \rangle > 0$, so we have $r_i(\lambda) = \lambda + c \alpha_i^\vee$. We compute:
\begin{align}
w'(\lambda) = w(r_i(\lambda)) = w(\lambda) + c w(\alpha_i^\vee)
\end{align}
Because $c > 0$ and $w(\alpha_i^\vee) < 0$, we conclude that $w'(\lambda) <_{Q^\vee} w(\lambda)$. So every element of $\sT_{\underline{m'}}(\lambda)$ is less than or equal to $w(\lambda)$.
\end{proof}

\begin{Proposition}\label{p_Inclusion_TCCw_Sw}
 Let $\lambda\in Y^+$ and $w\in W_0$. Then  we have:
  \begin{align}
    \sT_w(\lambda) \subseteq \sS_w(\lambda)
  \end{align}
\end{Proposition}

\begin{Corollary}\label{corFiniteness_w_mu_in_Tw_lambda}
    Let $\lambda,\mu\in Y^{+}$. Then $\{w(\lambda)\mid w\in W_0,\mu\in \sT_w(\lambda)\}$ is finite. 
\end{Corollary}

\begin{proof}
    This follows from Proposition~\ref{p_Inclusion_TCCw_Sw} and Proposition~\ref{prop_finiteness_property_Sw}.
\end{proof}

\subsection{Depth for elements far from the dominant cone}\label{ss_control_depth}

In this subsection, we prove Lemma~\ref{lemDominant_height_simple_reflexion}, which studies the behavior of the depth of the elements of $]\lambda,r_i(\lambda)[_{Q^\vee}$, when $\lambda\in W_0\cdot \tilde{\lambda}$ goes ``far'' from $C^v_f$, for a fixed $\tilde{\lambda}\in Y^{++}$ and $i\in {I_A}$. For example, assume that the centralizer $C(r_i)$ is finite (which is the case when the Kac-Moody matrix $A$ has size $2$). Then this lemma asserts that if $\tilde{\lambda}\in Y^{++}$ is regular and  $w\in W_0$ has ``large'' length, then the elements of $]w(\tilde{\lambda}),r_iw(\tilde{\lambda})[_{Q^\vee}$ have ``large'' depth. 

\begin{Lemma}\label{l_characterization_centralizer_reflection}
    Let $i\in I_A$ and $w\in W_0$. Then $w\in C(r_i)$ if and only if $w(\alpha_i^\vee)=\pm \alpha_i^\vee$.
\end{Lemma}

\begin{proof}
    Let $w\in C(r_i)$. Let $x\in Y$. Then we have \[r_iw(x)=w(x)-\langle w(x),\alpha_i\rangle \alpha_i^\vee=wr_i(x)=w(x)-\langle x,\alpha_i\rangle w(\alpha_i^\vee).\] As $\alpha_i(Y)\neq \{0\}$, we deduce $w(\alpha_i^\vee)\in \R \alpha_i^\vee\cap \Phi^\vee=\{\pm \alpha_i^\vee\}$. 

    Conversely, let $w\in W_0$ be such that $w(\alpha_i^\vee)=\pm \alpha_i^\vee$. Up to considering $wr_i$ instead of $w$, we may assume $w(\alpha_i^\vee)=\alpha_i^\vee$. Then $w\in C(r_i)$, by \cite[1.3.11 Theorem (b5)]{kumar2002kac}. 
\end{proof}

\begin{Lemma}\label{lemDominant_height_simple_reflexion}
  
Let  $i\in {I_A}$.   For $\lambda\in Y^+$, set \[m_\lambda(i)=\min\{\htt(\lambda^{++}-\mu^{++})\mid \mu\in ]\lambda,r_i(\lambda)[_{Q^\vee} \}\in \N\] (if $]\lambda,r_i(\lambda)[=\emptyset$, i.e if $\alpha_i(\lambda)\leq 1$, one sets $m_\lambda(i)=+\infty$). Let $\tilde{\lambda}\in Y^{++}$. Then:\begin{enumerate}
\item for all $\lambda\in {W_0}\cdot\tilde{\lambda}$ and $\lambda'\in C(r_i)\cdot\lambda$, we have  $m_\lambda(i)=m_{\lambda'}(i)$.

\item Let ${k}\in \N$. Then  there exists a finite set $F\subset W_0\cdot\tilde{\lambda}$ such that for all  $\lambda\in W_0\cdot \tilde{\lambda}$ satisfying $m_\lambda(i)\leq {k}$, we have $\lambda\in C(r_i)\cdot F$.
\end{enumerate}  

\end{Lemma}

\begin{proof}
(1) Let $\lambda\in {W_0}\cdot \tilde{\lambda}$ and $\lambda'\in C(r_i)\cdot \lambda$. Write $\lambda'=w(\lambda)$ with $w\in C(r_i)$. We have $]\lambda',r_i(\lambda')[_{Q^\vee}=w\cdot ]\lambda,r_i(\lambda)[_{Q^\vee}$, thus $(]\lambda',r_i(\lambda')[_{Q^\vee})^{++}=(]\lambda,r_i(\lambda)[_{Q^\vee})^{++}$ and hence $m_\lambda(i)=m_{\lambda'}(i)$.

(2)   Let ${k}\in \N$ and let \[X_{k}(\tilde{\lambda})=\{\lambda\in W_0\cdot \tilde{\lambda} \mid \ m_{\lambda}\leq {k},\langle \lambda,\alpha_i\rangle>0\}.\] Note that for every $\lambda\in X_{k}(\tilde{\lambda})$, we have $\langle \lambda,\alpha_i\rangle\geq 2$, since $]\lambda,r_i(\lambda)[_{Q^\vee}\neq \emptyset$ and thus $r_i(\lambda)\neq \lambda-\alpha_i^\vee$. Let \[E_{k}(\tilde{\lambda})=\{\tau\in \tilde{\lambda}-\bigoplus_{i\in I_A} \Z_{\geq 0} \alpha_i^\vee\mid \htt(\tilde{\lambda}-\tau)\leq {k}\}.\]

Let $\lambda\in X_{k}(\tilde{\lambda})$. Let $\mu=\mu_\lambda=\lambda-\alpha_i^\vee$.  Then $\mu\in ]\lambda,r_i(\lambda)[_{Q^\vee}$ and thus $\mu\in Y^+$.  By Lemma~\ref{lemDominant_height_segments}, \begin{equation}\label{e_equation_m_lambda}
m_{\lambda}(i)=\htt( \tilde{\lambda}-\mu^{++})\leq {k}.
\end{equation}

We have $\mu=w_\mu(\mu^{++})\in ]\lambda,r_i(\lambda)[_{Q^\vee}$ and hence \[\mu^{++}\in ]w_\mu^{-1}(\lambda),w_{\mu}^{-1}r_i(\lambda)[_{Q^\vee}.\]

Set $\nu=\nu_\lambda=\max_{Q^\vee}\big((w_{\mu})^{-1}(\lambda),w_{\mu}^{-1}r_i(\lambda)\big)$. We have $w_\mu^{-1}r_i(\lambda)=w_{\mu}^{-1}r_i w_\mu.w_\mu^{-1}(\lambda)$ and thus \begin{equation}\label{e_w_lambda}
    \{w_\mu^{-1}(\lambda),w_\mu^{-1}r_i(\lambda)\}=\{\nu,w_{\mu}^{-1}r_i w_\mu(\nu)\}.
\end{equation} Set  \[\beta_\lambda^\vee=(w_{\mu})^{-1}(\alpha_i^\vee)\text{ and }\beta_\lambda=(w_{\mu})^{-1}(\alpha_i).\]

We have \begin{equation}\label{e_mu_plus_interval}
    \mu^{++}\in ]\nu,w_\mu^{-1}r_iw_\mu(\nu)[_{Q^\vee}=]\nu,\nu-\langle\nu,\beta_\lambda\rangle \beta_\lambda^\vee[_{Q^\vee}.
\end{equation}

Therefore we have $\min_{Q^\vee}(\nu,w_\mu^{-1}r_iw_\mu(\nu))\leq_{Q^\vee}\mu^{++}\leq_{Q^\vee} \max_{Q^\vee}(\nu,w_\mu^{-1}r_iw_\mu(\nu))$  and using \eqref{e_w_lambda}, we deduce $\mu^{++}\leq_{Q^\vee}\nu$. As $\mu^{++}\in E_{k}(\tilde{\lambda})$ (by \eqref{e_equation_m_lambda} and as $\nu\leq_{Q^\vee}\tilde{\lambda}$ (by \eqref{e_GR2.4}), we have \[\nu\in E_{k}(\tilde{\lambda}). \]

According to \eqref{e_mu_plus_interval}, we also have \[\nu+\R(\mu^{++}-\nu)=\nu+\R\beta_\lambda^\vee\] and hence \[\R\beta_\lambda^\vee\in \{\R(\tau-\tau')\mid \ \tau,\tau'\in E_{k}(\tilde{\lambda})\},\] which is a finite set. Moreover, $\beta_\lambda^\vee\in \Phi^\vee$ and $\R\alpha^\vee\cap \Phi^\vee=\{\alpha^\vee,-\alpha^\vee\}$ for every $\alpha^\vee\in \Phi^\vee$. Therefore \[\{\beta_\lambda^\vee\mid \lambda\in X_{k}(\tilde{\lambda})\}\] is finite.

For $\tau\in E_{k}(\tilde{\lambda})$,  set $X_{k}(\tilde{\lambda},\tau)=\{\lambda\in X_{k}(\tilde{\lambda})\mid (\lambda-\alpha_i^\vee)^{++}=\tau\}$. Let $\tau\in E_{k}(\tilde{\lambda})$ be such that $E_{k}(\tilde{\lambda},\tau)$ is non-empty. Let $\lambda,\lambda'\in E_{k}(\tilde{\lambda},\tau)$ be such  that $\beta_{\lambda}^\vee=\beta_{\lambda'}^\vee$. Then $w_{\lambda'-\alpha_i^\vee}(w_{\lambda-\alpha_i^\vee})^{-1}(\alpha_i^\vee)=\alpha_i^\vee$ and  thus $w_{\lambda'-\alpha_i^\vee}(w_{\lambda-\alpha_i^\vee})^{-1}\in C(r_i)$, by Lemma~\ref{l_characterization_centralizer_reflection}.   Therefore: \[(w_{\lambda-\alpha_i^\vee})^{-1} C(r_i)=(w_{\lambda'-\alpha_i^\vee})^{-1} C(r_i)\] and thus $C(r_i)w_{\lambda-\alpha_i^\vee}=C(r_i)w_{\lambda'-\alpha_i^\vee}.$ Consequently: \begin{align*}C(r_i)\cdot w_{\lambda-\alpha_i^\vee}(\tau)&=C(r_i)\cdot w_{\lambda'-\alpha_i^\vee}(\tau)\\
&=C(r_i)\cdot w_{\lambda-\alpha_i^\vee}((\lambda-\alpha_i^\vee)^{++})\\
&=C(r_i)\cdot w_{\lambda'-\alpha_i^\vee}((\lambda'-\alpha_i^\vee)^{++})\\
&=C(r_i)\cdot(\lambda- \alpha_i^\vee)=C(r_i)\cdot(\lambda'-\alpha_i^\vee).\end{align*} 
Take $u \in C(r_i)$ such that $\lambda'-\alpha_i^\vee=u(\lambda-\alpha_i^\vee).$ Then $u(\lambda-\alpha_i^\vee)=u(\lambda)-u(\alpha_i^\vee)\in \{u(\lambda)\pm \alpha_i^\vee\}$, by Lemma~\ref{l_characterization_centralizer_reflection}. Therefore $\lambda'\in  \{u(\lambda),u(\lambda)+2\alpha_i^\vee\}$. Let $F$ be a finite subset of $X_{k}(\tilde{\lambda})$ such that $\{\beta_\lambda^\vee\mid\lambda\in X_{k}(\tilde{\lambda})\}=\{\beta_\lambda^\vee\mid \lambda\in F\}$. Then: \[X_{{k}}(\tilde{\lambda})=\bigcup_{\tau\in X_{k}(\tilde{\lambda})}X_{k}(\tilde{\lambda},\tau)\subset\{C(r_i)\cdot \lambda\mid \lambda\in F\}\cup \{C(r_i).(\lambda+2\alpha_i^\vee)\mid \lambda\in F\text{ and }\lambda+2\alpha_i^\vee\in Y^+\}.\] Lemma follows. 
\end{proof}

\subsection{Finiteness result for $\overline{\sT_w}$}\label{ss_finiteness_result_Tw_bar}
We now prove the main technical result of this paper. This enables to expand $T$-series in $Z$-series.

\begin{Lemma}\label{lemFiniteness_support_TWhat}
Let $\tilde{\lambda}\in Y^{++}$ and $\mu\in Y^+$. Then $\{w(\lambda)\mid w\in W_0, \lambda\in W_0\cdot \tilde{\lambda} \text{ such that } \overline{\sT_w}(\lambda)\ni \mu\}$ is finite. In particular, $\{\lambda\in W_0\cdot\tilde{\lambda}\mid \mu\in \supp^{Z}(\Tb_\lambda)\}$ is finite. 
\end{Lemma}

\begin{proof}
  Let $\lambda\in W_0\cdot \tilde{\lambda}$. If $\mu\in \overline{\sT_w}(\lambda)$ for some $w\in W_0$, then by Lemma~\ref{lemGrowth_dominant_height_segments}, we have  $\mu^{++}\leq_{Q^\vee}\lambda^{++}=\tilde{\lambda}$. 
  For $k \in {\Z_{\geq 0}}$, we set $\sP_k$ to be the statement: 

``for all $\tilde{\lambda}\in Y^{++}$, for all $\mu\in Y^+$ such that $\tilde{\lambda}\geq_{Q^\vee} \mu^{++}$ and $\htt(\tilde{\lambda}-\mu^{++})\leq k$, there exists a finite set $F(\tilde{\lambda},\mu)\subset W_0\cdot\tilde{\lambda}$ such that for all $w\in W_0$, for all $\lambda\in W_0\cdot \tilde{\lambda}$,  \[\mu\in \overline{\sT_{w}}(\lambda) \Rightarrow w(\lambda)\in F(\tilde{\lambda},\mu). \text{''}\]

We will prove $\sP_k$ for all $k$ by induction, which is the Lemma. Let us first prove $\sP_0$. Let $\tilde{\lambda}\in Y^+$ and $\mu\in W_0\cdot\tilde{\lambda}$. Let $\lambda\in W_0\cdot\tilde{\lambda}$ and $w\in W_0$ be such that $\mu\in \overline{\sT_w}(\lambda)$. Then by Lemma~\ref{lemGrowth_dominant_height_segments}, we have  $\mu\in \sT_w(\lambda)$. Using Proposition~\ref{p_Inclusion_TCCw_Sw}, we deduce $\mu\leq_{Q^\vee} w(\lambda)\leq_{Q^\vee} \tilde{\lambda}$. Consequently, $\sP_0$ is true, with  $F(\tilde{\lambda},\mu)=[\mu,\tilde{\lambda}]_{\leq_{Q^\vee}}\cap W_0\cdot \tilde{\lambda}$. 

Let $k\in \Z_{\geq 1}$ be such that $\sP_{k-1}$ is true. Let $\tilde{\lambda}\in Y^{++}$ and $\mu\in Y^+$ be such that $\mu^{++}\leq_{Q^\vee} \tilde{\lambda}$ and $\htt(\tilde{\lambda}-\mu)=k$. Using Lemma~\ref{lemDominant_height_simple_reflexion}, we choose $K\in {\Z_{\geq 0}}$ such that for every $w\in W_0$, for all $i\in {I_A}$, \[\left(\exists \nu\in ]w(\tilde{\lambda}),r_iw(\tilde{\lambda})[_{Q^\vee}\mid \htt(\tilde{\lambda}-\nu^{++})\leq k\right) \Rightarrow \left( \exists v\in C(r_i),u\in  W_0\mid  w(\tilde{\lambda})=vu(\tilde{\lambda}),\ell(u)\leq K\right).\]

Let $\lambda\in W_0\cdot\tilde{\lambda}$, $w\in W_0$ and assume $\mu\in \overline{\sT_w}(\lambda)$. Let $\underline{m}\in \Red(w)$. Then $\mu\in \overline{\sT_{\underline{m}}}(\lambda)$. Write $\underline{m}=(s_n,\ldots,s_1)$, with $n\in {\Z_{\geq 0}}$ and $s_i\in \{r_j\mid j\in {I_A}\}$ for all $i \in \llbracket 1,n\rrbracket$. 
Then there exists a sequence $\lambda_0,\lambda_1,\ldots,\lambda_n\in Y^+$ such that $\lambda_0=\lambda$, $\lambda_n=\mu$ and $\lambda_{i+1}\in \overline{\sT_{s_i}}(\lambda_i)$ for all $i\in \llbracket 0,n-1\rrbracket$. 
As $\mu\notin W_0\cdot \tilde{\lambda}$, there exists $n''\in \llbracket 1,n-1\rrbracket$ such that $\lambda_{n''+1}\notin \sT_{s_{n''}}(\lambda_{n''})$. Let $n'$ be the smallest integer satisfying this property.
 Let $\underline{m'}=(s_{n'-1},\ldots,s_1)$, $s=s_{n'}$ and $\underline{\tilde{m}}=(s_n,\ldots,s_{n'+1})$. Let $\lambda'=\lambda_{n'}$, $\nu'=\lambda_{n'+1}$. Then \begin{equation}\label{e_lambda'_T_m'}
\nu'\in ]\lambda',s(\lambda')[_{Q^\vee},  \mu\in \overline{\sT_{\underline{\tilde{m}}}}(\nu')\text{ and }
    \lambda'\in \sT_{\underline{m'}}(\lambda).
\end{equation}

Write $\lambda'=vu(\tilde{\lambda})$, with $v\in C(s)$ and $u\in W_0$ such that $\ell(u)\leq K$. Then: \[\begin{aligned} \nu',s(\nu') \in \,\,]\lambda',s(\lambda')[\,\,=\,\, ]vu(\tilde{\lambda}),svu(\tilde{\lambda})[ \,\, 
=  \,\,]vu(\tilde{\lambda}),vsu(\tilde{\lambda})[ \,\,= vu\cdot ]\tilde{\lambda},u^{-1}su(\tilde{\lambda})[. \end{aligned}\]

Set \begin{equation}\label{e_definition_nu}
    \nu=u^{-1}v^{-1}s(\nu')\in ]\tilde{\lambda},u^{-1}su(\tilde{\lambda})[.
\end{equation} Let $\tilde{w}=s_{n}\ldots s_{n'+1}\in W_0$.  Then $\tilde{w}(\nu')=\tilde{w}svu(\nu)$. 

By \eqref{e_lambda'_T_m'} and Lemma~\ref{l_t} (2), we have: \[
    \tilde{w}svu(\tilde{\lambda})=\tilde{w} s(\lambda')\in \sT_{(\underline{\tilde{m}},s)}(\sT_{\underline{m'}}(\lambda))=\sT_{(\underline{\tilde{m}},s,\underline{m'})}(\lambda).
\] Using Proposition~\ref{p_Inclusion_TCCw_Sw} we deduce \begin{equation}\label{e_bound_w_lambda}
\tilde{w}svu(\tilde{\lambda})\leq_{Q^\vee} w(\lambda)\leq_{Q^\vee} \tilde{\lambda}.\end{equation}

By \eqref{e_lambda'_T_m'}, we have $\mu\in \overline{\sT_{\underline{\tilde{m}}}}(\nu')$,  thus $\sP_{k-1}$ implies that  $\tilde{w}(\nu')\in F((\nu')^{++},\mu)$  or, said differently, \begin{equation}\label{e_tilde_w_s_v_u}
    \tilde{w}svu(\nu)\in F(\nu^{++},\mu).
\end{equation}

Let $E=\{x\in W_0\mid x(\nu)\in F(\nu^{++},\mu)\}$. Then $E\cdot \nu=F(\nu^{++},\mu)$ is finite and thus $E\cdot \tilde{\lambda}$ is finite by Corollary~\ref{c_finiteness_property_segments} and \eqref{e_definition_nu}.

By \eqref{e_tilde_w_s_v_u}, $\tilde{w}svu\in E$ and thus according to \eqref{e_bound_w_lambda}, we can take \[F(\tilde{\lambda},\mu)=  \bigcup_{\tau\in E\cdot \tilde{\lambda}}[\tau,\tilde{\lambda}]_{\leq_{Q^\vee}}.  \]

This proves $\sP_k$. The second part of the Lemma follows, using Proposition~\ref{P_BL_support_IM_basis}. 
\end{proof}

\begin{Theorem}\label{t_finiteness_coefficients}
 For $\mu\in Y^+$, one sets $S^Z(T,\mu)=\{\lambda\in Y^+\mid \mu\in \supp^Z(T_\lambda)\}$ and  $S^T(Z,\mu)=\{\lambda\in Y^+\mid \mu\in \supp^T(Z_\lambda)\}$. 
 Let $S$ be a Weyl almost finite subset of $Y^+$ and $\mu\in Y^+$. Then \[S\cap S^Z(T,\mu)\text{ and }S\cap S^T(Z,\mu)\] are finite.  
\end{Theorem}

\begin{proof}
 Let $\lambda\in S\cap (S^Z(T,\mu)\cup S^T(Z,\mu))$. Let $F$ be a finite subset of $Y^{++}$ such that $S^{++}\subset \bigcup_{\nu\in F}]-\infty,\nu]_{\leq_{Q^\vee}}$.  Then by Proposition~\ref{P_BL_support_IM_basis}, we have $\mu^{++}\leq_{Q^\vee}\lambda^{++} \leq_{Q^\vee}\nu$, for some $\nu\in F$.  Equivalently, $\lambda^{++}$ belongs to $E$, where $E=Y^{++}\cap \bigcup_{\nu\in F} [\mu^{++},\nu]_{\leq_{Q^\vee}}$ is a finite set. Now if $\tilde{\lambda}\in E$, then by Lemma~\ref{lemFiniteness_support_TWhat} and Proposition~\ref{P_equivalence_finiteness}, $E^1_{\tilde\lambda}=W_0\cdot \tilde{\lambda}\cap S^Z(T,\mu)$ and $E^2_{\tilde\lambda}=W_0\cdot \tilde{\lambda}\cap S^T(Z,\mu)$  are finite. Therefore $S\cap S^Z(T,\mu)\subset \bigcup_{\tilde{\lambda}\in E}E^1_{\tilde{\lambda}}$ and  $S\cap S^T(Z,\mu)\subset \bigcup_{\tilde{\lambda}\in E}E^2_{\tilde{\lambda}}$ are finite. 
    \end{proof}

\section{Completed Iwahori-Hecke algebra: the $T$-basis viewpoint}\label{s_T_basis}

In \S\ref{s_Z_presentation}, we defined the completed Iwahori-Hecke algebra $\HBL$ as the set of series $\sum_{\lambda \in Y^+} h_\lambda Z^\lambda$, such that $(h_\lambda)\in (\HCW)^{Y^+}_{\WAF}$.  This is convenient for computations, but is not adapted to regarding elements of $\HBL$ as $I$-bi-invariant functions on $G^{\geq 0}$. 

In this section, we describe $\HBL$ as the set of series $\sum_{\lambda\in Y^+} h_\lambda T_\lambda$, where $(h_\lambda)\in (\HCW)^{Y^+}_{\WAF}$ (see Theorem~\ref{t_T_presentation_HBL} and Theorem~\ref{c_T_version_completed_algebra}). We also give a description of $\HBL$ as a set of series $\sum_{\bw\in W^+}a_\bw T_{\bw}$, where $(a_{\bw})\in \cR^{W^+}$ satisfies an explicit support condition. This gives explicit realization of $\HBL$ as functions on $G^{\geq 0}$ satisfying an explicit support condition.

We then compare $\HBL$ and the algebra $\widetilde{\cH}$ defined in \cite{abdellatif2022erratum} (see \S\ref{ss_comparison_HBL_tilde_cH}), and eventually we explain why the right-hand version of the construction with series of the form $\sum Z^\lambda h_\lambda$ does not work (see \S\ref{ss_definitionr_right}).

\subsection{$T$-series in $\HBL$}\label{ss_T_series_in_HBL}

\begin{Theorem}\label{t_T_presentation_HBL}
    Let $\cR$ be  a ring satisfying Assumption~\ref{a_assumption_ring}. Let $\cS$ be a Kac-Moody root datum and $\cH$ be the associated Iwahori-Hecke algebra over $\cR$ defined in \S\ref{ss_Def_IH_algebras}. Let $\HBL$ be the completed Iwahori-Hecke algebra defined in \S\ref{ss_definition_BL_completed}.  \begin{enumerate}
        \item Let $J$ be a set and $(a_j)_{j\in J}\in \cH^J$. Then $(a_j)$ is summable in the sense of Definition~\ref{d_Summable_family} if and only if:\begin{enumerate}
            \item  $\bigcup_{j\in J}\supp^T(a_j)$ is Weyl almost finite, 

            \item for every $\lambda\in Y^+$, $\{j\in J\mid \lambda\in \supp^T(a_j)\}$ is finite. 
        \end{enumerate}

Moreover if $(a_j)_{j\in J}$ is summable, then $\sum_{j\in J} a_j=\sum_{\lambda\in Y^+}h_\lambda T_\lambda$, where $h_\lambda=\sum_{j\in J}\coeff^T_\lambda(a_j)$, for every $\lambda\in Y^+$. 

 \item For every $(h_\lambda)\in (\HCW)^{Y^+}_{\WAF}$, $(h_\lambda T_\lambda)_{\lambda\in Y^+}$ is summable in $\HBL$ and for every $h\in \HBL$, there exists a unique $(h_\lambda)\in (\HCW)^{Y^+}_{\WAF}$ such that $h=\sum_{\lambda\in Y^+} h_\lambda T_\lambda$.
    \end{enumerate} 
\end{Theorem}

\begin{proof}
    (1) Let $(a_j)\in \cH^J$ and assume that $a_j$ is summable in the sense of Definition~\ref{d_Summable_family}.  Where $U$ denotes either the symbol $T$ or the symbol $Z$, one sets $S^U=\bigcup_{j\in J} \supp^U(a_j)$. By assumption, $S^Z$ is Weyl almost finite and thus by Proposition~\ref{p_Weyl_almost_finite_sets}, there exists a finite subset $F$ of $Y^{++}$ such that $(S^Z)^{++}\subset Y^{++}\cap \bigcup_{\nu\in F} ]-\infty, \nu]_{\leq_{Q^\vee}}$.  If $j\in J$, then \begin{equation}\label{e_support_T_aj}\supp^T(a_j)\subset \bigcup_{\nu\in \supp^Z(a_j)} \supp^T(Z^\nu)\end{equation} and hence \[S^T=\bigcup_{j\in J} \supp^T(a_j)\subset \bigcup_{\nu\in S^Z}\supp^T(Z^\nu).\] Using Proposition~\ref{P_BL_support_IM_basis} (1), we deduce that $(S^T)^{++}\subset Y^{++}\cap \bigcup_{\nu\in F}]-\infty,\nu]_{\leq_{Q^\vee}}$ is   almost finite and hence that $S^T$ is Weyl almost finite.
    
    For  $\lambda\in Y^+$ and $U\in \{T,Z\}$, we set $J^U(\lambda)=\{j\in J\mid \lambda\in \supp^U(a_j)\}$. Let $\lambda\in Y^+$. For $j\in J^T(\lambda)$, there exists $\nu\in \supp^Z(a_j)$ such that $\lambda\in \supp^T(Z^\nu)$, by \eqref{e_support_T_aj}. Therefore \[J^T(\lambda)\subset \bigcup_{\nu\in S^Z\mid \lambda\in \supp^T(Z^\nu)} J^Z(\nu).\] The $J^Z(\nu)$ are finite by assumption. As $S^Z$ is Weyl almost finite by assumption,   we deduce that $J^T(\lambda)$ is finite, using Theorem~\ref{t_finiteness_coefficients}.  Thus we proved that a summable family satisfies (a) and (b).  To prove the other direction, we remark that conditions (a) and (b) are the conditions for a family to be summable, except that we swapped $Z$ and $T$. Thus it suffices to exchange $Z$ and $T$ in our proof.

    It remains to prove the last assertion of (1). Let $(a_j)\in \cH^J$ be a summable family and set $h_\lambda=\sum_{j\in J}\coeff^T_\lambda(a_j)$, for $\lambda\in Y^+$.   Let $\mu\in Y^+$. Then  we have \begin{align*}\coeff^Z_\mu(\sum_{\lambda\in Y^+}h_\lambda T_\lambda)=\sum_{\lambda\in Y^+}\coeff^Z_\mu(h_\lambda T_\lambda)
    &=\sum_{\lambda\in Y^+}\sum_{j\in J} \coeff^Z_\mu(\coeff^T_\lambda(a_j) T_\lambda)\\
    &=\sum_{j\in J} \sum_{\lambda\in Y^+}\coeff_\lambda^T(a_j) \coeff^Z_\mu(T_\lambda)\\
    &= \sum_{j\in J} \coeff_\mu^Z(\sum_{\lambda\in Y^+} \coeff^T_\lambda(a_j) T_\lambda)\\
    &=\sum_{j\in J}\coeff^Z_\mu(a_j)=\coeff^Z_\mu(\sum_{j\in J}a_j),\end{align*} where we can permute the sums since only finitely many non-zero terms appear. Therefore $\sum_{\lambda\in Y^+} h_\lambda T_\lambda =\sum_{j\in J} a_j$, which proves (1).

    (2) Let $(h_\lambda)\in (\HCW)^{Y^+}_{\WAF}$. Then $(h_\lambda T_\lambda)$ clearly satisfies (a) and (b) of (1) and thus $(h_\lambda T_\lambda)$ is summable by (1). 

    Let $h\in \HBL$. Write $h=\sum_{\lambda\in Y^+} \tilde{h}_\lambda Z^\lambda$, with $(\tilde{h}_\lambda)\in (\HCW)^{Y^+}_{\WAF}$. For $\lambda\in Y^+$, set $h_\lambda=\sum_{\mu\in Y^+}\coeff^T_\lambda(\tilde{h}_\mu Z^\mu)$. Then  by (1) applied with $J=Y^+$ and $a_\mu=\tilde{h}_\mu Z^\mu$, for $\mu \in J$, the family $(h_\lambda)$ is a well defined element of $(\HCW)^{Y^+}_{\WAF}$  and we have $h=\sum_{\lambda\in Y^+}h_\lambda T_\lambda$. It remains to prove the uniqueness in (2). Let $(h_\lambda^{(1)}),(h_\lambda^{(2)}) \in (\HCW)^{Y^+}_{\WAF}$. We assume that $(h_\lambda^{(1)})\neq (h_\lambda^{(2)})$. Let us prove that \begin{equation}\label{e_difference}
     \sum_{\lambda\in Y^+} h_\lambda^{(1)} T_\lambda\neq \sum_{\lambda\in Y^+}h_\lambda^{(2)} T_\lambda.   
    \end{equation}  Let $S=\{\lambda\in Y^+\mid h_\lambda^{(1)}\neq h_\lambda^{(2)}\}.$ As $S\subset \supp((h_\lambda^{(1)})_{\lambda\in Y^+})\cup \supp((h_\lambda^{(2)})_{\lambda\in Y^+})$, the set $S$ is Weyl almost finite. Let $\nu\in S$ be such that $\nu^{++}$ is maximal in $S^{++}$, which exists by Proposition~\ref{p_Weyl_almost_finite_sets}. Then by (1) and Proposition~\ref{P_BL_support_IM_basis}, we have, for $i\in \{1,2\}$: \[\pr_\nu^Z(\sum_{\mu\in Y^+}h_\mu^{(i)} T_\mu)=\sum_{\mu\in Y^+}h_\mu^{(i)}\pr^Z_\nu(T_\mu)=h_\nu^{(i)} \pr_\nu^Z( T_\nu)\text{ and }\pr_\nu^Z( T_\nu)\in (\HCW)^\times Z^\nu.\] Therefore by Remark~\ref{r_invertibility_Z_T}, we have $\pr_\nu^Z(\sum_{\mu\in Y^+}h_\mu^{(1)} T_\mu)\neq \pr_\nu^Z(\sum_{\mu\in Y^+}h_\mu^{(2)} T_\mu)$ and in particular, we have \eqref{e_difference}, which proved  the desired uniqueness.

\end{proof}

\begin{Theorem}\label{c_T_version_completed_algebra}
    Let $(h_\lambda),(j_\lambda)\in (\HCW)^{Y^+}_{\WAF}$. Then \[(\sum_{\lambda\in Y^+} h_\lambda T_\lambda) (\sum_{\mu\in Y^+} j_\mu T_\mu)=\sum_{\nu\in Y^+} k_\nu T_\nu,\] where $k_\nu$ is the well defined (i.e with finitely many non-zero terms) sum $\sum_{\lambda,\mu\in Y^+}\coeff^T_\nu(h_\lambda T_\lambda j_\mu T_\mu)$, for every $\nu\in Y^+$. Moreover, if $M_h=\max_{Q^\vee} W_0\cdot \supp((h_\lambda))$ and $M_j=\max_{Q^\vee} W_0\cdot \supp((j_\mu))$, then for all $\nu\in \supp(k)$, we have $\nu^{++}\leq_{Q^\vee} \tilde{\lambda}+\tilde{\mu}$ for some $\tilde{\lambda}\in M_h$ and $\tilde{\mu}\in M_j$. 
\end{Theorem}

\begin{proof}
    This follows from Theorem~\ref{thmConvolution_summable_family}, Theorem~\ref{t_T_presentation_HBL} and Corollary~\ref{c_support_Tlambda_Tmu}. 
\end{proof} 

\begin{Remark}
  It would be more natural to take Theorem~\ref{c_T_version_completed_algebra} as a definition of $\HBL$, but we do not know how to prove directly that it is well defined. 
\end{Remark}

Recall from Proposition~\ref{p_HCW_bimodule} that $\cH=\bigoplus_{\lambda\in Y^+} \HCW  T_\lambda$ and that $\HCW T_\lambda=\bigoplus_{w\in W_0} \cR T_{w(\lambda).w}$ for every $\lambda\in Y^+$. Thus if $(h_\lambda)\in (\HCW)^{Y^+}_{\WAF}$, then we can write the element  $h=\sum_{\lambda\in Y^+}h_\lambda T_\lambda$ as a formal series $\sum_{\bw\in W^+} a_{\bw}T_{\bw}$, with $a_{\bw}\in \cR$ for every $\bw\in W^+$ (recall that $W^+=W_0\ltimes Y^+$). We call it the expansion of $h$ in the basis $(T_{\bw})_{\bw\in W^+}$.  For $\bw\in W^+$, define $\coeff^T_{\bw}:\HBL\rightarrow \cR$ by $\coeff^T_{\bw}(h)=a_{\bw}$ if $h=\sum_{\bv\in W^+} a_{\bv} T_{\bv}\in \HBL$. 

To avoid confusion, we sometimes add an index $\HCW$ to $\coeff^T_\lambda$: we write $\coeff^T_{\lambda,\HCW}:\HBL\rightarrow \HCW$ the map defined by $\coeff^T_{\lambda,\HCW}(\sum_{\mu\in Y^+}h_\mu T_\mu)=h_\lambda$, for $\sum_{\lambda\in Y^+} h_\lambda T_\lambda\in \HBL$ (this is well defined by the uniqueness in Theorem~\ref{t_T_presentation_HBL}).

\begin{Theorem}\label{c_T_presentation}
The algebra $\HBL$ is equal to the set of series $\sum_{\bw\in W^+} a_{\bw} T_{\bw}$ such that 
\begin{enumerate}
    
\item $a_{\lambda.w}\in \cR$ for every $\lambda.w\in W^+$,

\item $\{\lambda\in Y^+\mid  \exists w\in W_0, a_{\lambda.w}\neq 0\}$ is Weyl almost finite,

\item for every $\lambda\in Y^+$, $\{w\in W_0, a_{w(\lambda).w}\neq 0\}$ is finite.
\end{enumerate} Moreover every element of $\HBL$ can be decomposed uniquely as a series of this form.

For $a=\sum_{\bu\in W^+}a_{\bu}\Tb_{\bu}$ and $b=\sum_{\bv\in W^+}b_{\bv}\Tb_{\bv}$ in $\HBL$, we have $a b=\sum_{\bw\in W^+} c_{\bw} \Tb_{\bw}$, where $c_{\bw}=\sum_{\bu,\bv\in W^+} \coeff_{\bw}^T (a_{\bu}b_{\bv}\Tb_{\bu} \Tb_{\bv})$, for $\bw\in W^+$. These sums are well defined in the sense that they contain only finitely many non-zero terms and for all $\nu\in Y^+$, \[\left\{\bw\in \{w(\nu)\cdot w\mid w\in W_0\}\mid \exists  \bu,\bv\in W^+, \coeff^T_{\bw}(a_\bu T_\bu b_{\bv}T_{\bv})\neq 0\right\}\] is finite. 
\end{Theorem}

\begin{proof}
Let $(h_\lambda)\in (\HCW)^{Y^+}_{\WAF}$. Let $h=\sum_{\lambda\in Y^+} h_\lambda T_\lambda\in \HBL$. For $\lambda.w\in W^+$, set $a_{\lambda.w}=\coeff^T_{\lambda.w}(h_{w^{-1}(\lambda)})\in \cR$.  Let $S=\supp^T(h)=\{\lambda\in Y^+\mid h_\lambda\neq 0\}$. By Proposition~\ref{p_HCW_bimodule}, $\{\lambda\in Y^+\mid \exists w\in W_0, a_{\lambda.w}\neq 0\}\subset W_0\cdot S$, which is Weyl almost finite since $S$ is. Let $\lambda\in Y^+$. Then by Proposition~\ref{p_HCW_bimodule}, $\{w\in W_0\mid a_{w(\lambda).w}\neq 0\}= \{w\in W_0\mid \coeff^T_{w(\lambda).w}(h_\lambda T_\lambda)\neq 0\}$, which is finite. Thus the expansion in the basis $(T_{\bw})_{\bw\in W^+}$ of $h$ satisfies (1), (2) and (3). 

Conversely, let $(a_{\bw})\in \cR^{W^+}$ satisfying (2) and (3). By Proposition~\ref{p_HCW_bimodule} and Remark~\ref{r_invertibility_Z_T}f, for every $\lambda\in Y^+$, there exists a unique $h_\lambda\in \HCW$ such that $\sum_{w\in W_0} a_{w(\lambda).w}T_{w(\lambda).w}=h_\lambda T_\lambda$. Then by (2), $\supp(h_\lambda)$ is Weyl almost finite and we have $\sum_{\bw\in W^+} a_\bw T_{\bw}=\sum_{\lambda\in Y^+}h_\lambda T_\lambda$.  This proves the first part of the corollary.

Let now $a=\sum_{\bu\in W^+}a_\bu T_\bu,b=\sum_{\bv\in W^+} b_\bv T_\bv\in \HBL$. Let $c=a b\in \HBL$. Write $a=\sum_{\lambda\in Y^+}h_\lambda T_\lambda$ and $b=\sum_{\mu\in Y^+} j_\mu T_\mu$, with $(h_\lambda),(j_\mu)\in (\HCW)^{Y^+}_{\WAF}$.  For $\bu=\lambda.u\in W^+$, we set \[\tau(\bu)=u^{-1}(\lambda)\in Y^+.\] For $\lambda\in Y^+$, we set \[W^+(\lambda)=\{w(\lambda).w\mid w\in W_0\}\] Write $c=\sum_{\bw\in W^+} c_{\bw}T_{\bw}$. Fix $\bw\in W^+$. We want to prove that $\{(\bu,\bv)\in (W^+)^2\mid \coeff^T_{\bw}(a_\bu T_{\bu} b_{\bv} T_{\bv})\neq 0\}$ is finite. For this we would like to use Theorem~\ref{c_T_version_completed_algebra}. A difficulty is that  if $\bu,\bv\in W^+$, $\lambda=\tau(\bu)$, $\mu=\tau(\bv)$ and  $\nu=\tau(\bw)$, we might have $\coeff^T_{\bw}(a_{\bu}T_{\bu} b_{\bv} T_{\bv})\neq 0$  but $\coeff^T_{\bw}(\pr_\nu^T(h_\lambda T_\lambda j_\mu T_\mu))= 0$, where $\pr_\nu:\HBL\rightarrow \cH$ is defined as in Definition~\ref{d_Z_T_support}. To circumvent this problem, we apply Theorem~\ref{c_T_version_completed_algebra} to variants $\tilde{a},\tilde{b}$ of $a$ and $b$, where $\tilde{a},\tilde{b}\in \HBL_{\tilde{\cR}}$, where $\tilde{\cR}$ is a ring containing $\cR$.

Let $\tilde{\cR}=\cR[U,V]$, where $U$ and $V$ are indeterminates. Let $\sigma:W^+\rightarrow \N$ be an injective map. For $\bu\in W^+$, set $\tilde{a}_{\bu}=a_{\bu} U^{\sigma(\bu)}$ and $\tilde{b}_{\bv}=b_{\bv} V^{\sigma(\bv)}$. Set $\tilde{a}=\sum_{\bu\in W^+} \tilde{a}_{\bu} T_{\bu}\in \HBL_{\tilde{\cR}}$ and $\tilde{b}=\sum_{\bv\in W^+} \tilde{b}_{\bv} T_{\bv}\in \HBL_{\tilde{\cR}}$. Write $\tilde{a}=\sum_{\lambda\in Y^+}\tilde{h}_\lambda T_\lambda$ and 
$\tilde{b}=\sum_{\mu\in Y^+} \tilde{j}_\mu T_\mu$,
with $(\tilde{h}_\lambda), (\tilde{j}_\mu)\in (\cH_{W_0,\tilde{\cR}})^{Y^+}_{\WAF}$. 

Let $\bu,\bv\in W^+$. Assume that $\coeff^T_{\bw}(a_\bu T_{\bu} b_{\bv} T_{\bv})\neq 0$. Then \begin{equation}\label{e_coeff_product}
    \sum_{\bu'\in W^+(\tau(\bu)), \bv'\in W^+(\tau(\bv))}\coeff^T_{\bw}(\tilde{a}_{\bu'} T_{\bu'} \tilde{b}_{\bv'} T_{\bv'})\neq 0,
\end{equation} since the  $U^{\sigma(\bu)}V^{\sigma(\bv)}$ component of the left hand side is $\coeff^T_{\bw}(a_{\bu} T_{\bu} b_{\bv}T_{\bv})\neq 0$. Now if $\nu=\tau(\bw)$, then $\tilde{h}_{\tau(\bu)}T_{\tau(\bu)}=\sum_{\bu'\in W^+(\tau(\bu))} \tilde{a}_{\bu'} T_{\bu'}$ and $\tilde{j}_{\tau(\bv)}T_{\tau(\bv)}=\sum_{\bv'\in W^+(\tau(\bv))} \tilde{b}_{\bv'}T_{\bv'}$. Therefore  \[\pr^T_{\nu}(\tilde{h}_{\tau(\bu)}T_{\tau(\bu)}\tilde{j}_{\tau(\bv)} T_{\tau(\bv)})=\sum_{\bu'\in W^+(\tau(\bu)),\bv'\in W^+(\tau(\bv))}\pr^T_{\nu}(\tilde{a}_{\bu'} T_{\bu'} \tilde{b}_{\bv'}T_{\bv'}). \]

Moreover \begin{align*}\coeff^T_{\bw}(\tilde{h}_{\tau(\bu)}T_{\tau(\bu)}\tilde{j}_{\tau(\bv)}T_{\tau(\bv)})&=\coeff^T_{\bw}(\pr_\nu^T(\tilde{h}_{\tau(\bu)} T_{\tau(\bu)} \tilde{j}_{\tau(\bv)} T_{\tau(\bv)}))\\&=\sum_{\bu'\in W^+(\tau(\bu)), \bv'\in W^+(\tau(\bv))}\coeff^T_{\bw}(\tilde{a}_{\bu'} T_{\bu'} \tilde{b}_{\bv'} T_{\bv'})\\&\neq 0 \text{ by \eqref{e_coeff_product}},\end{align*}  and in particular $\pr_\nu^T(\tilde{h}_{\tau(\bu))} T_{\tau(\bu)} \tilde{j}_{\tau(\bv)} T_{\tau(\bv)})\neq 0$.  Thus we proved that for every $\bu,\bv\in W^+$, we have:
\[\coeff^T_{\bw}(a_{\bu}T_{\bu} b_{\bv}T_{\bv})\neq 0 \Rightarrow \coeff^T_{\tau(\bw),\HCW}(\tilde{h}_{\tau(\bu)}T_{\bu} \tilde{j}_{\tau(\bv)} T_{\bv})\neq 0. \]

By Theorem~\ref{c_T_version_completed_algebra} we deduce that there exist finite sets $F_1',F_2'\subset Y^+$ such that for all $\bu,\bv\in W^+$, we have: \begin{equation}\label{e_condition_nonzero1} \coeff^T_{\tau(\bw),\HCW}(a_{\bu} T_{\bu} b_{\bv}T_{\bv})\neq 0\Rightarrow \tau(\bu)\in F_1',\tau(\bv)\in  F_2'.\end{equation} Set $F_i=\{(w(\lambda).w)\mid w\in W_0,\lambda\in F_i\}=\bigsqcup_{\lambda\in F_i}W^+( \lambda)$, for $i\in \{1,2\}$. By condition (3), $F_1$ and $F_2$ are finite and we have, for all $\bu,\bv\in W^+$: \begin{equation}\label{e_condition_nonzero2} \coeff^T_{\bw}(a_{\bu} T_{\bu} b_{\bv}T_{\bv})\neq 0 \Rightarrow \bu\in F_1, \bv\in  F_2. \end{equation}

We have \begin{align*}
    \coeff^T_{\bw}(ab) &= \coeff^T_{\bw}(\sum_{\lambda\in Y^+,\mu\in Y^+}h_\lambda T_\lambda j_\mu T_\mu)\\ &= \coeff^T_{\bw}(\pr_{\nu}(\sum_{\lambda,\mu\in Y^+} h_\lambda T_\lambda j_\mu T_\mu))\\
    &= \coeff^T_{\bw}(\pr_\nu(\sum_{\lambda\in F_1',\mu\in F_2'}h_\lambda T_\lambda j_\mu T_\mu))\\
    &=\sum_{\lambda\in F_1',\mu\in F_2'}\coeff^T_{\bw}(\pr_\nu(h_\lambda T_\lambda j_\mu T_\mu))\\  
    &=\sum_{\bu\in F_1,\bv\in F_2}\coeff^T_\bw(\pr_\nu(a_{\bu}T_{\bu} b_{\bv}T_{\bv}))  \\
    &=  \sum_{\bu\in F_1,\bv\in F_2}\coeff^T_\bw(a_{\bu}T_{\bu} b_{\bv}T_{\bv}).
\end{align*}
 Moreover, write $\tilde{a} \tilde{b}=\sum_{\tau\in Y^+} \tilde{k}_\tau T_\tau$, with $(\tilde{k}_\tau)\in (\cH_{W_0,\tilde{\cR}})^{Y^+}_{\WAF}$. Let $\nu\in Y^+$ and \[E=\{\bw\in W^+(\nu)\mid \exists (\bu,\bv)\in (W^+)^2,\coeff^T_{\bw}(a_\bu T_\bu b_\bv T_{\bv})\neq 0\}.\] Let $\bw\in W^+.$ Assume that $\bw\in E$ and take $\bu,\bv\in W^+$ such that $\coeff^T_{\bw}(a_\bu T_\bu b_\bv T_\bv)\neq 0$. We have \[\tilde{k}_\nu T_\nu=\sum_{\bw'\in W^+(\nu),\bu',\bv'\in W^+} \coeff^T_{\bw'}(\tilde{a}_{\bu}\tilde{b}_{\bv} T_{\bu} T_{\bv})T_{\bw'}. \] and the $U^{\sigma(\bu)}V^{\sigma(\bv)}$ component of $\coeff^T_{\bw}(\tilde{k}_\nu T_\nu)$ is $\coeff^T_{\bw}(a_\bu T_\bu b_{\bv} T_{\bv})\neq 0$.  Therefore $E$ is contained in $\{\bw'\in W^+(\nu)\mid \coeff^T_{\bw'}(\tilde{k}_\nu T_\nu)\neq 0\}$, which is finite. 
\end{proof}

For example $\sum_{w\in W_0} T_w\notin \HBL$, except if $W_0$ is finite. When $\cH$ is associated with a Kac-Moody group $G=\mathbf{G}(\cK)$, where $\cK$ is a non-Archimedean local field, then if $K=\mathbf{G}(\cO)$ and $I$ is the Iwahori subgroup, then we have $K=\bigsqcup_{w\in W_0} IwI$ and thus $\sum_{w\in W_0} T_w$ would correspond to $\One_{K}$.

\subsection{Comparison of $\HBL$ with the algebra constructed by Abdellatif and H\'ebert}\label{ss_comparison_HBL_tilde_cH}

We now compare $\HBL$ and the ``completed Iwahori-Hecke algebra $\widetilde{\cH}$'' defined in \cite{abdellatif2019completed} and \cite{abdellatif2022erratum}. 

 Let $\cR[\![Y^+]\!]$ be the subalgebra of $\RY$ consisting of the elements whose support is contained in $Y^+$. 

\begin{Definition}\label{d_Wv_AF}
 A subset $E$ of $Y^+$ is called \textbf{$W_0$-almost finite} if $u.E$ is almost finite, for every $u\in W_0$. Let $\sA$ be the subalgebra of $\cR[\![Y^+]\!]$ consisting of the elements having support contained in $Y^+$ and  $W_0$-almost finite support. The completed Iwahori-Hecke  defined in \cite{abdellatif2022erratum} is the algebra $\tilde{\cH}=\bigoplus_{w\in W_0} T_w \sA=\bigoplus_{w\in W_0}\sA T_w$, equipped with the product defined in \cite[Theorem 3.12]{abdellatif2022erratum}. 
\end{Definition}

In general, we have $\RY_{\WAF}\subset \sA\subset \RY$. The elements of $\widetilde{\cH}$ involve finitely many elements of $W_0$ and thus when $W_0$ is infinite, $\widetilde{\cH}$ is not  contained in $\HBL$. For example $\tilde{\lambda}\in C^v_f\cap Y$, then $\sum_{w\in W_0}T_w Z^{w(\tilde{\lambda})}\in \HBL$ but it is not an element of $\widetilde{\cH}$. By Example~\ref{ex_Weyl_almost_finite_sets}, if the Kac-Moody matrix is indefinite of size $2$, then every subset of $Y^+$ is almost finite. In particular, we have $\RY_\WAF=\sA=\cR[\![Y^+]\!]$ and thus $\widetilde{\cH}\subset \HBL$.

In general however,   $\widetilde{\cH}$ and $\widehat{\cH}$ are not comparable for $\subset$, by Lemma~\ref{lemCounter_example_u-almost_finite_not_Weyl_almost_finite} below.

\begin{Lemma}\label{lemCounter_example_u-almost_finite_not_Weyl_almost_finite}
We assume that $\A$ is associated with an indecomposable affine Kac-Moody matrix. Let $\lambda\in Y\cap C^v_f$ and $i\in I_A$. Let $E=\{w(\lambda)+\alpha_i^\vee \mid w\in W_0\}$.  Then $E$ is $u$-almost finite for every $u\in W_0$ but $E$ is not Weyl almost finite.
\end{Lemma}

\begin{proof}
 Then by \cite[Proposition 5.8]{kac1994infinite}, there exists a $W_0$-invariant map $\delta$ in the dual $X$ of $Y$ such that $\cT=\delta^{-1}(\R_{>0})\cup \bigcap_{i\in {I_A}}\ker(\alpha_i)$. Then $\delta(\alpha_i^\vee)=0$, for all $i\in {I_A}$.

We have $\delta(E)=\delta(\{\lambda\})>0$ and thus $E\subset Y^+$. Let $u\in W_0$. Then by \eqref{e_GR2.4}, $\mu\leq_{Q^\vee} \lambda+u(\alpha_i^\vee)$, for all $\mu\in u.E$ and thus $E$ is $u$-almost finite. 

We have $W_0\cdot E=\{uw(\lambda)-u(\alpha_i^\vee)\mid u,w\in W_0\}=\{v(\lambda)+u(\alpha_i^\vee)\mid u,v\in W_0\}$. Suppose that $E$ is Weyl almost finite. Then $W_0\cdot E$ and thus $\{\lambda+u(\alpha_i^\vee)\mid u\in W_0\}$ is almost finite.  Therefore there exists a finite set $F\subset Y^+$ such that for all $u\in W_0$, there exists $\nu_u\in F$ such that $\lambda+u(\alpha_i^\vee)\leq_{Q^\vee} \nu_u$. Therefore $u(\alpha_i^\vee)\leq_{Q^\vee} \nu_u-\lambda$ for all $u\in W_0$. Thus there exists $u\in W_0$ such that $u(\alpha_i^\vee)$ is maximal for $\leq_{Q^\vee}$ among $W_0\cdot \alpha_i^\vee$. Using \cite[Lemma 4.8]{abdellatif2019completed} we deduce $\alpha_i^\vee\in \cT$. But then we reach a contradiction since $\alpha_i^\vee\notin \delta^{-1}(\R_{>0})\cup \bigcap_{j\in {I_A}}\ker(\alpha_i)$. Thus $E$ is not Weyl almost finite.
\end{proof}

It seems that $\HBL$ is more adapted than $\widetilde{\HC}$ to handle $T$-series, see for example \eqref{e_T_series_not_in_tilde_H}. We do not know if every element of $\widetilde{\HC}$ admits a $T$-series expansion. However we can study the $T$-expansion of $\widetilde{\HC}\cap \HBL=\bigoplus_{w\in W_0} \Tb_w \RY_\WAF$. By Lemma~\ref{lemNecessity_infinite_length_IM_BL}, it seems that $\widetilde{\HC}\cap \HBL$ is complicated to describe in terms of $T$-series.

For $h=\sum_{w\in W_0} a_w T_w\in \HCW$, we set $\ell(h)=\max\{\ell(w)\mid w\in W_0, a_w\neq 0\}$.

\begin{Lemma}\label{lemNecessity_infinite_length_IM_BL}
Assume $q\neq 1$.  Let $i_1,i_2\in {I_A}$. We assume that $r_{i_1}r_{i_2}$ has infinite order. For $k\in {\Z_{\geq 0}}$, we set $w_k= \ldots r_{i_1}r_{i_2}r_{i_1}$, where the product has $k$ factors. Then  $\ell(T_{w_k^{-1}}^{-1} T_{w_k}^{-1})=2k-1$, for $k\in \Z_{> 0}$. 
\end{Lemma}

\begin{proof}
Let $k\in {\Z_{\geq 0}}$ and write $w_k=r_j\ldots r_{i_2}r_{i_1}$, with $j\in \{i_1,i_2\}$. Then \[\Tb_{w_k^{-1}}^{-1} \Tb_{w_k}^{-1}=\Tb_{j}^{-1}\ldots \Tb_{i_2}^{-1} \Tb_{i_1}^{-2} \Tb_{i_2}^{-1}\ldots \Tb_j^{-1}.\] We have $\Tb_{i}^{-2},\Tb_i^{-1}\in (\cR \Tb_{i}\oplus \cR)\setminus \cR$ for all $i\in {I_A}$. For $i\in {I_A}$, we have $\Tb_i^{-1}=q^{-1}\Tb_i+q^{-1}-1$ and $\Tb_i^{-2}=a\Tb_i+b$, with $a\in \cR\setminus \{0\}$. Then we have $\Tb_{j}^{-1}\ldots \Tb_{i_2}^{-1} \Tb_{i_1}^{-2} \Tb_{i_2}^{-1}\ldots \Tb_j^{-1}=q^{-2(k-1)}a\Tb_{r_j\ldots r_{i_2} r_{i_1}r_{i_2}\ldots r_j}+h$, for some $h\in \HC_{W_0}$ such that $\ell(h)\leq 2(k-1)$, which proves the lemma.
\end{proof}

Assume for example that the Kac-Moody matrix defining $\cS$ has size $2$ and that $W_0$ is infinite. Let  $\lambda\in C^v_f\cap Y$. Then \begin{equation}\label{e_T_series_not_in_tilde_H}
    \sum_{w\in W_0} T_{w(\lambda)}\in \HBL\setminus \widetilde{\cH}.
\end{equation} Indeed, by Proposition~\ref{P_BL_support_IM_basis}, we have $\coeff^Z_{w(\lambda)}(T_{w(\lambda)})=q^{-\ell(w)}\delta^{-1/2}(\lambda)T_{w}T_{w^{-1}}$, for $w\in W_0$. By Lemma~\ref{lemNecessity_infinite_length_IM_BL}, we deduce the claim. 

\subsection{Swapping left and right does not work}\label{ss_definitionr_right}

We defined the completed Iwahori-Hecke algebra as the set of series $\sum_{\lambda\in Y^+}h_\lambda Z^\lambda$, with $(h_\lambda)\in (\HCW)^{Y^+}_{\WAF}$. We saw that this is exactly the set of series $\sum_{\lambda\in Y^+} h_\lambda T_\lambda$, with $(h_\lambda)\in (\HCW)^{Y^+}_{\WAF}$. 

It seems also natural to ask whether we could also consider series of the form $\sum_{\lambda\in Y^+} \Zb^\lambda  h_\lambda$ or $\sum_{\lambda\in Y^+} T_\lambda h_\lambda$, for $(h_\lambda)\in (\HCW)^{Y^+}_{\WAF}$. However, as we shall see below, if we did so, then we could not pass from $Z$-series to $T$-series, by Proposition~\ref{p_non-zero_coefficient}. 

For $\nu\in Y^+$, define $\coeff_\nu^{T,\mathrm{right}}:\HC\rightarrow \cR$ by  $\coeff_\nu^{T,\mathrm{right}}(\sum_{\lambda\in Y^+} T_\lambda h_\lambda)=h_\nu$, for $(h_\lambda)\in (\HCW)^{(Y^+)}$, which is well defined by Remark~\ref{r_invertibility_Z_T}. 

\begin{Lemma}\label{l_interval_Bruhat}
    Let $w\in W_0$ and $i\in {I_A}$ be such that $r_iw>w$. Then $[1,r_iw]=[1,w]\cup r_i [1,w]$. Moreover for $u\in [1,r_iw]$, we have: \[(r_iu<u\Rightarrow r_iu\leq w), (r_iu>u\Rightarrow u\leq w). \]
\end{Lemma}

\begin{proof}
    Let $u\in [1,w]$ and $u'=r_iu$. If $u'<u$, then $u'<u<w<r_iw$. If $u'>u$, then by \cite[Proposition 2.2.7]{bjorner2005combinatorics} (applied with $(u,r_iw)$) we have $r_iu\leq r_iw$ and thus $u'\in [1,r_iw]$. Consequently $[1,w]\cup r_i[1,w]\subset [1,r_iw]$. 
    
    Conversely, take $u\in [1,r_iw]$. Set $w'=r_iw$ and $u'=r_iu$. First assume $u'<u$. Then by \cite[Proposition 2.2.7]{bjorner2005combinatorics} we have $u'\leq r_i w'=w$ and thus $u=r_iu'\in r_i [1,w]$. Assume now $u'>u$. Then by \cite[Proposition 2.2.7]{bjorner2005combinatorics}, we have $u\leq r_iw'=w$ and thus $u\in [1,w]$. Therefore $[1,w]\cup r_i[1,w]=[1,r_iw]$. 
\end{proof}

\begin{Lemma}\label{l_computation_inverse}
Let $x$ be an indeterminate. For $u,w\in W_0$ such that $u\leq w$, define $a_{u,w}(x)\in \Z[x]$ as follows. Set $a_{1,1}=1$ and for $w\in W_0\setminus \{1\}$, $u\in [1,w]$ and $i\in {I_A}$ such that $v<w$, where $v=r_iw$.  Set: \[a_{u,w}(x)=\begin{cases}a_{r_iu,v}(x) &\text{ if }r_iu<u \\ xa_{r_iu,v}(x)-(x-1)a_{u,v}(x) \ &\text{ if }r_i u>u   
    \end{cases}.\] Then   we have  $q^{\ell(w)}T_{w}^{-1}=\sum_{u\in [1,w]}a_{u,w}(q)T_u$ and $\deg a_{u,w}=\ell(w)-\ell(u)$ for all $u,w\in W_0$ such that $u\leq w$. In particular when $u\leq w$, we have $a_{u,w}\neq 0$. 
\end{Lemma}

\begin{proof}
 If $w=1$, the result is clear. Let $w\in W_0\setminus\{1\}$ and take $i\in {I_A}$ such that $v<w$, where $v=r_iw$.  Assume  that $q^{\ell(v)}T_v^{-1}=\sum_{u\in [1,v]}a_{u,v}(q)T_u$. We have \[q^{\ell(w)}T_w^{-1}=qT_i^{-1} q^{\ell(v)}T_v^{-1}=(T_i-(q-1))q^{\ell(v)}T_v^{-1}.\]

Let $E=\{u\in [1,v]\mid r_iu<u\}$ and $F=[1,v]\setminus E$. Write $q^{\ell(w)}T_w^{-1}=\sum_{u\in [1,w]}b_{u} T_u$.
    We have \begin{align*}
        q T_i^{-1} \sum_{u\leq v}a_{u,v}(q)T_u &= qT_i^{-1}\sum_{u\in E} a_{u,v}(q)T_u+qT_i^{-1} \sum_{u\in F}a_{u,v}(q) T_u\\
        &=qT_i^{-1} \sum_{u\in E} a_{u,v}(q) T_i T_{r_iu}+(T_i-(q-1))\sum_{u\in F}a_{u,v}(q)T_u\\
        &=\sum_{u\in E}qa_{u,v}(q) T_{r_iu}+\sum_{u\in F}a_{u,v}(q)T_{r_iu}-\sum_{u\in F}(q-1)a_{u,v}(q)T_u.
    \end{align*}

    Let $u\in [1,w]$. Then by Lemma~\ref{l_interval_Bruhat}, we have \[(u\in r_i(E)\Leftrightarrow r_iu>u), (u\in r_i(F)\Leftrightarrow r_iu<u), (u\in F \Leftrightarrow r_iu>u).\] Therefore 

  \[b_u=\begin{cases}a_{r_iu,v}(q) &\text{ if }r_iu<u \\ qa_{r_iu,v}(q)-(q-1)a_{u,v}(q) \ &\text{ if }r_i u>u   
    \end{cases}\] and consequently $b_u=a_{u,w}(q)$ for $u\in [1,w]$. By induction we deduce the lemma.
\end{proof}

Let $h\in \cH$. According to Proposition~\ref{p_HCW_bimodule} we can write $h=\sum_{\lambda\in Y^+} T_\lambda h_\lambda$. For $\lambda\in Y^+$, we set $\coeff^{T,\mathrm{right}}_{\lambda}(h)=h_\lambda$.

\begin{Proposition}\label{p_non-zero_coefficient}
We assume that $q\in \cR^\times$ is  such that for every $u,w\in W_0$ such that $u\leq w$, we have $a_{u,w}(q)\neq 0$, with the notation of Lemma~\ref{l_computation_inverse} (when $\cR$ is not countable, such a $q$ always exist).    Let $\lambda\in Y\cap C^v_f$.   Then for all $w\in W_0$, $\coeff_\lambda^{T,\mathrm{right}}(\Zb^{w(\lambda)})$ is nonzero. 
\end{Proposition}

\begin{proof}
By Lemma~\ref{l_pi_lambda_Z_w_lambda},  there exists $h_w\in \HC_{<^{++}\lambda}$ such that
 \[\Zb^{w(\lambda)}=q^{\ell(w)}\Tb_{w^{-1}}^{-1} \Zb^\lambda \Tb_w^{-1}+h_w=q^{\ell(w)}\delta^{-1/2}(\lambda)\Tb_{w^{-1}}^{-1} T_\lambda  \Tb_w^{-1}+h_w.\]
 
 Using \cite[Proposition 4.1]{bardy2016iwahori} and Lemma~\ref{l_computation_inverse}, we deduce that \[Z^{w(\lambda)}=q^{\ell(w)} \delta^{-1/2}(\lambda) \sum_{u\in [1,w^{-1}]}a_{u,w^{-1}}T_{u(\lambda)\cdot u} T_{w}^{-1}+h_w.\] By Proposition~\ref{p_HCW_bimodule}), we have $\coeff^{T,\mathrm{right}}_{\lambda}(h_w)=0$.  By Proposition~\ref{p_HCW_bimodule} and by assumption on $q$, we have $\coeff^{T,\mathrm{right}}_{\lambda}(Z^{w(\lambda)})=q^{\ell(w)}\delta^{-1/2}(\lambda)a_{1,w^{-1}}(q)T_{w}^{-1}\neq 0$, which proves the lemma.
\end{proof}

\begin{appendix}

  \newcommand{\Hbkp}{\cH_{\mathrm{BKP}}}
  \newcommand{\Hbgr}{\cH_{\mathrm{BGR}}}
  
\section{Comparing Braverman-Kazhdan-Patnaik conventions with Bardy-Panse-Gaussent-Rousseau conventions}
\label{sec:appendix}

The goal of this appendix is to write down an explicit anti-isomorphism between the Iwahori-Hecke algebras considered by Braverman-Kazhdan-Patnaik and Bardy-Panse-Gaussent-Rousseau that is compatible with $T$-basis (\S \ref{sec:matching-BKP-and-BGR-bases}).

As in the main body of the paper, let $\mathbf{G}$ be a split Kac-Moody group, $G=\mathbf{G}(\cK)$, where $\cK$ is a non-Archimedean local field, let $I$ denote the Iwahori subgroup, and let $K$ denote the spherical subgroup. Let $\bT$ denote the maximal torus of $\mathbf{G}$, and let $T = \bT(\cK)$.

Let $\pi$ denote a uniformizer of the ring of integers $\cO$ of $\cK$. We identify $Y$ with the cocharacter lattice of $\bT$. Given a cocharacter $\lambda \in Y$, consider the induced homomorphism $\cK^\times \rightarrow T$. We write $\pi^\lambda$ for the image of $\pi$ under this homomorphism.

We define the positive Cartan semigroup $G^{+}$ and the negative Cartan semigroup $G^-$ by
\begin{align*}
  G^+ = \bigcup_{\lambda \in Y^{++}}  K \pi^\lambda K \\
  G^- = \bigcup_{\lambda \in Y^{++}}  K \pi^{-\lambda} K
\end{align*}
Let $\cR$ be the base ring as in the body of the paper, and consider the two Iwahori-Hecke algebras
\begin{align}
  \label{eq:g-plus-hecke-alg}
  \cH^+ = \cR_{\mathrm{fin}}[I \backslash G^+ \!\!/I] 
\end{align}
and
\begin{align}
  \cH^- = \cR_{\mathrm{fin}}[I \backslash G^-\!\!/I], 
\end{align}
where if $E$ is a set, $\cR_{\mathrm{fin}}[E]$ denotes the set of functions from $E$ to $\cR$ with finite support. 

The semigroup denoted $G^{\geq 0}$ in the main body of the paper is equal to $G^-$. The reason is the minus sign that appears in the action of  the torus $T$ on the standard apartment $\AA$ in the definition of the masure (see \S \ref{subsec:intro-conventions}). Correspondingly, the Iwahori-Hecke algebra $\cH$ in the main body of the paper is equal to $\cH^-$. 

The inverse anti-involution $\iota : G \rightarrow G$ restricts to mutually inverse anti-isomorphisms $\iota : G^+ \rightarrow G^-$ and $\iota : G^- \rightarrow G^+$. This induces mutually inverse anti-isomorphisms:
\begin{align}
  \iota : \cH^+ \rightarrow \cH^- \\
  \iota : \cH^- \rightarrow \cH^+ 
\end{align}

Let $W_0$ be the (vectorial) Weyl group, and let  $Y^+$ denote the integral Tits cone. Then we know
\begin{align}
I \backslash G^+ \!\!/I = Y^+ \rtimes W_0
\end{align}
and
\begin{align}
I \backslash G^- \!\!/I = Y^- \rtimes W_0,
\end{align}
where $Y^-=-Y^+$.

For any $x = \pi^\mu w \in Y^+ \rtimes W_0 \cup Y^- \rtimes W_0$, write:
\begin{align}
  \TT_{x} = \TT_{\pi^\mu w} = \One_{I \pi^\mu wI}
\end{align}
Then we have:
\begin{align}
  \iota(\TT_{x}) = \TT_{x^{-1}}
\end{align}

\subsection{The BKP and BGR Hecke algebras}

For $ \pi^{\mu} w \in Y^+\rtimes W_0 $, the elements $\TT_{\pi^{\mu} w}$ agree with the basis (called $T_{ \pi^{\mu} w}$ there) considered by Braverman, Kazhdan, and Patnaik \cite{braverman2014affine}. To summarize, we have the following.
\begin{Proposition}
  \label{claim:bkp-basis}
  The Iwahori-Hecke algebra $\Hbkp$ considered by Braverman, Kazhdan, and Patnaik is naturally isomorphic to $\cH^+$.
  For  $ \pi^{\mu} w \in Y^+ \rtimes W_0$, we have:
  \begin{align}
   T_{ \pi^{\mu} w} = \TT_{ \pi^{\mu} w} 
  \end{align}
 where $T_{ \pi^{\mu} w}$ denotes the basis they defined in \cite{braverman2014affine}.
\end{Proposition}
We will call $\Hbkp$ the BKP Hecke algebra, and call the above indexing the BKP conventions. Note that the proposition follows by definition because Braverman, Kazhdan, and Patnaik define their Iwahori-Hecke algebra to be \eqref{eq:g-plus-hecke-alg}. In particular, the indexing matches exactly.

For $ \pi^{-\mu} w \in Y^- \rtimes W_0$, the elements $\TT_{ \pi^{-\mu} w}$ agree with the basis (called $T_{\mu\cdot w}$ there) considered by Bardy-Panse, Gaussent, and Rousseau \cite{bardy2016iwahori}. To summarize, we have the following.
\begin{Proposition}
  \label{claim:bgr-basis}
  The Iwahori-Hecke algebra $\Hbgr$ considered by Bardy-Panse, Gaussent, and Rousseau is naturally isomorphic to $\cH^-$.
  For  $ \pi^{-\mu} w \in Y^- \rtimes W_0$, we have:
  \begin{align}
   T_{\mu\cdot w} = \TT_{ \pi^{-\mu} w} 
  \end{align}
 where $T_{\mu\cdot w}$ is the basis they defined in \cite{bardy2016iwahori}.
\end{Proposition}
We will call $\Hbgr$ the BGR Hecke algebra, and call the above indexing the BGR conventions. This is a more subtle convention to match since Bardy-Panse, Gaussent, and Rousseau use the masure to define the Iwahori-Hecke algebra. However, we will not directly analyze the masure definition. Instead, we will use the algebraic characterization of the BGR Hecke algebra and its $T$-basis using the Iwahori-Matsumoto relations and the Bernstein-Lusztig presentation to match conventions.

\subsection{Algebraic characterization of the algebras and $T$-basis}

\subsubsection{Iwahori-Matsumoto relations}

To start we will recall the Iwahori-Matsumoto relations in this setting. 

\begin{Proposition}
For $\pi^\mu w \in Y^+ \rtimes W_0 \cup Y^- \rtimes W_0$ and $i \in {I_A}$, we have:
\begin{align}
  \label{eq:iwahori-matsumoto-relations}
 \TT_{\pi^\mu w r_i} =  
  \begin{cases}
    \TT_{\pi^\mu w} \TT_{r_i} &\textif \langle \mu, w(\alpha_i) \rangle > 0  \textor \left[\langle \mu, w(\alpha_i) \rangle = 0 \textand \ell(wr_i) > \ell(w) \right] \\
\TT_{\pi^\mu w} \TT^{-1}_{r_i} &\textif \langle \mu, w(\alpha_i) \rangle < 0  \textor \left[ \langle \mu, w(\alpha_i) \rangle = 0 \textand \ell(wr_i) < \ell(w) \right]
  \end{cases}
\end{align}
\end{Proposition}

We use the identification $\cH^+ = \Hbkp$. For $Y^+ \rtimes W_0$, this formula is exactly \cite[Theorem 3.1]{muthiah2018iwahori}. For $Y^- \rtimes W_0$, we apply the anti-isomorphism $\iota$ to \cite[3.17]{muthiah2018iwahori}, which is a left-handed version of \cite[Theorem 3.1]{muthiah2018iwahori}. 

Observe, that a version of the Iwahori-Matsumoto relations is also proved by Bardy-Panse, Gaussent, and Rousseau (\cite[Proposition 4.1]{bardy2016iwahori}) for $\Hbgr$. One can immediately verify that relations \eqref{eq:iwahori-matsumoto-relations} exactly match those relations under the isomorphism of Proposition \ref{claim:bgr-basis}. In particular, Proposition \ref{claim:bgr-basis} is compatible with Iwahori-Matsumoto relations.

\subsubsection{Bernstein-Lusztig generators}

Recall that we can consider the Garland-Grojnowski version of the Iwahori-Hecke algebra \cite{garland1995affineheckealgebrasassociated}. This is generated by a copy of the (vectorial) Hecke algebra, which is spanned by $\TT_{w}$ for $w \in W$. We also have a lattice $\left\{\Theta^\mu\right\}_{\mu \in Y}$ of elements isomorphic to the coweight lattice $Y$. The Hecke algebra interacts with the lattice via the Bernstein relation. For all $i\in {I_A}$ and  $\mu \in Y$, we have: 
\begin{align}
  \label{eq:2}
 \TT_i \Theta^\mu - \Theta^{r_i(\mu)} \TT_i = (q-1) \frac{ \Theta^\mu - \Theta^{r_i(\mu)}}{1 - \Theta^{-\alpha_i^\vee}} 
\end{align}
Let us write
\begin{align}
  \label{eq:3}
\cH_{GG,\Theta} = \langle \TT_w, \Theta^\mu \suchthat w \in W_0, \mu \in Y \rangle
\end{align}
for this algebra.

Let us write another copy of this algebra generated by $T_{w}$ for $w \in W$ and a lattice $\left\{ Z^{\mu} \right\}_{\mu \in Y}$. They interact via the same Bernstein relation:
\begin{align}
  \label{eq:4}
  T_i Z^\mu - Z^{r_i(\mu)} T_i = (q-1) \frac{Z^{\mu} - Z^{r_i(\mu)}}{1 - Z^{-\alpha_i^\vee}}
\end{align}
Let us write
\begin{align}
  \label{eq:3}
\cH_{GG,Z} = \langle T_w, Z^\mu \suchthat w \in W, \mu \in Y \rangle
\end{align}
for this algebra.

\subsubsection{Restriction to Tits cone}

Braverman, Kazhdan, and Patnaik proved that the map $\TT_i \mapsto \TT_i$ and $\TT_{\pi^\lambda} \mapsto q^{\htt(\lambda)} \Theta^\lambda$ 
for $\lambda$ dominant defines an injective homomorphism:
\begin{align}
  \label{eq:5}
  \cH^+ = \Hbkp \hookrightarrow \cH_{GG,\Theta}
\end{align}
Furthermore, the image is exactly the subalgebra $\langle \TT_w, \Theta^\mu \suchthat w \in W_0, \mu \in Y^+ \rangle$.

Bardy-Panse, Gaussent, Rousseau proved that the map $T_i \mapsto T_i$ and $T_{\lambda} \mapsto q^{\htt(\lambda)}Z^\lambda$
for $\lambda$ dominant defines an 
injective homomorphism:
\begin{align}
  \label{eq:7}
  \Hbgr \hookrightarrow \cH_{GG,Z}
\end{align}
Furthermore, the image is exactly the subalgebra $\langle \TT_w, Z^\mu \suchthat w \in W_0, \mu \in Y^+ \rangle$.

\subsubsection{Compatibility with the anti-isomorphism $\iota$}

Define $\iota(T_i) = T_i$, and $\iota(Z^\mu) = \Theta^{\mu}$. It is clear that this extends to an anti-isomorphism:
\begin{align}
  \label{eq:8}
  \iota : \cH_{GG,Z} \rightarrow \cH_{GG,\Theta}
\end{align}
Moreover, it is clear that there is a unique embedding $\cH^- \hookrightarrow \cH_{GG,Z}$ such that the following diagram commutes:
\begin{equation}
  \label{eq:6}
 \begin{tikzcd}
\cH^- \arrow[r, "\iota"] \arrow[d, hook] & \cH^+ \arrow[d, hook] \\
{\cH_{GG,Z}} \arrow[r, "\iota"]          & {\cH_{GG,\Theta}}    
\end{tikzcd} 
\end{equation}
The image of $\cH^- \hookrightarrow \cH_{GG,Z}$ exactly agrees with the image of $\Hbgr$, so we can identify $\cH^- = \Hbgr$. Moreover, this identification is compatible with the Bernstein-Lusztig generators, the Iwahori-Matsumoto relations, and the anti-dominant $T$-basis elements of $\cH^-$ under the indexing stated in Proposition \ref{claim:bgr-basis}. Therefore, it is compatible for all $T$-basis elements. We therefore verify Proposition \ref{claim:bgr-basis}.

\subsection{Matching BKP and BGR bases}
\label{sec:matching-BKP-and-BGR-bases}
Now that we have explained Propositions \ref{claim:bkp-basis} and \ref{claim:bgr-basis}, and can identify $\Hbkp = \cH^+$ and $\Hbgr = \cH^-$, we can identify these two algebras and their $T$-bases.  The anti-isomorphism $\iota: \cH^{-} \rightarrow \cH^{+}$ therefore gives rise to an anti-isomophism $\iota : \Hbgr \rightarrow \Hbkp$. Fix $\mu \in Y^+$ and $w \in W$. Applying Propositions \ref{claim:bkp-basis} and \ref{claim:bgr-basis} we have
\begin{align}
  \label{eq:1}
\iota(T_{\mu\cdot w}) = T_{w^{-1} \pi^{\mu}}
\end{align} 
where the left hand side is written in BGR notation and the right hand side is written in BKP notation. So given any formula in BGR involving the $T$-basis can be reinterpreted in the BKP conventions by applying $\iota$. Remember: $\iota$ is an anti-isomorphism, so one needs to reverse the order.

\begin{Warning}
  We have $\iota(T_w) = T_{w^{-1}}$, so we need to be careful when writing $T_w$ to emphasize which conventions we are using. 
\end{Warning}

\end{appendix}

\bibliographystyle{amsalpha}
\bibliography{bibliographie.bib}

\providecommand{\bysame}{\leavevmode\hbox to3em{\hrulefill}\thinspace}
\providecommand{\MR}{\relax\ifhmode\unskip\space\fi MR }
\providecommand{\MRhref}[2]{%
  \href{http://www.ams.org/mathscinet-getitem?mr=#1}{#2}
}
\providecommand{\href}[2]{#2}
\begin{thebibliography}{BGKP14}

\bibitem[AH19]{abdellatif2019completed}
Ramla Abdellatif and Auguste H\'ebert, \emph{Completed {I}wahori-{H}ecke
  algebras and parahoric {H}ecke algebras for {K}ac-{M}oody~groups over local
  fields}, Journal de l'\'Ecole polytechnique --- Math\'ematiques \textbf{6}
  (2019), 79--118 (en).

\bibitem[AH22]{abdellatif2022erratum}
\bysame, \emph{Erratum to {{\textquotedblleft}Completed} {Iwahori-Hecke}
  algebra and parahoric {Hecke} algebras for {Kac-Moody~groups} over local
  fields{\textquotedblright}}, Journal de l{\textquoteright}\'Ecole
  polytechnique {\textemdash} Math\'ematiques \textbf{9} (2022), 1293--1304
  (en).

\bibitem[BB05]{bjorner2005combinatorics}
Anders Bj\"{o}rner and Francesco Brenti, \emph{Combinatorics of {C}oxeter
  groups}, Graduate Texts in Mathematics, vol. 231, Springer, New York, 2005.
  \MR{2133266}

\bibitem[BGKP14]{braverman2014affine}
Alexander Braverman, Howard Garland, David Kazhdan, and Manish Patnaik,
  \emph{An affine {G}indikin-{K}arpelevich formula}, Perspectives in
  Representation Theory (Yale University, May 12--17, 2012)(P. Etingof, M.
  Khovanov, and A. Savage, eds.), Contemp. Math \textbf{610} (2014), 43--64.

\bibitem[BGR16]{bardy2016iwahori}
Nicole Bardy{-Panse}, St\'ephane Gaussent, and Guy Rousseau,
  \emph{Iwahori-{H}ecke algebras for {K}ac-{M}oody groups over local fields},
  Pacific J. Math. \textbf{285} (2016), no.~1, 1--61. \MR{3554242}

\bibitem[BGR19]{bardy2019macdonald}
Nicole Bardy{-Panse}, St\'{e}phane Gaussent, and Guy Rousseau,
  \emph{Macdonald's formula for {K}ac-{M}oody groups over local fields}, Proc.
  Lond. Math. Soc. (3) \textbf{119} (2019), no.~1, 135--175. \MR{3957833}

\bibitem[BK11]{braverman2011spherical}
Alexander Braverman and David Kazhdan, \emph{The spherical {Hecke} algebra for
  affine {Kac}-{Moody} groups. {I}}, Ann. Math. (2) \textbf{174} (2011), no.~3,
  1603--1642 (English).

\bibitem[BKP16]{braverman2016iwahori}
Alexander Braverman, David Kazhdan, and Manish~M. Patnaik,
  \emph{Iwahori-{H}ecke algebras for {$p$}-adic loop groups}, Invent. Math.
  \textbf{204} (2016), no.~2, 347--442. \MR{3489701}

\bibitem[BPHR25]{bardy2025twin}
Nicole Bardy-Panse, Auguste H{\'e}bert, and Guy Rousseau, \emph{Twin masures
  associated with {Kac}-{Moody} groups over {Laurent} polynomials}, Ann.
  Represent. Theory \textbf{2} (2025), no.~3, 281--353 (English).

\bibitem[BR21]{bardy2021structure}
Nicole Bardy{-Panse} and Guy Rousseau, \emph{On structure constants of
  {I}wahori-{H}ecke algebras for {K}ac-{M}oody groups}, Algebr. Comb.
  \textbf{4} (2021), no.~3, 465--490 (English).

\bibitem[GG95]{garland1995affineheckealgebrasassociated}
H.~Garland and I.~Grojnowski, \emph{Affine {H}ecke algebras associated to
  {K}ac-{M}oody groups}, 1995.

\bibitem[GR08]{gaussent2008kac}
St{\'e}phane Gaussent and Guy Rousseau, \emph{Kac-{M}oody groups, hovels and
  {L}ittelmann paths}, Annales de l'institut Fourier, vol.~58, 2008,
  pp.~2605--2657.

\bibitem[GR14]{gaussent2014spherical}
\bysame, \emph{Spherical {H}ecke algebras for {K}ac-{M}oody groups over local
  fields}, Annals of Mathematics \textbf{180} (2014), no.~3, 1051--1087.

\bibitem[H{\'e}b22]{hebert2022principal}
Auguste H{\'e}bert, \emph{Principal series representations of
  {Iwahori{\textendash}Hecke} algebras for {Kac{\textendash}Moody} groups over
  local fields}, Annales de l'Institut Fourier \textbf{72} (2022), no.~1,
  187--259 (en).

\bibitem[Kac94]{kac1994infinite}
Victor~G Kac, \emph{Infinite-dimensional {L}ie algebras}, vol.~44, Cambridge
  university press, 1994.

\bibitem[Kum02]{kumar2002kac}
Shrawan Kumar, \emph{Kac-{M}oody groups, their flag varieties and
  representation theory}, Progress in Mathematics, vol. 204, Birkh\"auser
  Boston, Inc., Boston, MA, 2002. \MR{1923198}

\bibitem[Loo80]{looijenga1980invariant}
Eduard Looijenga, \emph{Invariant theory for generalized root systems},
  Inventiones mathematicae \textbf{61} (1980), no.~1, 1--32.

\bibitem[Lus83]{lusztig1983singularities}
George Lusztig, \emph{Singularities, character formulas, and a q-analog of
  weight multiplicities}, Ast{\'e}risque \textbf{101} (1983), no.~102,
  208--229.

\bibitem[Mar18]{marquis2018introduction}
Timoth{\'e}e Marquis, \emph{An introduction to {Kac}-{Moody} groups over
  fields}, EMS Textb. Math., Z{\"u}rich: European Mathematical Society (EMS),
  2018 (English).

\bibitem[MP24]{muthiah2024pursuing}
Dinakar Muthiah and Anna Pusk{\'a}s, \emph{Pursuing coxeter theory for
  kac-moody affine hecke algebras}, arXiv preprint arXiv:2406.14447 (2024).

\bibitem[Mut18]{muthiah2018iwahori}
Dinakar Muthiah, \emph{On {I}wahori-{H}ecke algebras for {$p$}-adic loop
  groups: double coset basis and {B}ruhat order}, Amer. J. Math. \textbf{140}
  (2018), no.~1, 221--244. \MR{3749194}

\bibitem[R{\'e}m02]{remy2002groupes}
Bertrand R{\'e}my, \emph{Groupes de {K}ac-{M}oody d\'eploy\'es et presque
  d\'eploy\'es}, Ast\'erisque (2002), no.~277, viii+348. \MR{1909671}

\bibitem[Rou16]{rousseau2016groupes}
Guy Rousseau, \emph{Groupes de {K}ac-{M}oody d\'eploy\'es sur un corps local
  {II}. {M}asures ordonn\'ees}, Bull. Soc. Math. France \textbf{144} (2016),
  no.~4, 613--692. \MR{3562609}

\bibitem[Tit87]{tits1987uniqueness}
Jacques Tits, \emph{Uniqueness and presentation of {K}ac-{M}oody groups over
  fields}, J. Algebra \textbf{105} (1987), no.~2, 542--573. \MR{873684}

\end{thebibliography}

\end{document}